\documentclass{article}
\usepackage{graphicx,fullpage}

\usepackage{caption}
\usepackage{amsmath}
\usepackage{xspace}
\usepackage{hyperref}       
\usepackage{url}            
\usepackage{booktabs}       
\usepackage{amsfonts}       
\usepackage{nicefrac}       
\usepackage{microtype}      
\usepackage{xcolor}         
\usepackage{graphicx}
\usepackage{cancel}
\usepackage{ulem}
\usepackage{cleveref}
\usepackage{amssymb}
\usepackage{setspace}
\newcommand*{\vertbar}{\rule[-1ex]{0.5pt}{2.5ex}}

\usepackage{cite}
\definecolor{darkgreen}{rgb}{0.0, 0.2, 0.13}

\usepackage{cite}
\usepackage{times, natbib}
\usepackage{soul}
\usepackage{url}
\usepackage{booktabs}
\usepackage{algorithm}
\usepackage{algorithmic}
\usepackage[switch]{lineno}
\usepackage{enumitem}

\usepackage{amsfonts}
\usepackage{amsmath}
\usepackage{graphicx}
\usepackage{amssymb}
\usepackage{tabularx}
\usepackage{booktabs}
\usepackage{tabularx}
\usepackage{cancel}
\usepackage{booktabs}
\usepackage{multirow}
\usepackage{siunitx}
\usepackage{algorithm}

\usepackage{amsmath,amssymb,amsfonts}
\usepackage{algorithmic}
\usepackage{textcomp}
\usepackage{xcolor}

\newcommand{\R}{\mathbb R}

\newcommand{\minimize}[1]{\underset{#1}{\mathrm{minimize}}}

\newcommand{\argmax}[1]{\underset{#1}{\mathrm{argmax}}}

\newcommand{\tr}{\mathbf{tr}}

\newcommand{\mb}{\mathbf}

\usepackage{amsthm}

\newtheorem{theorem}{Theorem}[section]
\newtheorem{lemma}[theorem]{Lemma}

\theoremstyle{definition}

\theoremstyle{remark}

\begin{document}


\title{Advancing Multi-Secant Quasi-Newton Methods for General Convex Functions}


\author{Mokhwa Lee and Yifan Sun}

%

\maketitle

\begin{abstract}
Quasi-Newton (QN) methods provide an efficient alternative to second-order methods for minimizing smooth unconstrained problems.
While QN methods generally compose a Hessian estimate based on one secant interpolation per iteration, multisecant methods use multiple secant interpolations and can improve the quality of the Hessian estimate at small additional overhead cost. 
However, implementing multisecant QN methods has several key challenges involving method stability, the most critical of which is that when the objective function is convex but not quadratic, the Hessian approximate is not, in general, symmetric positive semidefinite (PSD), and the steps are not guaranteed to be descent directions. 

We therefore investigate a symmetrized and PSD-perturbed Hessian approximation method for multisecant QN. We offer an efficiently computable method for producing the PSD perturbation, show superlinear convergence of the new method, and demonstrate improved numerical experiments over general convex minimization problems. 
We also investigate the limited memory extension of the method, focusing on BFGS, on both convex and non-convex functions. 
Our results suggest that in ill-conditioned optimization landscapes, leveraging multiple secants can accelerate convergence and yield higher-quality solutions compared to traditional single-secant methods.

\end{abstract}



\section{Introduction}
 

We consider the unconstrained minimization problem 
\begin{equation}
    \minimize{x\in \R^n} \quad f(x)
    \label{eq:main}
\end{equation}
where $f:\R^n\to\R$ is a convex function in $\mathcal C^2$, and   bounded below. Newton's method iteratively solves the linear system of order $n$ to get a search direction $d_t$,
\begin{equation}
   \nabla^2 f(x_t) d_t = -\nabla f(x_t)
\label{eq:newton-linsys}
\end{equation}
 where $\nabla^2 f(x_t)$ is the Hessian and $\nabla f(x_t)$ is the gradient of the $t$th iterate. In this case, the next iterate is updated as
 \[
 x_{t+1} = x_t+\alpha d_t
 \]
 where $d_t = -[\nabla^2 f(x_t)]^{-1}\nabla f(x_t)$ and $\alpha > 0$ is a step length parameter. However, while this method is foundational in continuous optimization, obtaining the Hessian matrix and solving \eqref{eq:newton-linsys} becomes computationally impractical for large-scale problems. For this reason, quasi-Newton (QN) methods, such as BFGS \citep{broyden1970convergence, fletcher1970new, goldfarb1970family, shanno1970conditioning}, have been introduced as effective alternatives. These methods efficiently approximate the Hessian using simple operations performed on successive gradient vectors.

In particular, QN methods are  designed to construct the matrix $B_t$ at each iteration which satisfies the \textit{secant condition}
\begin{equation}
B_{t+1}(x_{t+1} -x_t)=  \nabla f(x_{t+1}) -\nabla f(x_t)   
\label{eq:secant}
\end{equation}
where $B_{t+1} \in \mathbb{R}^{n \times n}$ is a Hessian approximation of $f$ at $x_{t+1}$.  The subsequent iterates are then updated 
\begin{equation}
    x_{t+1} = x_t - \alpha B_t^{-1} \nabla f(x_{k}).
    \label{eq:qnstep}
\end{equation}
For $n > 1$, the secant condition \eqref{eq:secant} represents $n$ equations involving $n(n+1)/2$ variables, and is always underdetermined. Thus, a stronger, lesser-explored family of approximations are the \textit{multisecant conditions}, which satisfy
\begin{align}
        B_t(x_{i}-x_{j}) = \nabla f(x_{i})-\nabla f(x_{j}) 
    \label{eq:multisecant}
\end{align}
for some subset of $i\neq j\in \{t,t-1,...,t-q+1\}$ where $q>1$ is the number of past iterates taken into account; this strategy promotes a more accurate Hessian approximation. Conventionally, $q$ is a small positive integer such that $q \ll n$.

While multisecant extensions have been explored in the past literature \citep{gay1978solving}, and are shown to be more powerful approximations than single-secant approaches, they often struggle with stability. Specifically, in the case of DFP \citep{davidon1991variable} and BFGS, a single-secant update is guaranteed to be a \textit{descent search direction}; however, incorporating multisecant conditions destroys this valuable descent property.
 For this reason, multisecant QN methods seem popular only in quadratic optimization, and are not easily generalizable even for convex functions.

\subsection{Related works}

Perhaps the most well-known family of single-secant quasi-Newton methods are Broyden's method \cite{broyden1965class,gay1979some} which gives a rank-1 and non-symmetric update, Powell's method (PSB) which introduces symmetric updates \citep{powell1964efficient}, Davidson-Fletcher-Powell (DFP) \citep{davidon1991variable} , and BFGS named after the concurrent works of \citet{broyden1970convergence}, \citet{powell1964efficient}, \citet{goldfarb2020practical}, and \citet{shanno1970conditioning}. The latter two methods introduce symmetric positive semidefiniteness (PSD) and can be seen as nearest matrices in a modified norm.
These qualities (symmetric and PSD) are often desired to ensure that $d_t = -B_t^{-1}\nabla f(x_t)$ is indeed a descent direction, and as a result the methods are more stable in practice.
There are also many recent works concerning improvements of the single-secant QN methods that involve subsampling \citep{berahas2022quasi}, sketching \citep{pilanci2016iterative} or other forms of stochasticity \citep{goldfarb2020practical}, as well as greedy updates \citep{rodomanov2021greedy}, incremental updates \citep{mokhtari2018iqn}, mixing strategies \citep{eyert1996comparative}, trust regions for improved numerical stability \citep{brust2024trust}, and dense initializations \citep{brust2019dense}  (to name only a few).
The work which is similar in spirit to ours is
\citet{goldfeld1966maximization} which uses a diagonal perturbation to improve conditioning but does not contain  convergence analysis; and 
\citet{abd2023positive} which explores perturbations for Powell's method to ensure PSD estimates.
Finally, we highlight \citet{dennis1974characterization,dennis1977quasi}
for first showing superlinear convergence of Broyden's, and then BFGS method, and 
\citet{nocedal1999numerical,lui2021superlinear}  whose expositions help fill in some of the blanks in the convergence proof.

The multisecant extensions were first explored not long later; \citet{gay1978solving} offer a version of Broyden's method that satisfies the secant condition with multiple prior updates; an argument for superliner convergence is given. This extension was generalized \citep{schnabel1983quasi} for extensions of Broyden's, Powell's method, DFP, and BFGS updates, and offers a perturbation in the Cholesky factorization to maintain PSD and symmetry. 
More recent explorations of multisecant QN methods include 
\citet{fang2009two} which explores the integration of Andersen mixing and
\citet{burdakov2018multipoint} which integrates a sophisticated line search method.
The works most related ours include 
\citet{gao2018block}, which maintain positive semidefinite estimates using eigendecompositions, and \citet{scieur2021generalization} which perform complete positive semidefinite projections at each step. 
Where they explore eigendecompositions, our perturbation is in adding a diagonal, with a carefully tuned magnitude.

The development of limited memory multisecant methods is an important extension, in the regime where even storing a dense $n\times n$ Hessian estimate is prohibitive. Limited memory QN methods have been previously studied \citep{van2005limited,reed2009broyden,kolda1997limited,erway2015efficiently}, especially for Broyden's, DFP, and BFGS \citep{zhu1997algorithm,byrd1994representations,liu1989limited}. 
The L-BFGS method especially was recently popularized for large-scale machine learning systems, such as Google's Sandblaster method \citep{dean2012large}; other methods, such as SR1 \citep{byrd1994representations}, have also been studied in this context. 
More recently, several papers showed superlinear convergence limited memory QN methods, under certain modifications; 
\citep{asl2021analysis} via sophisticated line search strategies;  \citep{gao2023limited} with specific greedy updates; and \citep{berahas2022limited} via a displacement aggregation strategy to mimic a full-memory system. 



\subsection{Contributions and outline}
In this paper, 
we investigate techniques for imposing symmetric and PSD updates in multisecant QN methods through perturbation strategies, especially for ill-conditioned non-quadratic problems. 
Our contributions include
\begin{enumerate}
\item a method of carefully tuned diagonal updates for stable method perturbations, improved through secant rejection methods and scaling techniques;
    \item a superlinear convergence rate of the proposed strategy;
    \item a limited memory extension and usage on nonconvex neural network training.
\end{enumerate}
Section \ref{sec:prelim} reviews the single-secant and multi-secant QN methods, as well as the inverse update using the Woodbury property. 
Section \ref{sec:perturbation} introduces our perturbation method, along with a complexity analysis.
Section \ref{sec:superlinear} gives the superlinear convergence result, and section \ref{sec:numerical} gives extensive numerical comparisons.
Section \ref{sec:extension} discusses the important extensions for limited memory and use in non-convex optimization.

\section{Quasi-Newton methods}
\label{sec:prelim}
\subsection{Single secant methods}
The well-known Newton's method for solving \eqref{eq:main} follows the iterative scheme
\begin{equation}
x_{t+1} = x_t-[\nabla^2 f(x_t)]^{-1}\nabla f(x_t)
\label{eq:newtonstep}
\end{equation}
where $\nabla^2 f(x_t)$ and $\nabla f(x_t)$ are the Hessian and the gradient of $f$ at $x_t$, respectively. 
Newton's method is derived from the truncated second-order Taylor series expanded at the iterate $x_t$, as
\begin{equation}
\nabla f(x_{t+1})\approx \nabla f(x_t)+\nabla^2f(x_t)(x_{t+1}-x_t),
\label{eq:approx-secant-cond}
\end{equation}
under the assumption that $\nabla f(x_{t+1}) = \nabla f(x^*) = 0$ for some $k$, resulting in the Newton step \eqref{eq:newtonstep}. However, computing $\nabla^2 f(x)$ and solving the linear system \eqref{eq:newtonstep} are costly and may suffer from numerical issues.
It is approximated by an $n\times n$ matrix that satisfies \eqref{eq:approx-secant-cond} at each step
\begin{equation}
   B_{t+1}\underbrace{(x_{t+1}-x_t)}_{=:s_t\in \R^n} =  \underbrace{\nabla f(x_{t+1}) - \nabla f(x_t)}_{=:y_t\in \mathbb{R}^n}
   \label{eq:secant_equation}
\end{equation}
where 
$B_{t+1}\in \mathbb{R}^{n\times n}$ estimates the Hessian at iteration ${t+1}$.
Note that this linear system includes $n$ constraints, while a symmetric matrix $B_{t+1}$ contains $\frac{n(n+1)}{2}$ free variables; that is to say, \eqref{eq:secant_equation} is \textit{underdetermined}. 
Thus, QN methods satisfying \eqref{eq:secant_equation} are far from unique, and there is the potential to continually develop improvements.
Four well-known single-secant QN methods are described below.
\begin{itemize}
    \item[$\bullet$] \textit{Broyden's method \citep{broyden1965class}}
    forms rank-1 and non-symmetric updates satisfying secant equations described in \eqref{eq:secant_equation}:
\[
B_{t+1} = B_t + \frac{(y_t-B_ts_t)s_t^\top}{s_t^\top s_t};\tag{Broyden}
\]

\item[$\bullet$] \textit{Powell symmetric Broyden's (PSB) \citep{powell1964efficient}}
symmetrizes the Hessian estimate 
\begin{align*}
B_{t+1} = B_t + \frac{(y_t-B_ts_t)s_t^\top + s_t(y_t-B_ts_t)^\top}{s_t^\top s_t}
+ \frac{1}{2}\frac{(y_t-B_ts_t)^\top s_t}{(s_t^\top s_t)^2}s_ts_t^\top;
\tag{Powell}
\end{align*}

\item[$\bullet$] \textit{ DFP \citep{davidon1991variable}}
provides symmetry and PSD Hessian approximation
\begin{align*}
B_{t+1} &= B_t + \frac{(y_t-Bs_t)y_t^\top + y_t(y_t-B_ts_t)^\top }{y_t^\top s_t} -\frac{y_t(y_t-B_ts_t)^\top s_t y_t^\top }{(y_t^\top s_t)^2} ;\tag{DFP}
\end{align*}

\item[$\bullet$] and  \textit{BFGS  \citep{broyden1970convergence,powell1964efficient,goldfarb1970family,shanno1970conditioning}}
 is the most popular QN algorithm with rank-2 and symmetric updates, and maintains PSD estimates
\begin{equation*}
    B_{t+1} = B_t + \frac{y_ty_t^\top }{y_t^\top s_t}-\frac{B_ts_ts_t^\top B_t}{s_t^\top B_ts_t}. \tag{BFGS}
\end{equation*}
\end{itemize}
Note that each QN method will update the next iterate  as 
\[
    x_{t+1} = x_t-\alpha B_t^{-1}\nabla f(x_t).
    \]
    So, if $B_t$ is PSD, then for $\alpha>0$ small enough, this step is guaranteed to descend ($f(x_{t+1})<f(x_t)$), since   
\begin{equation}
     -\nabla f(x_t)^\top  B_t^{-1} \nabla f(x_t)<0.
    \label{eq:descent_cond}
\end{equation}
However, if $B_{t}$ is not PSD, the inequality \eqref{eq:descent_cond} is not necessarily satisfied and the algorithm will necessarily monotonically decrease at each iteration; the resulting behavior is usually instability and divergence. Therefore, maintaining $B_{t}$ PSD is an important key for QN methods.

\subsection{Multisecant methods}
We now consider incorporating more secant conditions than just on the last two iterates. There are two natural constructions to consider: the ``curve-hugging" version for $i = t,...,t-q+1$, such that
\begin{equation}
s_i = x_{i+1} - x_{i}, \quad y_i = \nabla f(x_{i+1}) - \nabla f(x_{i}),
\label{update-sy-1}
\end{equation}
and the ``anchored at most recent" version for $i = t,...,t-q$, such that
\begin{equation}
s_i = x_{t+1} - x_i, \quad y_i = \nabla f(x_{t+1}) - \nabla f(x_i).
\label{update-sy-2}
\end{equation}
In practice, we find that the two versions seem to have similar performance.   
We represent these choices with 
matrices $S_t\in \R^{n\times q}$ and $Y_t\in \R^{n\times q}$ as 
\begin{equation}
    S_t = \left[
  \begin{array}{cccc}
    \vertbar & \vertbar &        & \vertbar \\
    s_{t-q}    & s_{t-q+1}    & \ldots & s_{t}    \\
    \vertbar & \vertbar &        & \vertbar 
  \end{array}
\right], \qquad 
Y_t = 
\left[
  \begin{array}{cccc}
    \vertbar & \vertbar &        & \vertbar \\
    y_{t-q}    & y_{t-q+1}    & \ldots & y_{t}    \\
    \vertbar & \vertbar &        & \vertbar 
  \end{array}
\right].
\label{update-bigYS}
\end{equation}
Then, the multisecant condition is $B_{t+1}S_t=Y_t$,
which interpolates $q$ previous iterates.
Schabel \citep{schnabel1983quasi} presented the following four multisecant generalizations of QN methods: 
\begin{align}
B_{t+1} &= B_t + (Y_t-B_tS_t)(S_t^\top S_t)^{-1}S_t^\top 
\tag{MS Broyden} \label{eq:MS_Broyden}\\
B_{t+1} &= B_t + (Y_t-B_tS_t)(S_t^\top S_t)^{-1}S_t^\top  + S_t(S_t^\top S_t)^{-1}(Y_t-B_tS_t)^\top \nonumber\\
&\qquad\quad - S_t (S_t^\top S_t)^{-1}(Y_t-B_tS_t)^\top S_t(S_t^\top S_t)^{-1}S_t^\top 
\tag{MS PSB} \label{eq:MS_PSB}\\ 
B_{t+1} &= B_t + (Y_t-B_tS_t)(Y_t^\top S_t)^{-1}Y_t^\top  
+ Y_t(Y_t^\top S_t)^{-1}(Y_t-B_tS_t)^\top  \nonumber\\ 
&\qquad\quad - Y_t(Y_t^\top S_t)^{-1}(Y_t-B_tS_t)^\top S_t(Y_t^\top S_t)^{-1} Y_t^\top 
\tag{MS DFP} \label{eq:MS_DFP}\\
B_{t+1} &= B_t + Y_t(Y_t^\top S_t)^{-1}Y_t^\top 
-B_tS_t(S_t^\top B_tS_t)^{-1}S_t^\top B_t
\tag{MS BFGS} \label{eq:MS_BFGS}
\end{align}
Unlike the single-secant case, symmetry and PSD are only guaranteed to hold in a restricted problem setting. Specifically, Powell's $B_{t+1}$ is guaranteed to be symmetric only if $S_t^\top Y_t$ is symmetric, and DFP's and BFGS's $B_{t+1}$ is symmetric and PSD only if $Y_t^\top S_t$ is symmetric and PSD. However, this is not true in general;  note that the multisecant constraint ($B_{t+1}S_t=Y_t$) enforces $S_t^\top B_{t+1}S_t = S_t^\top Y_t$, so the symmetry or PSD-ness of $B_{t+1}$ is \textit{not possible if $S_t^\top Y_t$ does not have the same corresponding properties,} of which are generally not true for non-quadratic convex functions $f$.

\subsection{Woodbury Inversion of Multisecant BFGS}

The four update rules \eqref{eq:MS_Broyden}, \eqref{eq:MS_DFP}, \eqref{eq:MS_PSB},\eqref{eq:MS_BFGS}
 can be succinctly written as 
 \begin{equation}
 B_{t+1} = B_t + C_{1,t}A^{-1}_tC_{2,t}^\top 
 \label{eq:direct-update}
 \end{equation}
 where $C_{1,t}$, $C_{2,t}$ and $A_t$ depend on $B_t$, $S_t$, and $Y_t$. Specifically, the update is \textit{low-rank}; $A_t$ is $q\times q$ for Broyden's method, and $2q\times 2q$ for the others. 
 To avoid computing inverses, low-rank updates of $B_t$ are updated using the Sherman-Morrison-Woodbury inversion lemma \citep{woodbury1950inverting}. 
This crucial step is a key differentiating feature between Newton's method and QN methods, as it avoids solving an expensive linear system at each step.

We now give the inverse update step for the MS-QN methods.
The inverse update can be directly computed for Broyden's method
\[
B_{t+1}^{-1} =  B_t^{-1}-(B_t^{-1}Y_t-S_t)(S_t^\top  B_t^{-1} Y_t)^{-1}S_t^\top B_t^{-1}
\] 
and BFGS
\begin{align*}
B_{t+1}^{-1}  & = B_t^{-1}-\begin{bmatrix} B_t^{-1}Y_t,& S_t\end{bmatrix}
\begin{bmatrix}
Y_t^\top S_t+Y_t^\top B_t^{-1}Y_t & Y_t^\top S_t\\
S_t^\top Y_t & 0
\end{bmatrix}^{-1}
\begin{bmatrix}
Y_t^\top B_t^{-1}\\ S_t^\top 
\end{bmatrix}.
\end{align*}  
For PSB, the updates are first symmetrized, and 
\begin{eqnarray*}
\frac{ B_{t+1}^{-1}  + B_{t+1}^{-\top} }{2}
=\frac{B_t^{-1}+B_t^{-\top}}{2} + D_{1,t} W_t^{-1} D_{2,t}
\end{eqnarray*}
where for PSB, 
\begin{eqnarray*}
D_{1,k} &=& \begin{bmatrix} B_t^{-1} Y_t - S_t, & B_t^{-1} S_t,& B_t^{-1} S_t,&B_t^{-1} S_t\end{bmatrix}, \\
D_{2,k} &=& \begin{bmatrix} B_t^{-1} S_t,& B_t^{-1} Y_t - S_t, &  B_t^{-1} S_t,&B_t^{-1} S_t\end{bmatrix}, \\
W_t &=& -\begin{bmatrix}
     S_t^\top B_t^{-1}Y_t, & V_t,& V_t, & V_t \\
  X_t,  &   Y_t^\top  B_t^{-1}S_t&  U_t^\top  & U_t^\top  \\
   U_t & V_t&  Z_t +  V_t& V_t\\
   U_t& V_t& V_t&  G_t +  V_t\\  
\end{bmatrix}
\end{eqnarray*}
 where
 \begin{eqnarray*}
U_t &=& S_t^\top B_t^{-1}Y_t - S_t^\top S_t\\
V_t &=& S_t^\top B_t^{-1}S_t\\  
X_t &=& Y_t^\top B_t^{-1}Y_t - S_t^\top Y_t - Y_t^\top S_t + S_t^\top B_tS_t\\
Z_t &=& S_t^\top S_t(S_t^\top B_tS_t)^{-1} S_t^\top S_t\\
G_t &=& -\frac{1}{2} S_t^\top S_t (Y_t^\top S_t  + S_t^\top Y_t)^{-1} S_t^\top S_t.
\end{eqnarray*}
 For DFP, 
 \[
D_{1,t} = \begin{bmatrix} Q_t-S_t,&Q_t,&Q_t  \end{bmatrix} , \qquad 
  D_{2,t} =\begin{bmatrix} Q_t &Q_t-S_t&Q_t\end{bmatrix}\\
  \]
  \[
W_t =-\begin{bmatrix}
       T_t ,&T_t,&T_t\\ 
      T_t - R_t^\top  - R_t + H_t, & R_t +T_t  - R_t^\top ,&T_t - R_t^\top \\
      T_t - R_t,&T_t,& -R_t (H_t-R_t)^{-1}R_t + T_t
  \end{bmatrix}
\]
where 
\[
   R_t = Y_t^\top S_t, \qquad 
   H_t =  S_t^\top B_tS_t,\qquad 
   Q_t =B_t^{-1}Y_t,\qquad 
   T_t = Y_t^\top B_t^{-1}Y_t.
\]
In both cases, to avoid computing $B_t$, we use the relation under the assumption that $S_t$ has full column rank \footnote{This is usually true in the vanilla implementation, and always true when we use the rejection extension.}
\begin{equation}
S_t^\top B_tS_t = (S_t^\top S_t)^{-1} (S_t^\top B_t^{-1}S_t)^{-1}(S_t^\top S_t)^{-1}.
\label{eq:trick:SBS}
\end{equation}
The symmetrization is necessary to avoid the term $S_t^\top B_t^{T}B_t^{-1}Y_t$, which cannot be easily simplified nor cheaply computed if both $B_t$ and $B_t^{-1}$ are not both involved, defeating the purpose of the Woodbury inversion.


Note that by using this inverse update, we reduce computational requirements  from \( O(n^3) \) to \( O(qn^2 + q^3) \). 
 In later sections, we differentiate between using a \textit{direct update} \eqref{eq:direct-update} and an inverse update, via the Woodbury formula.

\section{Multisecant methods with positive semidefinite perturbation}
\label{sec:perturbation}

\subsection{Diagonal perturbation}

The simplest version of our perturbation method is
\begin{equation}
H_{t+1} = H_t + \frac{D_{1,t} W_t^{-1}D_{2,t}^\top +(D_{1,t} W_t^{-1}D_{2,t}^\top )^\top }{2} + \mu_t I
\label{eq:symmetrizedqn}
\end{equation}
where $H_t$ can be $B_t$ (direct solve) or $B_t^{-1}$ (Woodbury inverse), and $D_{1,t}$, $D_{2,t}$, and $W_t$ are chosen such that the vanilla (symmetrized) updates are achieved when $\mu_t = 0$. 
We now introduce a 
 \textit{computationally cheap} ($O(q^3+q^2n)$) method (Alg. \ref{alg:mu}) of producing a $\mu_t$ that ensures $H_{t+1}$ is symmetric PSD, as long as $H_t$ is symmetric PSD.

\begin{theorem}
\label{th:lowrankupdate}
Consider $W$ a non-symmetric matrix, $c>0$ and 
\begin{equation}
 \Delta =  \frac{1}{2} 
    \begin{bmatrix} D_{1} & D_{2} \end{bmatrix}
    \begin{bmatrix} 0 & W^{-1} \\ W^{-T} & 0 \end{bmatrix}
    \begin{bmatrix}D_{1}^\top  \\ D_{2}^\top  \end{bmatrix} \in\mathbb{R}^{n\times n}.
    \label{eq:deltadef}
\end{equation}
Then $\Delta + \mu I$ is PSD if and only if 
\begin{equation}
H_2 = \begin{bmatrix}
  cI &  F\\
 F^\top & cI
\end{bmatrix}- 
(2\mu)^{-1} G + (2\mu)^{-1}     
G
(2\mu C^{-1}+ G )^{-1}
G \in \R^{2\tilde q\times 2\tilde q}
\label{eq:nonsym_testmat}
\end{equation}
is PSD, for 
    \[C = \begin{bmatrix} (cI - c^{-1}FF^\top )^{-1} & W^{-1}-c^{-1}F  (cI - c^{-1}F^\top F)^{-1}  \\ W^{-T}-c^{-1} (cI - c^{-1}F^\top F)^{-1}  F^\top  & (cI - c^{-1}F^\top F)^{-1} \end{bmatrix},
    \]
$G = \begin{bmatrix} D_1^\top  \\  D_2^\top \end{bmatrix} \begin{bmatrix} D_1 &  D_2\end{bmatrix}$ and  $F =  VSU^\top $. Here,  $W = U\Sigma V^\top $ is the SVD of $W$, and $S_{i,i} = \frac{-1+\sqrt{1+4c^2\Sigma_{ii}^{-2}}}{2\Sigma_{i,i}^{-1}}$.
Here, $\tilde q = q$ for Broyden's method, and $\tilde q = 2q$ for Powell, DFP, and BFGS.
\end{theorem}

\begin{proof}
Consider first the matrix
\[
H = \begin{bmatrix}2 \mu I + A & D_1 & D_2 \\ D_1^\top  & cI & F\\D_2^\top  & F^\top  & cI\end{bmatrix}
\]
Then one Schur complement of $H$ is 
\[
H_1 := 2\mu I + A - \begin{bmatrix} D_1 &  D_2\end{bmatrix}\begin{bmatrix} cI & F \\ F^\top  & cI \end{bmatrix}^{-1} \begin{bmatrix} D_1^\top  \\  D_2^\top \end{bmatrix} = 2\mu I + 2\Delta
\]
and another is $H_2$ (to be shown later).
So, $H$ is PSD if  either $A+2\mu I$ is PSD and $H_2$ is PSD, or   $\begin{bmatrix} cI & F \\ F^\top  & cI \end{bmatrix}$ is PSD and $H_1$ is PSD, or both.
Here, 
\begin{eqnarray*}
A &=&  \begin{bmatrix} D_1 &  D_2\end{bmatrix}\begin{bmatrix} cI & F \\ F^\top  & cI \end{bmatrix}^{-1} \begin{bmatrix} D_1^\top  \\  D_2^\top \end{bmatrix} +  \begin{bmatrix} D_1 &  D_2 \end{bmatrix}
    \begin{bmatrix} 0 & W^{-1} \\ W^{-T} & 0 \end{bmatrix}
    \begin{bmatrix}D_1^\top  \\ D_2^\top  \end{bmatrix} \\
    &=&  \begin{bmatrix} D_1 &  D_2\end{bmatrix}
    \underbrace{\begin{bmatrix} (cI - c^{-1}FF^\top )^{-1} & E^T  \\ E & (cI - c^{-1}F^\top F)^{-1} \end{bmatrix}}_{=:C}
    \begin{bmatrix} D_1^\top  \\  D_2^\top \end{bmatrix}  
\end{eqnarray*}
where $E = W^{-\top}-c^{-1}  F^\top (cI - c^{-1}F F^\top)^{-1}  $.

Next, we construct $F$ to have the same left and right singular vectors as $W$, so $W = U\Sigma V^\top $ and $F = VSU^\top $. Then   the eigenvalues of $A$ are the same as that of 
 \begin{eqnarray*}
 \begin{bmatrix} V  & 0 \\ 0 & U \end{bmatrix}^\top 
 C
 \begin{bmatrix} V & 0 \\ 0 & U\end{bmatrix}=
\begin{bmatrix} (cI - c^{-1}S^2)^{-1} & \Sigma^{-1}-c^{-1}  (cI - c^{-1}S^2)^{-1} S \\ \Sigma^{-1}-c^{-1} S(cI - c^{-1}S^2)^{-1} & (cI - c^{-1}S^2)^{-1} \end{bmatrix}\\
 \end{eqnarray*}
 which can be rearranged into a block diagonal matrix whose $2\times 2$ blocks are 
 \[
 B_i = \begin{bmatrix} (c - c^{-1}S_{ii}^2)^{-1} & \Sigma_{ii}^{-1}-c^{-1}S_{ii}  (c - c^{-1}S_{ii}^2)^{-1}  \\ \Sigma_{ii}^{-1}-c^{-1} (c - c^{-1}S_{ii}^2)^{-1}  S_{ii} & (c - c^{-1}S_{ii}^2)^{-1} \end{bmatrix}
 \]
These blocks are PSD if $S_{ii}<c$ and
\[
0<
(c - c^{-1}S_{ii}^2)^{-1}   - \frac{(\Sigma_{ii}^{-1}-c^{-1}S_{ii}  (c - c^{-1}S_{ii}^2)^{-1}  )^2}{(c - c^{-1}S_{ii}^2)^{-1}}  
\]
\[\iff\]
\[S_{ii} > (c^2 - S_{ii}^2) (\Sigma_{ii}^{-1}  - c^{-1}(c^2 - S_{ii}^2)) \text{ or }
 S_{ii}  <   (\Sigma_{ii}^{-1} + c^{-1}(c^2-S_{ii}^2))(c^2 - S_{ii}^2)
\]
which is satisfied if 
\[
S_{ii} =  (c^2 - S_{ii}^2) \Sigma_{ii}^{-1}   \quad \iff \quad S_{i,i} = \frac{-1+\sqrt{1+4c^2\Sigma_{ii}^{-2}}}{2\Sigma_{i,i}^{-1}}\leq \frac{\sqrt{4c^2\Sigma_{ii}^{-2}}}{2\Sigma_{i,i}^{-1}} = c.
\]
 Note that  $\begin{bmatrix} cI & F \\ F^\top  & cI \end{bmatrix}$ is PSD whenever $ c > 0$ and the Schur complement $\frac{1}{c}(c^2I-F^\top F)  \succeq 0 $. Since $\|F^\top F\|_2 = c^2\|S\|^2_2 \leq \max_i\, c^4$ then this property holds whenever $  c<1$.

Finally,  the expansion of $H_2$  
\begin{eqnarray*}
H_2&=& \begin{bmatrix}
  cI &  F\\
 F^\top & cI
\end{bmatrix}-\begin{bmatrix} D_1^\top  \\ D_2^\top \end{bmatrix} (A+2\mu I)^{-1}\begin{bmatrix} D_1 & D_2\end{bmatrix} \\
&=& \begin{bmatrix}
  cI &  F\\
 F^\top & cI
\end{bmatrix}- 
\left(
(2\mu)^{-1} G - (2\mu)^{-1}     
G
(2\mu C^{-1}+ G )^{-1}
G \right)\\
\end{eqnarray*}
for $G = \begin{bmatrix} D_1^\top  \\  D_2^\top \end{bmatrix} \begin{bmatrix} D_1 &  D_2\end{bmatrix}$.
This can be shown using elementary calculations.

\end{proof}


Figure \ref{fig:eig_runtime} shows the runtime of Algorithm \ref{alg:mu} vs eigenvalue decompositions using full (\texttt{eig}) or fast partial (\texttt{eigs}) operations. By leveraging low-rank structure, we significantly reduce the runtime complexity to depend critically on $q$ rather than $n$.

\begin{tabular}[t]{ccc}
\begin{minipage}[t]{.4\textwidth}
\begin{algorithm}[H]
    \caption{Compute $\mu$}
    \label{alg:mu}
    \textbf{Input}: $D_1,D_2,W$, $\mu_0$\\
    \textbf{Output}: $\mu$ so that $\Delta + \mu\succeq 0$ in \eqref{eq:deltadef}
    
    \begin{algorithmic}[1] 
    \STATE $[U,\Sigma ,V] = \mathbf{svd}(W)$
    \STATE Compute $S$, $C$, $G$,  $F$ as in Th.\ref{th:lowrankupdate}.
    \STATE Initialize $\mu =  \mu_0$
    \WHILE  {$\lambda_{\min}(H_2)<10^{-15}$}
    \STATE Compute $H_2$ as in Th.
    \ref{th:lowrankupdate}.
      \STATE $\mu \leftarrow 2 \mu$
    \ENDWHILE
    \end{algorithmic}
\end{algorithm} 
\end{minipage}
&\hspace{1em}&
\begin{minipage}[t]{.5\textwidth}
\begin{figure}[H]
    \centering
\includegraphics[width=.9\linewidth]{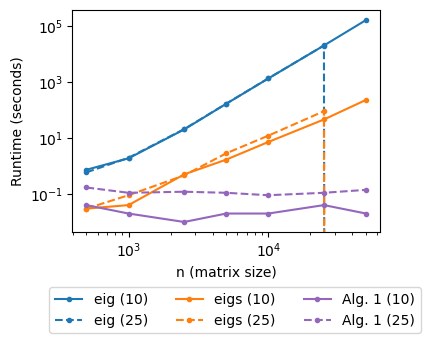}
    \caption{Runtime comparison of full eigenvalue decomposition (eig), sparse iterative eigenvalue solver (eigs), and our method (Alg 1). Legend includes ($q$) value.}
    \label{fig:eig_runtime}
\end{figure}
\end{minipage}
\end{tabular}

\subsection{Enhancements}
In the next few sections, we discuss important enhancements to the method, to improve stability and convergence properties. We refer to updating $H_t = B_t$ as a \textit{direct} update, and  $H_t = B_t^{-1}$ as an \textit{inverse} update.

\subsubsection{$\mu$ correction}
In Theorem \ref{th:lowrankupdate}, note that this choice of $\mu_t$ guarantees that $\Delta_t +\mu_t I$ is PSD, which is a sufficient, but not necessary, condition for $H_{t+1}$ to be PSD (provided $H_t$ is PSD). 
However, often estimating $\mu_t$ this way is overly pessimistic, and can be estimated to be far larger than needed for $B_t + \Delta_t + \mu_t I$ to be PSD. In cases where $\mu_t\to \infty$, this presents numerical stability issues in the inverse update. Moreover, even if $\mu_t$ is simply bounded away from 0, this prohibits superlinear convergence. Therefore, periodically, we use a Lanczos method to estimate $\hat \mu_{t} = \lambda_{\min}(H_t^{-1})$ directly ($O(n^2)$). Then, in iterations $i_t \geq  t$ between these periodic estimates, if $\tilde \mu_{i_t}$ is the output of Alg. \ref{alg:mu}, we use 
\[
\Delta\mu_{i_t} = \min\{\mu_t,\hat\mu_{i_t}\}, \quad 
\mu_{i_t} = \tilde \mu_{i_t} - \Delta\mu_{i_t}, \quad \hat \mu_{i_t+1} = \hat\mu_{i_t} - \Delta\mu_{i_t}.
\]
Essentially, this method occasionally computes the ``surplus PSD'' of $H_t$, and uses it to taper out the updates of $\mu_t$. Specifically, note that $H_t - \bar \mu_t$ is PSD for all $t$. 
This offers a computationally cheap method of producing a diagonal estimate $\mu_t\to 0$ as 
 $x_t\to x^*$.

\subsubsection{$\mu$ rescaling}
As previously mentioned, another downfall of having $\mu_t$ grow too quickly is that the inverse update can become unstable, especially if $\mu_t\to \infty$. 
To mitigate this, a $\mu$-rescaling method modifies the step size to compensate:
\[
x_{t+1} = x_t + \min\{1,\mu_{t-1}^{-1}\} B_{t}^{-1}\nabla f(x_t).
\]
In practice, rescaling helps stabilize the iterates, but can sometimes prevent the speedup of using multiple secants.

\subsubsection{Rejecting vectors}
A key source of instability in all QN methods is the ill-conditioning of the matrices $S_t^\top S_t$ and $S_t^\top Y_t$. This is especially noticable  in minimizing functions with low curvature, because sequential steps often point in the same direction, so $S_t$ quickly becomes nearly low rank. 
In \citet{schnabel1983quasi}, it was proposed to use a \textit{rejection method} to mitigate this problem, by constantly removing secant vectors $s_t$ and $y_t$ to maintain good conditioning of these key matrices. While several methods are offered in \citet{schnabel1983quasi}, we focus on the \textit{inner product rule}
\[
\text{reject }s_t\text{ if }\frac{|s_t^\top s_j|}{\|s_t\|_2\|s_j\|_2}\leq \epsilon, t\neq j.
\]
The rejection is usually done with preferential treatment toward rejecting older vectors, since they are less relevant.

\subsection{Full algorithm}
The full almost-multisecant method is presented in Alg. \ref{alg:ams-qn}. Note that we allow for two variations: direct, where $B_{t+1}$ is updated and inverted at each step, and inverse, where $B_{t+1}^{-1}$ is updated at each step using the Woodbury inversion.

\begin{algorithm}[ht!]
    \caption{Almost multisecant Quasi-Newton (AMS-QN)}
    \label{alg:ams-qn}
    \textbf{Input}: $x_0, \alpha, q, f(x), \nabla f(x)$,  $\mu$-correction period $\nu$\\
    \textbf{Output}: $f_{t+1}(x)$
    \begin{algorithmic}[1] 
        \STATE $B_0 = I$
        \FOR{$k = 1,\dots,T$}
        \STATE Update $S_t$ and $Y_t$ using \eqref{update-sy-1} or \eqref{update-sy-2}, and \eqref{update-bigYS}
         \STATE Reject all violating secant vectors in $S_t$ and correspondingly in $Y_t$

        \STATE Compute $D_1$, $D_2$, $W$ according to the specific QN method 
        \STATE Update 
        
        $ \tilde H = B_t + C_{1,t} A^{-1}_t C_{2,t}^\top $  (direct), or
        
        $\tilde H = B_t^{-1} + D_{1,t} W^{-1}_t D_{2,t}^\top $ (inverse)
        
        and symmetrize $H = \frac{(\tilde H+\tilde H^\top )}{2}$.
        
        \STATE \textbf{Compute $\mu$}: Use Alg. \ref{alg:mu} to pick $ \tilde \mu_t$ such that $H+\tilde \mu_t I$ is PSD.
        \IF {$\mod(k,\nu) == 0$}
        \STATE \textbf{Correct $\mu$}: $\hat \mu = \min(\lambda_t(H))$
        \ENDIF
        \STATE \textbf{Correct $\mu$}: 
        $\Delta \mu = \min(\hat \mu,\mu_t)$, $\hat \mu = \hat \mu - \Delta\mu$, $\mu_t = \tilde \mu_t - \Delta\mu$

        \STATE Update Hessian estimate
        
        $B_{t+1} = H + \mu_t I$ (direct) or 
 
        $B^{-1}_{t+1} = H  + \mu_t I$ (inverse)

        \STATE Update with $\mu$-scaled step size
        $x_{t+1} = x_t - \alpha_t B^{-1}_{t+1}\nabla f(x_t)$  
 where 
 \[
 \alpha_t = \begin{cases} \alpha & \text{ if direct update or no $\mu$-scaling}\\ \min\{\alpha, 1/\mu_t\} & \text{ if inverse update and $\mu$-scaling}\end{cases}
 \]
        \ENDFOR  
    \end{algorithmic}
\end{algorithm}

\section{Superlinear convergence}
\label{sec:superlinear}

We now give the superlinear convergence proof, which extends the well-known results of single-secant BFGS to MS-BFGS with symmetrization and diagonal perturbation. 

\paragraph{Assumptions.}
 Take $F_0 = \nabla^2 f(x_0)$.
The following are assumed : 
\begin{enumerate}
    \item 
    The function $f$ is strongly convex and smooth, and there exists constants $m$ and $M$ such that for all $\xi = F_0^{-1/2} x$, within a relevant local neighborhood of the solution,
    \begin{equation}
    mI \preceq g'(\xi) = F_0^{-1/2}\nabla^2 f(F_0^{-1/2}\xi)F_0^{-1/2} \preceq MI.
    \label{eq:hess_cond}
    \end{equation}

    \item The Hessians are $L$-Lipschitz, such that 
    \begin{equation}
    \|g'(\xi_1)-g'(\xi_2)\|_2\leq L \|\xi_1-\xi_2\|_2 \label{eq:lipschitz_hessian}
    \end{equation}
    Using Lemma \ref{lem:smoothness_vectors},   this implies
\[
\|g(\omega) - g(\tau) - g'(\tau)(\omega-\tau)\|\leq \frac{L }{2   } \|\omega-\tau\|^2.
\]
 \item The diagonal perturbation constant $\mu_t$ is a decaying sequence, such that 
\[
\sum_{t=0}^\infty\mu_t \leq \bar\epsilon :=\min\{1/4,1/(8M)\}
\]
\end{enumerate}

\begin{theorem}[$q$-superlinear conv.]
\label{th:superlinear}
Given the listed assumptions,
\[
\frac{\| B_t  S_t -  S_t\|_F}{\| S_t\|_F} \to 0
\]
which implies $q$-superlinear convergence.
\end{theorem} 
The proof is long and given in Appendix \ref{app:superlinear_conv}, Th. \ref{th:app:superlinear}. The exact statement in Th. \ref{th:app:superlinear} uses scaled variables, but is equivalent to the statement with unscaled variables. The proof structure follows the original structure presented in \citet{dennis1977quasi}, and further expanded in \citet{lui2021superlinear} and \citet{nocedal1999numerical}. 
The key steps to extending to multisecant is Lemma \ref{lem:multisecant_asymmproj}, which characterizes the size of the asymmetric  projection operator. The rest of the linear algebra facts extended more naturally, with a constant overhead factor of $p$ at times. 
Regarding symmetrization, Lemma \ref{app:onestep_contract} demonstrates that, contrary to expectations, it does not affect convergence analysis significantly.
The main difficulty is extending to the PSD perturbation. Notably, if the parameter $\mu_t$ does not decay to 0, it is impossible to achieve superlinear convergence. 
This is in spite of the PSD perturbation being proposed in other works \citep{goldfeld1966maximization}. 
   To overcome this, our two-stage   perturbation of $\mu_t$ is essential to force $\mu_t\to 0$ in such a way that it is summable. 
   Then, initializing close enough to the optimum, we are indeed able to maintain local linear convergence, which is a key step in proving local superlinear convergence. 
Finally, the extension of the linear-to-superlinear convergence from single-secant \citep{nocedal1999numerical} to multisecant requires some manipulations of trace and determinants of $q\times q$ matrices, and the use of the AM/GM inequality, but otherwise follows the standard framework.

\section{Numerical results}
\label{sec:numerical}

We now explore the performance of these methods on unconstrained, smooth, convex, non-quadratic problems which are bounded below. First, we do a deep study into logistic regression problems with variable conditioning, and then we apply the method on a wider array of problems.
All tables include the number of iterations until the stopping condition of 
$\frac{\|\nabla f(x_t)\|}{\|\nabla f(x_0)\| }\leq \epsilon_{\mathrm{tol}}$.

\subsection{Logistic regression}
\label{sec:logistic_regression}
 The logisitc regression problem is defined as 
\begin{equation}
    \displaystyle\min_{x\in \mathbb{R}^n} f(x) = \displaystyle\min_{x\in \mathbb{R}^n}-\frac{1}{p}\sum_{i=1}^{p}\log(\sigma (b_i a_i^\top  x)), \qquad \sigma(x) = \frac{1}{1+e^{-x}}
    \label{eq:logistic_problem_2}
\end{equation}
where $a_i^\top $ is the $i$th row vector in the data matrix $A\in\mathbb{R}^{m\times n}$ and $b_i$ the $i$th element of the label vector $b\in\{-1,1\}^{m}$. 
Here,
\begin{equation}
 A_{i,j} =  b_iz_{i,j}(1-c_j) + \omega z_{i,j} c_j
\end{equation}
where
$c_j = \exp(-\bar c j/n)$ is the data decay rate (decaying influence of each feature), and 
$z_{i,j} \sim \mathcal N(0,1)$ Gaussian distributed i.i.d. $\omega$ controls the signal to noise ratio of the data, and the labels $b_i \in \{1,-1\}$ with equal probability (class balanced). In appendix \ref{app:extranumerics:logreg}, we experiment with both a high and low signal model.

Figure \ref{fig:ablation-baselines} illustrates the destructive effect of multisecant QN methods when applied to convex problems. 
Note the trade-off between computational efficiency and numerical conditioning; while in both cases multisecant methods suffer stability issues, the inverse update (inv case) is more debilitating. 

Figure \ref{fig:ablation-ours} compares the performance of enhancements for the multisecant methods, on a difficult (ill-conditioned with high signal) problem, where only inverse updates are used. We explore the effects of symmetrization, PSD projection (infeasible in practice), and our diagonal perturbation, with and without vector rejection. 
There are two clear observations. First, as demonstrated in Figure \ref{fig:ablation-baselines}, the Woodbury inverse update, despite its instability, can sometimes suddenly converge to points with low gradient norms. Second, our approach—particularly with the rejection mechanism—demonstrably enhances the existing methods. While it does not guarantee stability, it appears to improve the situation, and overall reduces the time until convergence.

Figure \ref{fig:ablation-leftover}
gives a closer comparison of the three extra techniques: PSD correction, in which $\mu_t\to 0$ by occasionally recomputing the smallest eigenvalue of $B_t$ or $B_t^{-1}$; scaling, e.g. $d_t = \mu_t^{-1}(B_{t,\text{symm}}+\mu_tI)$ whenever $\mu_t>1$; and rejection. Although PSD correction is essential for the convergence results, it does not really have noticeable positive effect in the numerics, and moreover causes the most overhead. Scaling helps sometimes, but not consistently. The most significant improvement is through rejection. These observations are also reflected in the more extensive tables, to be presented next.

\begin{figure}[ht!]
    \centering
\includegraphics[width=\linewidth]{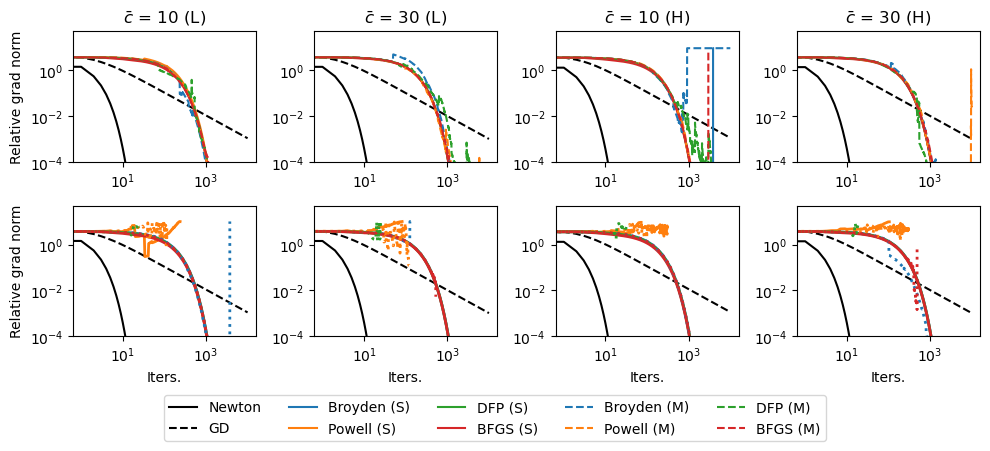}
    \caption{Comparison of Newton, gradient descent (GD), single-secant QN methods (S), and multi-secant QN methods (M) on logistic regression with $m = 200, n = 100, q = 5$. \textbf{Top}: direct solve. \textbf{Bottom}: Woodbury inverse.  Both high (H) signal and low (L) signal regime problems are tested. }
    \label{fig:ablation-baselines}
\end{figure}

\begin{figure}[ht!]
    \centering
\includegraphics[width=\linewidth]{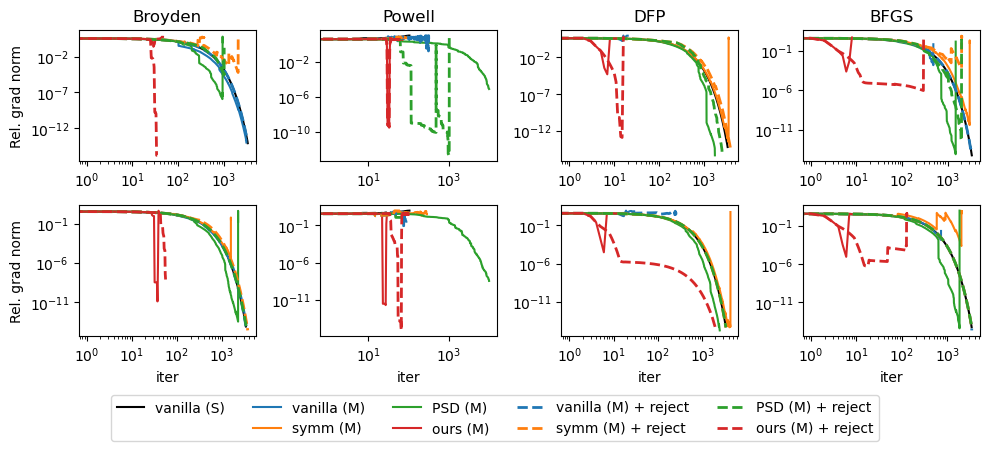}
 \caption{Comparison of QN method improvements, including symmetrization, PSD projection, and our simple diagonal boost. The problem sizes are  $m=200,n=100$ and $q=5$ for multisecant methods. All are using Woodbury inverse update.  \textbf{Top}: secants built using curve-hugging. \textbf{Bottom}: secants built using  anchored at most recent. The problem is $\bar c = 30$ (H).    }
    \label{fig:ablation-ours}
\end{figure}

\begin{figure}[ht!]
    \centering
\includegraphics[width=\linewidth]{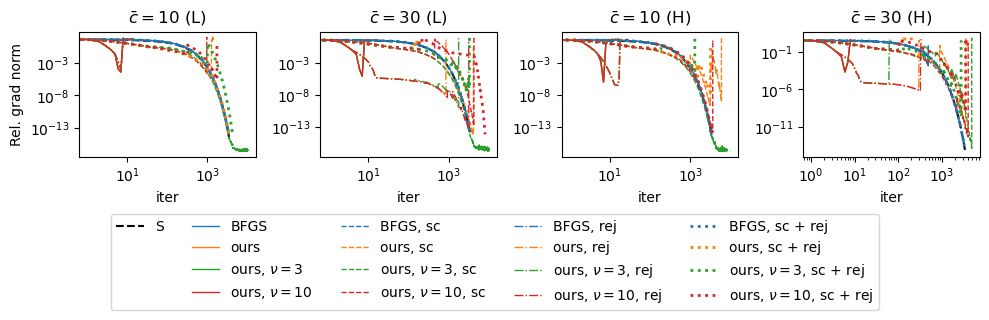}
 \caption{  Ablation of several techniques: PSD correction ($\nu > 0$), scaling, and rejection.    The problem sizes are  $m=200,n=100$ and $q=5$ for multisecant methods.   }
    \label{fig:ablation-leftover}
\end{figure}

Table \ref{tab:logreg_main} gives the number of iterations $\bar t$ to reach $\epsilon_{\mathrm{tol}} = 10^{-4}$. In each of these tables, the best result in each problem is bold. If the best result is PSD projection, which is unrealistic, it is marked by (*) and the second best score is also bold. 

There were several factors which were not numerically significant. We did not observe much performance difference in using curve-hugging \eqref{update-sy-1} or anchor-at-recent \eqref{update-sy-2} for secant updates. We also did not observe significant benefit to driving $\mu_t\to 0$ in practice ($\mu$-correction), nor of $\mu$-scaling. Almost all good results happened with inverse updates rather than direct updates. Extended tables include more ablations, and are in Appendix \ref{app:numerical}.

\begin{table}[ht!]
{\small
\begin{center}
\begin{tabular}{|l||cc|cc||l|cc|cc|}
\hline
 & \multicolumn{2}{c|}{$\bar c = 10$}  	& \multicolumn{2}{c||}{$\bar c =  30$} 	&& \multicolumn{2}{c|}{$\bar c = 10$}  	& \multicolumn{2}{c|}{$\bar c =  30$} 		\\
&cu 	&an 	&cu 	&an &&cu 	&an 	&cu 	&an 	\\\hline
Newton's	&  11	 &  11	 &  11	 &  11&Grad. Desc. 	&2051	 &2051	 &2010	 &2010	 \\\hline
Br.* (1)        	& 520	 & 520	 & 513	 & 513	 & Pow.* (1)       	& 532	 & 532	 & 529	 & 529	\\
Br. (1)       	& 520	 & 520	 & 513	 & 513&Pow. (1)       	& Inf	 & Inf	 & Inf	 & Inf		\\
Br. (v) 	& 558	 & 507	 & 471	 & 593	 &Pow. (v)	& Inf	 & Inf	 & Inf	 & Inf\\
Br. (v,r)& 505	 & 521	 & 502	 & 514	&Pow. (v,r)	&  Inf	 & Inf	 & Inf	 & Inf\\
Br. (s) 	& Inf	 &2122	 & 903	 & 559&	Pow. (s)	 & Inf	 & Inf	 & 407	 & 308	\\
Br. (s,r)	&1000	 & 631	 &2025	 & 712	&Pow. (s,r)	& Inf	 & Inf	 & Inf	 & Inf	\\
Br. (p)	& 425	 & 454	 & 144	 &   \textbf{6*}	&Pow. (p)	& 537	 &1822	 & 367	 & 377	 \\
Br. (p,r)	& Inf	 & 631	 & Inf	 & 677	&Pow. (p,r)	&2002	 & Inf	 & 465	 & 754\\
Br. (o)	&  \textbf{21}	 &   \textbf{21}	 &   \textbf{8}	 &  Inf&	Pow. (o)	&   8	 &  Inf	 &   8	 &   \textbf{6} \\
Br. (o,r)	            & 119	 & 599	 &    \textbf{8}	 & 602	&  Pow. (o,r)	&   \textbf{7}	 &   \textbf{7}	 &   \textbf{7}	 &   7	\\
Br. (o,32)	&  \textbf{21}	 &  \textbf{21}	 &    \textbf{8}	 &  Inf	&Pow. (o,32)	&   8	 &   Inf	 &   8	 &   \textbf{6}	\\
 Br. (o,32,r)	& Inf	 & 599	 &    \textbf{8}	 & 602	&Pow. (o,32,r)	&   \textbf{7}	 &   \textbf{7}	 &   \textbf{7}	 &   7\\
\hline
DFP* (1)       	& 504	 & 504	 & 500	 & 500	 &BFGS* (1)       	& 502	 & 502	 & 498	 & 498	\\
DFP (1)       	& 504	 & 504	 & 500	 & 500	&BFGS (1)       	& 502	 & 502	 & 498	 & 498 \\
DFP (v)	& Inf	 & Inf	 & Inf	 & Inf	& 	BFGS (v) 	& 499	 & 502	 & 500	 & 502\\
DFP (v,r)	& Inf	 & Inf	 & Inf	 & Inf&	BFGS (v,r)& 502	 & 503	 & 500	 & 501\\
DFP (s)	& 530	 & 513	 & 524	 & 513	 &	BFGS (s) 	& 530	 & 539	 & 926	 &1151 \\
DFP (s,r)	& 760	 & 511	 & 548	 & 510	& 	BFGS (s,r)& 884	 & 507	 & 588	 & 505\\
DFP (p)	& 425	 & 708	 & 434	 & 369	 	&BFGS (p) 	& 265	 & 268	 & 152	 & 183	\\
DFP (p,r)	& 703	 & 511	 & 438	 & 501	&BFGS (p,r)& 665	 & 507	 &1006	 & 505	 \\
 DFP (o)	&   \textbf{7}	 &   \textbf{7}	 &   \textbf{6}	 &   \textbf{6}	&BFGS (o) 	&   \textbf{5}	 &   \textbf{5}	 &   \textbf{5}	 &  \textbf{5} \\
 DFP (o,r)	&   Inf	 &  12	 &   \textbf{6}	 &  12		& BFGS (o,r)&  10	 &  10	 &  10	 &  10	\\
 DFP (o,32)	& Inf	 & Inf &   \textbf{6}	 &   \textbf{6}	&BFGS (o,32)&   \textbf{5}	 &   \textbf{5}	 &   \textbf{5}	 &   \textbf{5}\\
 DFP (o,32,r)	&   Inf	 &  12	 &   \textbf{6}	 &  12	& BFGS (o,32,r)	&  10	 &  10	 &  10	 &  10	\\ 
\hline
\end{tabular}
\end{center}
}
\caption{\textbf{LogReg results summary.} Number of iterations with $\epsilon_{\mathrm{tol}} = 10^{-4}$. $q = 5$ multisecant vectors. Inf = more than 10000 iterations, or diverged. $\sigma = 10, m = 2000, n = 1000$.       None use $\mu$-scaling. * = direct update, all else are inverse updates. 1 = single secant, v = vanilla, s = symmetric, p = PSD projection, o = ours, r = rejection used, with tolerance 0.01. cu = curve fitting, an = anchored at most recent. The number refers to $\nu$, in $\mu$-correction. A more extensive table can be found in Appendix \ref{app:extranumerics:logreg}.}
\label{tab:logreg_main}
\end{table}

 \subsection{$p$-order minimization}
 In Table \ref{tab:pnormraised_main_2.5}, we investigate an important problem in robust optimization 
 \[
f(x) =\frac{1}{2m} \|Ax - b\|^p_p, \qquad p>1.
\]
Here, we generate the data as 
\[
Z_{i,j}\sim \mathcal N(0,1), \quad W_{i,j}\sim \mathcal N(0,1), \quad x_{j}\sim \mathcal N(0,1), \quad i = 1,...,m,\;j = 1,...,n
\]
and 
\[
\tilde A_{i,j} = Z_{i,j} c_j, \quad A = \frac{\tilde A}{ \|\tilde A\|_2},  \quad b = \frac{Ax + \sigma N}{ \|Ax + \sigma N\|_2}.
\]
The normalization steps are used to control the signal-to-noise ratio, and so that the same step size can be applied for all values of $m$, $n$, $\sigma$, etc. 

\begin{table}[ht!]
{\small
\begin{center}
\begin{tabular}{|l||ccc|ccc|}
\hline
&\multicolumn{3}{c|}{Medium noise ($\sigma = 1$)}&\multicolumn{3}{c|}{Low noise ($\sigma = 0.1$)}\\
  & $\bar c = 10$  	&$\bar c =  30$ 	&$\bar c =  50$ 	&$\bar c = 10$ 	&$\bar c =  30$ 	&$\bar c =  50$ 	\\\hline
Newton's 		 &5190	 &5178	 &5239	 &5228	 &5201	 &5255\\
Grad. Desc. 		 & Inf	 & Inf	 & Inf	 & Inf	 & Inf	 & Inf\\\hline
Br. (d,S)        		 & Inf	 &9670	 & Inf	 & Inf	 & Inf	 & Inf\\
Br. (d,v)		 &2197	 & 669	 & \textbf{428}	 & 737	 & 638	 & 837\\
Br. (d,s)		 &8542	 &3460	 &4897	 &7600	 &6288	 &3614\\
Br. (d,o)		 & Inf	 &4752	 &3798	 &7360	 &3426	 &4904\\
Br. (i,v)		 & \textbf{1549}	 & Inf	 &1379	 &1699	 &3903	 & 682\\
Br. (i,s)	 &2393	 &1083	 & 468	 & \textbf{560}	 & \textbf{531}	 & \textbf{486}\\
Br. (i,p)		 &9475	 &4280	 & Inf	 & Inf	 &4273	 &3889\\
Br.* (i,o) 		 & Inf	 &9856	 &8819	 &9910	 &8139	 & Inf\\
Br.* (i,o,22)		 &1860	 & \textbf{504}	 &3180	 &1475	 &1396	 &3653\\
Br.* (i,o,500)		 & Inf	 & Inf	 &7836	 &4707	 &7015	 &9680\\\hline
Pow. (d,v)		 & Inf	 &1158	 & Inf	 & Inf	 & 951	 & Inf\\
Pow. (d,s)		 &\textbf{4216}	 &2811	 &3901	 &6321	 &\textbf{3493}	 &\textbf{1868}\\
Pow. (d,o)		 &6100	 &\textbf{2588}	 &\textbf{3266}	 &5193	 &3814	 &3063\\
Pow. (d,o,$22$)		 &5183	 &2884	 &4511	 &\textbf{3943}	 &4482	 &3774\\
Pow. (d,o,$500$)	 &4336	 &3267	 &3883	 &4039	 &3155	 &3492\\\hline
DFP (i,s)		 & 512	 & 589	 & 495	 & 505	 & 506	 & 489\\
DFP (i,p)		 &7189	 &2324	 &2156	 &8332	 &4935	 &4825\\
DFP* (i,o)		 & \textbf{490}	 & 477	 & 479	 & 493	 & 480	 & 487\\
DFP (i,o)		 &3608	 & \textbf{465}	 & \textbf{470}	 & \textbf{479}	 & \textbf{462}	 & 486\\
DFP* (i,o,$22$)	 & 499	 & 497	 & 478	 & 493	 & 486	 & \textbf{475}\\
DFP (i,o,$22$)		 &1083	 & 542	 & \textbf{470}	 & 482	 & \textbf{462}	 & 479\\
DFP* (i,o,$500$)	 & 490	 & 477	 & 479	 & 493	 & 480	 & 487\\
DFP (i,o,$500$)	 &5673	 & \textbf{465}	 & \textbf{470}	 & \textbf{479}	 & \textbf{462}	 & 486\\
\hline
BFGS (d,v)		 & 511	 & Inf	 & 814	 &1141	 & \textbf{469}	 & 488\\
BFGS (d,s)		 &1097	 & 749	 &2294	 & 510	 & 553	 &4273\\
BFGS (d,o)		 & 614	 & 641	 &1951	 & Inf	 &4516	 & Inf\\
BFGS (i,v)		 & \textbf{459}	 &1190	 & 457	 & 472	 & 513	 &4122\\
BFGS (i,s)		 & 488	 & 823	 & 608	 & 776	 & 484	 & 502\\
BFGS (i,p)	 & 462	 & 485	 & \textbf{451}	 & \textbf{437}	 & 513	 & 805\\
BFGS* (i,o)		 & 536	 & \textbf{477}	 & 488	 & 634	 & 775	 & 512\\
BFGS (i,o)		 & 795	 & Inf	 & 466	 & 475	 & 480	 & Inf\\
BFGS* (i,o,$22$)		 & 497	 & \textbf{477}	 & 476	 &1658	 & 491	 & 573\\
BFGS (i,o,$22$)	 & 491	 & 539	 & 495	 & 564	 & 960	 & \textbf{479}\\
BFGS* (i,o,$500$)		 & 536	 & \textbf{477}	 & 488	 & 635	 & 775	 & 512\\
BFGS (i,o,$500$)		 & 796	 & Inf	 & 466	 & 475	 & 480	 & Inf\\
\hline\hline
\end{tabular}
\end{center}
}
\caption{\textbf{$p$ order minimization, $p = 2.5$.}   Number of iterations with $\epsilon_{\mathrm{tol}} = 10^{-2}$.  $q = 5$ multisecant vectors. Inf = more than 10000 iterations, or diverged. $m = 1000, n = 500$.  Lines were removed if they were all divergent, or not competitive based on similar variations.  
* = uses $\mu$-scaling. 
d = direct update, i = inverse update, 1 = single secant, v = vanilla, s = symmetric, p = PSD projection, o = ours, r = rejection used, with tolerance 0.01.
The number refers to $\nu$, in $\mu$-correction.
A more extensive table can be found in Appendix \ref{app:extranumerics:pnormraised}.}

\label{tab:pnormraised_main_2.5}
\end{table}

Table 
\ref{tab:pnormraised_main_2.5}
gives the number of iterations   to reach $\epsilon_{\mathrm{tol}}=10^{-3}$ for $p$-order minimization, $p = 2.5$. We also include experiments for $p = 1.5$ and $p = 3.5$ in Appendix \ref{app:extranumerics:pnormraised}.
Many ablations are not consistent; sometimes $\mu$-correction was essential to give convergence; other times it was not necessary, but not harmful (with a few exceptions). $\mu$-scaling also helped most of the time; in contrast, in the previous experiments (Tab. \ref{tab:logreg_main}) it prevented ``surprisingly fast" convergences.  We also include PSD convergence results which in some cases were competitive, but in most cases were not, showing that the conditions  of a descent direction, and of well-conditioning of the Hessian estimate, are both needed for good convergence. Overall, however, we conclude that a multisecant approach significantly enhances convergence speed, and diagonal perturbation often enables convergence in cases that would otherwise diverge.

\subsection{Cross-entropy loss}
Finally, we consider the cross-entropy loss function, commonly used in multiclass logistic regression in machine learning. Here, $x\in \R^{n\times n_c}$, where $n_c$ is the number of classes. Then, for data and labels generated as
\[
Z_{i,j}\sim \mathcal N(0,1), \quad W_{i,k}\sim \mathcal N(0,1), \quad x_{j,k}\sim \mathcal N(0,1),
\quad
\tilde A = Z_{i,j} c_j, \quad A_{i,j} = \frac{\tilde A}{ \|\tilde A\|_2}
\]
for $i = 1,...,m$, $j = 1,...,n$, $k = 1,...,n_c$
and for $a_i$ and $x_k$ the $i$th and $k$th column of $A$ and $X$,
\[
b_i = \argmax{k=1,...,n_c} a_i^\top x_k + \sigma W_{i,k}.
\]
Then the cross-entropy loss function is
 \[
f(X) = -\sum_{i=1}^m  a_i^\top x_{b_i} + \log \left( \sum_{j=1}^m e^{a_i^\top x_j} \right)
\]

\begin{table}[ht!]
{\small
\begin{center}
\begin{tabular}{|l||ccc|ccc|}
\hline
&\multicolumn{3}{c|}{High noise ($\sigma = 1.0$)}&\multicolumn{3}{c|}{Medium noise ($\sigma = 0.1$)}\\
 & $\bar c = 10$  	&$\bar c =  30$ 	&$\bar c =  50$ 	&$\bar c = 10$ 	&$\bar c =  30$ 	&$\bar c =  50$ 	 \\\hline
Grad. Desc.	& Inf	 & Inf	 & Inf	 & Inf	 & Inf	 & Inf	\\\hline

Br. (i,s) 	&\textbf{1006}	 &9679	 &4566	 &\textbf{1052}	 & Inf	 & Inf	\\
Br. (i,o,s,10) 	&1435	 &\textbf{8457}	 &\textbf{4519}	 &5284	 &\textbf{3089}	 & Inf	 \\\hline
  
DFP (i,s) 	&1206	 & Inf	 & Inf	 & Inf	 & Inf	 & Inf	 \\
DFP (i,o,s) 	&2924	 & Inf	 & Inf	 &3727	 &3469	 & Inf	\\
DFP (i,o,s,10) 	&\textbf{1037}	 & Inf	 & Inf	 & \textbf{852}	 & \textbf{917}	 & Inf	\\
DFP (i,o,10) 	& Inf	 & Inf	 & Inf	 & Inf	 & Inf	 & Inf	 \\
DFP (i,o,s,100) 	&2064	 & Inf	 &\textbf{8840}	 &1362	 & 972	 &\textbf{1315}	 \\\hline
BFGS (d,1)       	& Inf	 & Inf	 & Inf	 & Inf	 & Inf	 & Inf	 \\
BFGS (d,v)	& 817	 & Inf	 &\textbf{1177}	 & \textbf{681}	 &1035	 & Inf	 \\
BFGS (d,s)	&1093	 & Inf	 &6714	 & Inf	 & Inf	 & Inf	 \\
BFGS (d,o)	&1494	 & Inf	 &9609	 &1497	 &2012	 & Inf	 \\
BFGS (d,o,10)	&1153	 &7282	 &9609	 &1497	 &2683	 & Inf	\\
BFGS (d,o,100)	& Inf	 & Inf	 &9609	 &1497	 &2168	 & Inf	 \\
BFGS (i,v) 	& \textbf{666}	 &\textbf{3069}	 &1907	 & 691	 & \textbf{830}	 & Inf	 \\
BFGS (i,v,r)	& Inf	 & Inf	 & Inf	 & Inf	 & Inf	 & Inf	\\
BFGS (i,s) 	&1296	 &5729	 &2523	 &1001	 &1471	 & Inf	\\
BFGS (i,p) 	&1415	 &5838	 &3049	 &1109	 &1220	 & Inf	 \\
BFGS (i,o,s) 	&6664	 & Inf	 & Inf	 &4435	 &5581	 &\textbf{9759}	\\
BFGS (i,o,s,10) 	&1303	 & Inf	 &2649	 & Inf	 &1170	 & Inf	\\
BFGS (i,o,s,100) 	& Inf	 & Inf	 &2565	 &2244	 & Inf	 & Inf	\\
BFGS (i,o,s,r)	&3251	 & Inf	 & Inf	 &2768	 &4816	 & Inf	 \\
BFGS (i,o,s,100,r)	&5830	 & Inf	 & Inf	 &4750	 & Inf	 & Inf	 \\
\hline

\hline
\end{tabular}
\end{center}
}
\caption{\textbf{Cross entropy loss summary.} Number of iterations with $\epsilon_{\mathrm{tol}} = 10^{-3}$.  $q = 5$. Inf = more than 10000 iterations. $m = 200, n = 100, n_c = 10$.   Some lines where no experiments converged or were competitive were removed. * = uses $\mu$-scaling. 
d = direct update, i = inverse update, 1 = single secant, v = vanilla, s = symmetric, p = PSD projection, o = ours, r = rejection used, with tolerance 0.01.
The number refers to $\nu$, in $\mu$-correction.
A more extensive table can be found in Appendix \ref{app:extranumerics:mcloss}.}
\label{tab:mcloss_main}
\end{table}

 Table \ref{tab:mcloss_main}
 gives the number of iterations to reach $ \epsilon_{\mathrm{tol}} = 10^{-3}$. Many of the experiments did not converge; for example, none of the Powell variations, or the direct update variations for Broyden or DFP converged. In comparison, BFGS is much more stable across the board. 
 While in many cases, a vanilla or plain symmetrized version seems strong, there are also cases where our update, coupled with $\mu$-correction, $\mu$-scaling, and rejection, is competitive. 
 
 Overall, the multiclass cross-entropy problem served to be a far more difficult problem than its related counterpart, binary logistic regression. This is partially due to the block-diagonal structure of the Hessian, which seems to worsen conditioning. This also resulted in Newton's method being significantly slower for this problem, which is why we did not run it. (Note, however, that gradient descent is not much better.) 
 
\subsection{Discussion} 

From numerical experiments, we draw several conclusions. When tackling difficult problems (e.g., ill-conditioned Hessians, extreme SNR values common in real-world applications), gradient descent and Newton's method struggle significantly. 
Gradient descent requires many iterations to converge, though its complexity-per-iteration is comparable to that of the QN methods when memory is cheap. 
Newton's method noticeably requires fewer iterations, but that too can depend on problem conditioning; this is observed not only in a longer iteration complexity, but also in the   time required for each direct solve step  to complete within a tolerable precision. 
Additionally, there is almost always a marked improvement from using a single-secant QN method to a multisecant QN method in these problem settings, underscoring the value of developing multisecant QN methods.

 The case to improve MS-QN methods is now clear and well-motivated; in particular, as previously discussed, the quality of ``descent direction" does not carry over for MS-QN methods for general convex problems.  Yet curiously, this does not seem to consistently  hamper performance; in particular, Broyden's method seems to function well in vanilla form. Powell's method, on the other hand, is the most often unstable method, in both the inverse and direct update scenarios. 
 The BFGS method is overall the strongest method, and seems indeed improvable using our diagonal perturbation, though of varying degrees.

One unsatisfying aspect in this study is that the effects of the various improvements ($\mu$-correction, $\mu$-scaling, and rejection method) do not appear to offer consistent improvements. Finding a definitive solution for each problem setting remains elusive, though an adaptive approach—testing improvements and selecting the best at each step—might be promising.

\section{Limited memory multisecant BFGS}
\label{sec:extension}
For very large problems, the proposed QN methods become computational infeasible, even in their inverse update form (which avoids solving linear systems). In this case, even storing a dense $n\times n$ matrix is prohibitive.
Therefore, the limited memory extension is essential for this level of scalability. 
The general idea is to approximate $B_{t+1}^{-1} g_t$ using only the past $L$ terms $(s_i,y_i)$, $i = t,t-1,...,t-L+1$ in the single-secant methods, and $(S_i,Y_i)$, $i = t,t-1,...,t-L+1$ in the multisecant methods.
This is achieved via the approximation that $B_{t-L} = I$. 

Limited memory versions of QN methods have been previously studied \citep{van2005limited,reed2009broyden,kolda1997limited,erway2015efficiently} with the most popular the  L-BFGS method \citep{zhu1997algorithm}, which takes advantage of the specific form of BFGS to form a memory-optimized two-loop algorithm. 
To extend limited memory to general QN methods, one direct approach is to simply recompute the intermediate matrices $B^{-1}_iY_j$ and $B^{-1}_iS_j$ for all $i,j = t-L+1,...,t$, and use them to progressively build an approximate $B_t^{-1}\nabla f(x_t)$. 

However, our own experiments showed that such direct implementations were so numerically unstable that in general, picking $q = 1$ (single-secant) and (surprisingly) $L=1$ was always best. The exception to this observation is the multisecant L-BFGS method, where by using the well-known two-loop update strategy, the conditioning of the iterates was stable enough such that larger values of $q$ and $L$ indeed indicated speedups. 
Therefore, we focus on this specific extension; even then, often the $L = 1$ extension is the most stable.

\paragraph{Two loop L-BFGS.}
The key to the two-loop L-BFGS iteration  \citep{liu1989limited} is the fact that the inverse update can be written in a specific factored form. Specifically, when $Y_t^\top S_t$ and $Y_t^{-1}B_t^{-1}Y_t$ are invertible, then 

\begin{equation}
B_{t+1}^{-1}  
= (I-S_t(Y_t^\top S_t)^{-1} Y_t^\top ) B^{-1}_t(I-\underbrace{Y_t(S_t^\top Y_t)^{-1} S_t^\top }_{:=V_t}) - \underbrace{S_t(S_t^\top Y_t)^{-1} S_t^\top }_{\text{tail term:} R_tZ_t^\top }.\label{eq:lm-bfgs-update}
\end{equation} 
where $R_t$, $Z_t$ are easy to precompute, and $V_t\zeta$ is easy to apply. \footnote{These only require the terms $S_t$ and $Y_t$, and inverses of  $q\times q$ matrices.}
Then, defining for $i < j$, 
\[
A_{i,j} =  (I-V_i)(I-V_{i+1})\cdots (I-V_j)
\]
and
$q_{t+1} = g_t$, $q_i = (I-V_i) q_{i+1} = A_{i,k}g_t$ and
rolling out the iterates, 
\begin{multline*}
    H_{t+1} g_t = A^\top _{t-L+1,k}q_{t-L+1}      + A^\top _{t-L+2,k} R_{t-L+1}Z_{t-L+1}^\top q_{t-L+2,k}\\
     +  \cdots   +  (I-V_t) R_{t-1}Z_{t-1}^\top  q_t + R_tZ_t^\top g_t.
\end{multline*}
So, we may recursively define 
    \begin{eqnarray*}
    u_{t-L+1} &=& (I-V_{t-L+1})\gamma q_{t-L+1} + R_{t-L+1}Z_{t-L+1}^\top q_{t-L+2},\\
    u_{i+1} &=& (I-V_i)u_i + R_iZ_i^\top q_{i+1}
    \end{eqnarray*}
   where $B_{t-L+1}^{-1} = \gamma I$  and $H_{t+1}g_t = u_t$.
Furthermore, to avoid holding onto the $L$ vectors, it is custom to compute and save $a_i = Z_{i-1}^\top q_i$, and simply update $  u_{i+1} = (I-V_i)u_i + \alpha_{i+1} Z_i$.
Since $V_t$ is always a rank-$q$ matrix, then this gives a two-loop recursion that, by first computing $q_{t+1},...,q_{t-L}$, and then $u_{t-L},...,u_t$, we arrive at $H_{t+1}g_t$ without ever forming an $n\times n$ matrix, using $O(Lqn)$ operations.


 This two-loop implementation significantly reduces the amount of precompute and memory required at each iteration, from $O(q^3 L^2)$ to $O(q^3 L)$ precompute, and from $O(q n L^2)$ to $O(qnL)$  memory. However, it relies on the factored form $B_{t+1}^{-1}  = (I-U_t)B_t^{-1}(I-V_t^\top ) + R_tZ_t^\top $ where $U_t$, $V_t$, $R_t$, and $Z_t$ are low rank and do not depend on $B_t$. As was observed in previous works \citep{kolda1997limited}, the other QN methods do not seem to reduce to such convenient structure.

\paragraph{Almost multisecant L-BFGS.} As previously stated, forming a diagonal perturbation based on $D_{1,t}$, $W_t$, and $D_{2,t}$, as is done in the full-memory case, presents numerical instabilities in the limited memory case. Therefore, we modify our diagonal perturbation to only focus on the ``tail term" in \eqref{eq:lm-bfgs-update}, which only uses the most recent multisecant matrices $S_t$ and $Y_t$. The full algorithm is provided in Alg. \ref{alg:lim-ams-qn}.

\subsubsection{$\gamma$-scaling}
This scaling method (often called \textit{self-scaling}) for BFGS \citep{oren1974self} is a numerical method often used to attempt to contain the eigenvalues of the update matrix $H_k$. Specifically, adjusted for multisecant updates, the update, for $d_t = B_tS_t$ is 
\[
B_{t+1} = \gamma_t\left(B_t- B_tS_t(S_t^TB_tS_t)^{-1}S_t^\top B_t\right) +  Y_t(Y_t^TS_t)^{-1}Y_t^\top
\]
where the unscaled BFGS sets $\gamma_t = 1$ and the scaled one uses $\gamma_t = y_t^\top s_t/s_t^\top H_ts_t$. We find that for our experiments, this choice of $\gamma_t$ did not provide consistent improvements in numerical stability, but picking a constant $\gamma_t$ sometimes did.

\begin{center}
 
\begin{tabular}[t]{cc}
\begin{minipage}[t]{.43\textwidth}

\begin{algorithm}[H]
    \caption{One step L-MS-BFGS}
    \label{algo:lim_hg_bfgs}
    \textbf{Input}: $g_t$, $\gamma$, $V_j,R_j,Z_j$,  \\ 
    \hspace*{1em} for $j\in \{t-L+1,...,t\}$\\
    \textbf{Output}:  $d_t = B^{-1}_{t+1} g_t$ 
    
    \begin{algorithmic}[1] 

        \STATE $q=g_t$
        \FOR{$i = t,t-1,\dots,t-L+1$}
            \STATE $a_{i+1} = Z_i^\top q$
            \STATE $q= (I-V_i)q$
        \ENDFOR  
        
        \STATE $j = t-L+1$ 
        \STATE $u =   (I-V^\top _{j})\gamma q + R_{j}a_{j}$,

        \FOR{$i = t-L+2, \dots,t$}
            \STATE $u  = (I-V^\top _i)u + R_{i}a_{i+1}$
        \ENDFOR
        \RETURN $d_t = u$
    \end{algorithmic}
\end{algorithm}

\end{minipage}
&
\begin{minipage}[t]{.5\textwidth}

\begin{algorithm}[H]
    \caption{AMS-QN}
    \label{alg:lim-ams-qn}
    \textbf{Input}: $x_0, \alpha, p, f(x), \nabla f(x)$\\
    \textbf{Output}: $f_{t+1}(x)$
    \begin{algorithmic}[1] 
        \FOR{$k = 1,\dots,T$}
        \STATE Update $S_t$ and $Y_t$ using \eqref{update-sy-1} , \eqref{update-sy-2},  \eqref{update-bigYS}
         \STATE Reject violating secant vectors in $S_t$,$Y_t$

        \STATE Update $d_t = B_{t+1}^{-1} \nabla f(x_t)$ using Alg.  \ref{algo:lim_hg_bfgs}.

        \STATE Using 
        $D_{1,t} = D_{2,t} = S_t$,  $W_t=-S_t^\top Y_t$, use Alg. \ref{alg:mu} to pick $\tilde \mu_t$ so that $H+\tilde \mu_t I$ is PSD.

        \STATE Update with $\mu$-scaled step size
        \[
        x_{t+1} = x_t - \alpha(d_t + \mu_t\nabla f(x_t))
        \]
        \ENDFOR  
    \end{algorithmic}
\end{algorithm}

\end{minipage}

\end{tabular}

\end{center}

\begin{figure}[ht]
    \centering
    \includegraphics[width=\linewidth]{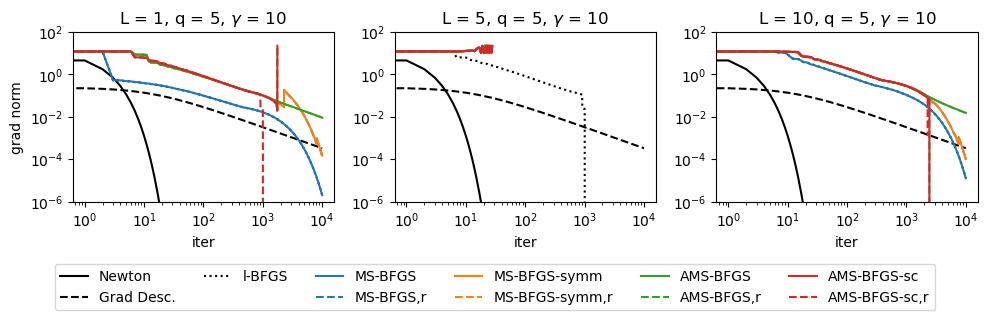}\\
    \includegraphics[width=\linewidth]{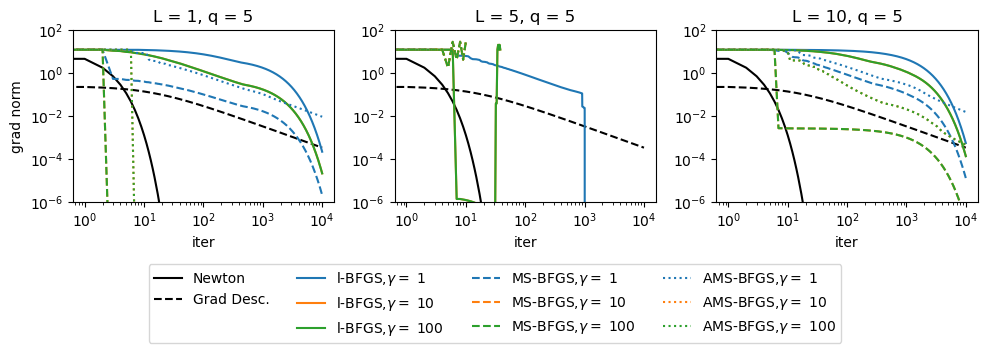}
    \caption{Performance of L-MS-BFGS on logistic regression. AMS = almost multi-secant (our method). \textbf{Top.} r = rejection.   \textbf{Bottom.} no rejection or scaling used. The problem sizes are  $m=2000,n=1000$.}
    \label{fig:lim_memory_logreg}
\end{figure}

Figure \ref{fig:lim_memory_logreg} shows the performance of the limited memory MS-BFGS method on the logistic regression problem. Stability is a critical issue, especially for larger $L$, to the point that the for lower precision solutions, gradient descent is clearly superior. However, for high precision solutions, a quasi-Newton method is still advantageous. Here, use of the  improvements (rejection, $\mu$-scaling, and $\gamma$-scaling) play a big role in improving stability.

\begin{table}[ht!]
{\small
\centering

\setlength{\tabcolsep}{5pt}
\begin{tabular}{|l||cc|cc||l||cc|cc|}
\hline
&\multicolumn{2}{c|}{Low} &\multicolumn{2}{c||}{High}&&\multicolumn{2}{c|}{Low} &\multicolumn{2}{c|}{High}  \\
&cu 	&an 	&cu 	&an  &	&cu 	&an 	 	&cu 	&an 	\\\hline   
Newton's	&  11	 &  11		&  11	 &  11	 & Grad Desc	&2051	 &2051		 &2357	 &2357	 \\\hline
($L$,$q$,type,$\gamma$,*)&&&&& ($L$,$q$,type,$\gamma$,*)&&&&\\

(1,1,1,100)      	&   4		 &   4	 &   4	 &   4	& (1,5,v,100)	&   8	 &   8	 &   8	 &   8\\
(5,1,1,100)      		& 508	 & 508	 	 &1644	 &1644& (1,5,v,100,r)	&   8	 &   8	 &   8	 &   8	\\
  \hline
(5,5,s,0.1)	& Inf	 & Inf		 &   6	 &   6	&  (1,5,o,0.1,r)	& Inf	 &   7		 & Inf	 &4125\\
(5,5,s,0.1,r)	& Inf	 &8933		 &   6	 &4368&  (1,5,o,0.1,r,sc)	& Inf	 &   7	 & Inf	 &4125	\\
(10,5,s,0.1)	& Inf	 & Inf	  &   6	 &   6	 &(10,5,o,0.1,r,sc)	& Inf	 &8456		 & Inf	 &8786	\\
(10,5,s,0.1,r)	& Inf	 &8933		 &   6	 &4138	& (10,5,o,0.1,r)	& Inf	 &8456		 & Inf	 &8786	\\
(1,5,s,1,r)	& Inf	 &7899		 & Inf	 &7993	&(5,5,o,10,sc)	& Inf	 & Inf	 & Inf	 &  28	\\
(1,5,s,100)	&   8	 &   8		 &   8	 &   8	&(5,5,o,10)	& Inf	 & Inf	 	 & Inf	 &  28\\
(1,5,s,100,r)	&   8	 &   8	  &   8	 &   8	& (10,5,o,10,sc)	& Inf	 &  39		 & Inf	 & Inf\\

\hline
\end{tabular}

}
\caption{\textbf{Logistic regression, L-MS-BFGS.}  Number  of iterations until $\|\nabla f(x_t)\|/\|\nabla f(x_0)\| \leq \epsilon = 10^{-4}$. $c = 10$.   inf = more than 10000 iterations. $\sigma = 10, m = 2000, n = 1000$.    
cu = curve hugging, an = anchored at most recent.
For type,
1 = single-secant, v = vanilla, s = symmetric, o = ours.
sc = $\mu$-scaling, r = rejection.
All rows where no experiment did better than gradient descent were removed. A more extensive table is found in Appendix \ref{app:extranumerics:logreg_limited}.
}
\label{tab:logreg_limited_main}
\end{table}

Table \ref{tab:logreg_limited_main} gives a summary of the limited memory MS-BFGS over logistic regression. 
Larger values of $L$ exacerbate the stability issue, a phenomenon that is also known in the L-BFGS literature. The method is especially powerful when $\gamma$ is hyper-tuned. There are some cases in which diagonal perturbation improves matters, but it is less consistent than in the full-memory BFGS methods.

Figure \ref{fig:logreg_timing} presents the runtimes of various logistic regression methods. For larger problems, the complexity ordering aligns with intuition: gradient descent, limited-memory BFGS, BFGS with inverse updates, BFGS with direct updates, and finally, Newton’s method.
Notably, the limited-memory extension is \textit{crucial for scalability}, though its practical implementation remains challenging.
\begin{figure}
    \centering
    \includegraphics[width=\linewidth]{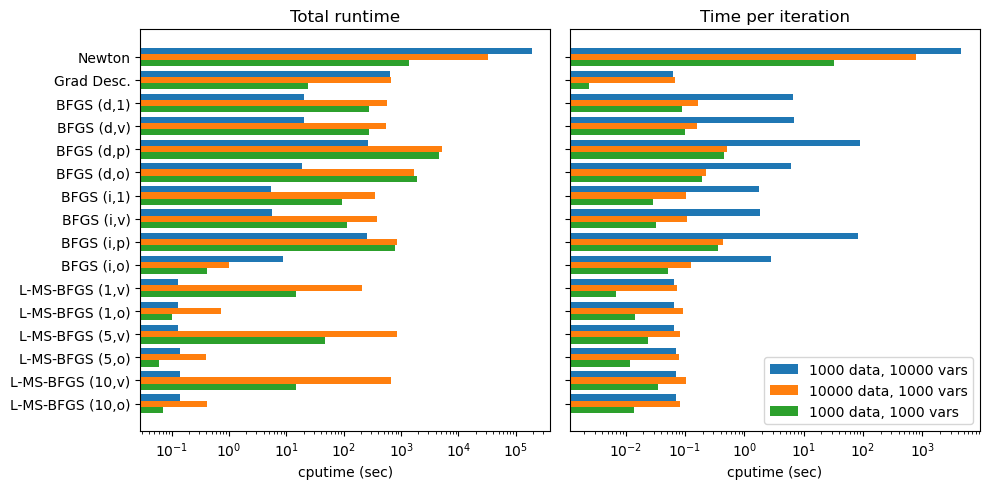}
    \caption{Runtime of various methods. d = direct update, i = inverse update. 1 = single-secant, v = vanilla multisecant, p = with PSD correction (infeasible in practice), o = with diagonal correction. For L-MS-BFGS, the first number is $L$, the limited memory size. For all MS methods, $q = 5$. }
    \label{fig:logreg_timing}
\end{figure}

\subsection{Application: Nonconvex neural network model training}
We investigate the efficacy of L-MS-BFGS in training a small neural network (Fig. \ref{fig:logreg_nn}). The nonconvex nature of the objective function introduces unique challenges; for instance, a non-decreasing loss or gradient norm trace does not necessarily indicate poor model training. Instead, performance must be assessed through the downstream task metric, such as the misclassification rate.

In this experiment, the MS methods exhibit greater instability in loss and gradient norm compared to gradient descent. However, \textit{they can sometimes achieve faster convergence in train and test misclassification rates}. This behavior aligns with a well-known phenomenon in deep learning: in networks where the final layer is logistic (for binary classification) or uses cross-entropy loss (for multiclass classification), the classifier effectively maximizes the margin. That is, \textit{even after the training data is fully fitted, further training to reduce the loss can enhance generalization}. In such landscapes, the goal is not merely to obtain a quick, suboptimal solution but to achieve a higher precision solution to the optimization problem.

\begin{figure}
    \centering
    \includegraphics[width=\linewidth]{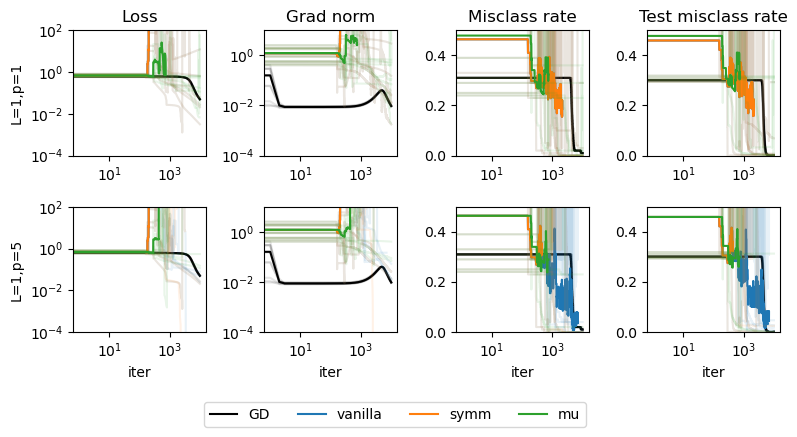}
    \caption{Two layer neural network, with 10 input features and 100 hidden neurons. Last layer is logistic layer. Problem is generated as described in Section \ref{sec:logistic_regression}. $L = 1$, $q = 5$. Dark trace is mean over 10 trial, light traces are the individual trials.}
    \label{fig:logreg_nn}
\end{figure}

\section{Conclusion}
In an era of growing problem sizes, higher-order methods leading to more precise solutions are often traded for lower-order, more approximate, and often stochastic methods.
The prevailing justification is that in many large-scale applications, approximate solutions are sufficient. 
However, even in deep neural network training, this assumption is not always true; achieving higher-precision solutions offers significant benefits for model generalization and robustness, particularly in margin-maximizing methods. Moreover, for over-parameterized models (a dominant trend in modern machine learning) the optimal solution often lies in an especially ill-conditioned region of the optimization landscape.
Therefore, scalable higher-order methods remain crucial for both scientific computing and machine learning. Multisecant methods provide a key tradeoff, improving second-order approximation while maintaining low per-iteration complexity. Additionally, limited-memory extensions integrate naturally with these methods.

Overall, there are still many areas to explore in multisecant QN methods, of which can lead to important contributions in large-scale optimization.
The most critical challenge in these methods remains numerical stability. In this work, we addressed this issue by introducing a diagonal perturbation, which efficiently approximates the full PSD projection approach of \citet{scieur2021generalization} to maintain descent steps. However, this is only a partial solution, as further refinements and hyperparameter tuning are still necessary to achieve consistently strong performance. Moreover, understanding how this technique generalizes across different methods is crucial. While our results suggest that BFGS is generally the most stable of the four methods examined, Broyden’s method often performs surprisingly well with minimal modifications, raising questions about the practical necessity of symmetric PSD Hessian approximations.
 
 Additionally, we note that the key benchmark is not gradient descent, which cannot achieve high precision solutions with competitive runtimes, but rather single-secant QN methods of each forms. That being said, in cases where the solution lies in a poorly conditioned region of the optimization landscape, multiple-secant methods show clear advantages. 
Finally, while we did not explore stochastic optimization in this work, existing research \citep{berahas2022limited} suggests that such extensions are feasible, and an interesting area of future study.

\clearpage
\newpage

\bibliographystyle{plainnat} 
\bibliography{refs} 



\appendix  
\section*{Appendix}



\section{Proofs for Theorem \ref{th:superlinear}}
\label{app:superlinear_conv}

 \subsection{Linear algebra facts}

\begin{lemma}\label{lem:multisecant_asymmproj} Take $U,V\in \R^{n\times p}$,
    as long as $U^\top V$ is invertible, and $p<n$, $\|I-U(V^\top U)^{-1}V^\top \|_2 =   \|U(V^\top U)^{-1}V^\top \|_2$.
\end{lemma}  

 \begin{proof}
First, we form $Q = I-U(V^\top U)^{-1}V^\top $. Here, we show that 
\[
QQ^\top  = I - U(V^\top U)^{-1}V^\top  - V(U^\top V)^{-1}U^\top  + U(V^\top U)^{-1}V^\top V(U^\top V)^{-1}U^\top 
\]
and we can derive
\[
QQ^\top V = V - U(V^\top U)^{-1}V^\top V - V + U(V^\top U)^{-1}V^\top V  = 0
\]
where $V$ is in the nullspace of $Q^\top$. So, if $x$ is a nontrivial eigenvector of $QQ^\top $, then it is a nontrivial eigenvector of 
\[
(I-V V^\dagger)QQ^\top (I-VV^\dagger) = (I-V V^\dagger) + (I-V V^\dagger)U(V^\top U)^{-1}V^\top V(U^\top V)^{-1}U^\top (I-V V^\dagger)
\]
or for some eigenvector $x$ in the nullspace of $V^\top $,
\begin{eqnarray*}
&&\max_{x:x^\top x=1, V^\top x = 0} x^\top (I-V V^\dagger)QQ^\top (I-VV^\dagger)x \\
&=& \underbrace{x^\top x}_{=1} - \underbrace{x^\top V V^\dagger x}_{= 0} + x^\top (I-V V^\dagger)U(V^\top U)^{-1}V^\top V(U^\top V)^{-1}U^\top (I-V V^\dagger)x\\
&=& 1 + x^\top  U(V^\top U)^{-1}V^\top V(U^\top V)^{-1}U^\top  x\\
&=& 1+\|V(U^\top V)^{-1}U^\top x\|_2^2\\
&=&  x^\top (I +  U(V^\top U)^{-1}V^\top V(U^\top V)^{-1}U^\top )x .
\end{eqnarray*}
Now, for $P = I-VV^\dagger$, the goal is to find the maximum eigenvalue of \\
$I+PU(V^\top U)^{-1}V^\top V(U^\top V)^{-1}UP^\top $.
Define $S = (V^\top U)(V^\top V)^{-1}(U^\top V)$ and 
\begin{align*}
\det((\lambda-1) I &  - PUS^{-1}U^\top P^\top ) \\
&= \frac{\det((\lambda -1)I)}{\det(S)}\det(S-(\lambda-1)^{-1} U^\top P^\top PU) \\
&= \frac{(\lambda-1)^{n-p}}{\det(S)}\det((V^\top U)(V^\top V)^{-1}(U^\top V)-(\lambda-1)^{-1} U^\top P^\top PU)\\
&= \frac{(\lambda-1)^{n-p-1}}{\det(S)}\det((\lambda-1)(V^\top U)(V^\top V)^{-1}(U^\top V)- U^\top P^\top PU).
\end{align*}
Since $P = I-VV^\dagger = I-V(V^\top V)^{-1}V^\top $, we compute
\[
PU = U - V(V^\top V)^{-1}V^\top U  
\]
\[
U^\top P^\top PU = U^\top U- U^\top  V(V^\top V)^{-1}V^\top U 
\]
and therefore
\begin{align*}
\det&((\lambda-1) I - PUS^{-1}U^\top P^\top ) 
\\
&= \frac{(\lambda-1)^{n-p-1}}{\det(S)}\det((\lambda-1)(V^\top U)(V^\top V)^{-1}(U^\top V)- U^\top U + U^\top V(V^\top V)^{-1}V^\top U)\\
&= \frac{(\lambda-1)^{n-p-1}}{\det(S)\det(U^\top V)^2}\det(\lambda(V^\top V)^{-1}- (U^\top V)^{-1}U^\top U(U^\top V)^{-1}  )\\
&= \frac{(\lambda-1)^{n-p-1}\det(V^\top V)}{\det(S)\det(U^\top V)^2}\det(\lambda- V^\top (U^\top V)^{-1}U^\top U(U^\top V)^{-1}V  ).
\end{align*}
where the zeros are the  eigenvalues of $V^\top (U^\top V)^{-1}U^\top U(U^\top V)^{-1}V $, and thus the largest is \\ $\|U(V^\top U)^{-1}V \|_2^2$.
\end{proof}

\begin{lemma}\label{lem:multisecant_contraction_angle}
     For $U\in \R^{n\times p}$, if $\|U-V\|_F\leq \alpha\|U\|_F$, then for $A = U(V^\top U)^{-1}V^\top $, $p-\|A\|^{-2}\leq \alpha^2$.
\end{lemma}

\begin{proof} The first step yields
\[
\tr((U-V)^\top (U-V)) -\tr(U^\top U )= -\tr(U^\top V)-\tr(V^\top U)+\tr(V^\top V) \leq (\alpha^2-1) \tr(U^\top U ).
\]
Multiplying left and right by $(V^\top U)^{-1}V^\top $ and simplifying gives 
\begin{multline*}
\tr(-2V\underbrace{(V^\top U)^{-1}V^\top }_{B^\top }+V(U^\top V)^{-1}V^\top V(V^\top U)^{-1}V^\top )\\
=\tr(V(U^\top V)^{-1}(-U^\top V -V^\top U+V^\top V)(V^\top U)^{-1}V^\top ) \\\leq (\alpha^2-1) \tr(V(U^\top V)^{-1}U^\top U (V^\top U)^{-1}V^\top ) = (\alpha^2-1)\tr(A^\top A)
\end{multline*}
so
\begin{multline*}
(\alpha^2-1)\tr(A^\top A)\geq \tr(-2VB^\top +VB^\top BV^\top ) \\
= \tr(-2B^\top V+BV^\top VB^\top ) \\
= \| BV^\top -I\|_F^2 - \tr(I_{p\times p})\geq -p.
\end{multline*}  
\end{proof}

\begin{lemma}[Smoothness for vectors]\label{lem:smoothness_vectors}
    If
\[
\|\nabla^2 f(u) - \nabla^2 f(v)\|\leq L \|u-v\|
\]
then
\[
\|\nabla f(u) - \nabla f(v) - \nabla^2 f(u)(u-v)\|\leq \frac{L }{2   } \|u-v\|^2.
\]

\end{lemma}  
\begin{proof}
    Consider the 1-D projection $h_c(u) = c^\top \nabla f(u)$. Then $\nabla h_c(u) = \nabla^2 f(u) c\;$ and
    \[
    \|\nabla h_c(u) - \nabla h_c(v)\| = \|(\nabla^2 f(u) -\nabla^2 f(v))c\| \leq L\|c\|\|u-v\|
    \]
    e.g. $h_c$ is $L\|c\|$-smooth. Therefore, 
    \[
    |h_c(u) - h_c(v) - \nabla h_c(u)^\top (u-v)| \leq \frac{L\|c\|}{2}\|u-v\|^2.
    \]
    Expanding the left hand since, 
    \[
       h_c(u) - h_c(v) - \nabla h_c(u)^\top (u-v) = 
        c^\top (\nabla f(u) - \nabla f(v)) - c^\top \nabla^2 f(u)^\top (u-v).
    \]
    Picking $c = \nabla f(u) - \nabla f(v)) - \nabla^2 f(u)^\top (u-v)$ gives  
     \[
      \|c\| \|\nabla f(u) - \nabla f(v)) - \nabla^2 f(u)^\top (u-v)\| \leq  \frac{L\|c\|}{2}\|u-v\|^2.
    \]
    Canceling out $\|c\|$ completes the proof.
\end{proof}

\subsection{Small lemmas}

\begin{lemma}[Primal dual contraction]
\label{lem:TZ_contraction}
Suppose $S_t,Y_t\in \R^{n\times p}$. Then 
    \[
\|G_t^{1/2}R_t-G_t^{-1/2}Z_t\|  \leq  \frac{p\|G_t^{-1/2}\|\|R_t\|^2 L}{2}. 
\]
\end{lemma}
\begin{proof}
    
For a single secant vector,
\[
z_t = g(\xi_{t+1})-g(\xi_t) = \int_0^1g'(\xi_t+\tau r_t)r_td\tau
\]
\[
\Rightarrow r_t-G_t^{-1}z_t = -G_t^{-1}\int_0^1(g'(\xi_t + \tau r_t) - g'(\xi_t))r_td\tau 
\]
\[
\Rightarrow  \|G_t^{1/2}r_t-G_t^{-1/2}z_t\|_2 \leq  \|G_t^{-1/2}\|  \int_0^1 L \tau \|r_t\|^2 d\tau = \frac{\|G_t^{-1/2}\| L  \|r_t\|^2}{2}.
\]
So if there are $p$ multisecant vectors in $R_t$,  then 
\begin{eqnarray*}
\|G_t^{1/2}R_t-G_t^{-1/2}Z_t\| &=& \sum_{j = t-p+1}^t \sum_{l=j+1}^t \|G_t^{1/2}(\xi_l-\xi_{l-1})-G_t^{-1/2}(g(\xi_l-\xi_{l-1}))\|  \\
&\leq&  \sum_{j = t-p+1}^t \sum_{l=j}^t \frac{\|G_t^{-1/2}\|\|r_l\|^2 L}{2}    \\
&\leq&  \frac{p\|G_t^{-1/2}\|\|R_t\|^2 L}{2}  .
\end{eqnarray*}

\end{proof}

\begin{lemma}[Inverse estimate local proximity]\label{lem:estimate_local}
Suppose that  $\|\xi_0-\xi_t\| < \tau /L$.
 Then
$\|G_t^{-1}\|  \leq \frac{1}{1-\tau}$.

\end{lemma}
\begin{proof}

Since 
\[
\|I-G_t\| = \|G_0-G_t\| \leq L\|\xi_0-\xi_t\| <\tau.
\]
Then
\[
\|G_t^{-1}\| \leq \frac{1}{1-\|I-G_t\|} \leq \frac{1}{1-L\|\xi_0-\xi_t\|} \leq \frac{1}{1-\tau}.
\]
\end{proof}

\begin{lemma}[Bound on $C$]\label{lem:Cbound}
If $\|G^{-1}\|\leq \gamma_1$,   $\|I-G\|\leq \gamma_2$, and $\|G^{-1}-C^{-1}\|\leq \gamma_3$, and $\gamma_1\gamma_2+\gamma_3 < 1\;$, 
 then $\|C\| \leq \frac{1}{1-\gamma_1\gamma_2-\gamma_3}$

\end{lemma}
\begin{proof}
From the following inequality,
\[
\|I-C^{-1}\| \leq \|I-G^{-1}\| + \|G^{-1}-C^{-1}\| \leq \|G^{-1}\| \|I-G\| + \|G^{-1}-C^{-1}\| \leq \gamma_1\gamma_2+\gamma_3 =: \gamma,
\]
if $\lambda_i$ is an eigenvalue of $C$, then 
\[
\max_i |1-\lambda^{-1}_i| \leq \gamma \Rightarrow \lambda_i^{-1}\in (1-\gamma,1+\gamma) \Rightarrow  \lambda_i \leq \frac{1}{1-\gamma}.
\]
\end{proof}

\subsection{Linear and superlinear convergence proofs}

\subsubsection{Setup}
We now consider the convergence proof for the symmetrized multisecant BFGS method with diagonal perturbation, e.g.  
    \begin{eqnarray*}
\hat B_{\mathrm{next}}^{-1}  &=&  
 (I-S(Y^\top S)^{-1}Y^\top ) B^{-1} (I-Y(S^\top Y)^{-1}S^\top ) + \frac{1}{2}S ((S^\top Y)^{-1} +(Y^\top S)^{-1})S^\top 
     \end{eqnarray*}
and
\[
x_{t+1} = x_t - (\hat B_t^{-1}+ \mu_t I) \nabla f(x_t).
\]

     \subsubsection{Scaling}
 We assume we start at some $x_0$ suitably close to $x^*$. 
Define $F_0 = \nabla^2 f(x_0)$
We then analyze the method, scaled by $F_0$.
Specifically, we define
\[
\xi_t = F_0^{1/2}x_t, \quad g(\xi) = F_0^{-1/2}\nabla f(F_0^{-1/2}\xi).
\]
Then 
\[
g(\xi_t) = F_0^{-1/2} \nabla f(x_t), \quad g(\xi^*) = 0 \iff \nabla f(F_0^{-1/2}\xi^*) = 0.
\]
Define $r_t = F_0^{1/2}s_t$, $z_t = F_0^{-1/2} y_t$, and $C_t = F_0^{-1/2} B_t F_0^{-1/2}$. Then 
\[
\xi_{t+1} = \xi_t + r_t, \quad z_t = g(\xi_{t+1})-g(\xi_t).
\]
By similar token, $R_t = F_0^{1/2}S_t$, $Z_t = F_0^{-1/2}Y_t$. We also include $G_t = g'(\xi_t)$,
$R_t = F_0 S_t$, $Z_t = F_0^{-1} Y_t$.
 Now to generalize our analysis to the three different multisecant methods, we consider two constructions of $C_t = F_0^{-1}B_tF_0^{-1}$, the scaled Hessian approximation: the asymmetric  version: 
\[
\tilde C_{t+1}^{-1} 
      =
 (I- R_t(Z_t^\top  R_t)^{-1}Z_t^\top ) C_t^{-1} (I-Z_t(R_t^\top Z_t)^{-1}R_t^\top ) +  R_t (R_t^\top Z_t)^{-1} R_t^\top   + \mu_t F_0
  \]
  and the symmetrized version
\begin{multline*}
\hat C_{t+1}^{-1} 
      =
 (I- R_t(Z_t^\top  R_t)^{-1}Z_t^\top ) C_t^{-1} (I-Z_t(R_t^\top Z_t)^{-1}R_t^\top )\\ + \frac{1}{2}( R_t (R_t^\top Z_t)^{-1} R_t^\top  + R_t(Z_t^\top R_t)^{-1} R_t^\top )  + \mu_t F_0
  \end{multline*}
  where $C_{t+1} = \tilde C_{t+1}^{-1}$ is the unsymmetrized update and $C_{t+1} = \hat C_{t+1}^{-1}$  is the symmetrized update. Taking $\mu_t = 0$ considers no diagonal perturbation, and $\mu_t>0$ with diagonal perturbation.

\subsubsection{Contraction steps}
\label{app:onestep_contract}

We use the above variable assignments, with $G = G_t$, $Z=Z_t$, $R=R_t$, $C_{\mathrm{next}} = C_{k+1}$, 
 and  $\|X\|_G = \| G^{1/2}XG^{T/2}\|_F$.
The next two lemmas show that for either the symmetric or asymmetric case, the one-step contraction analysis will eventually yield the same result.
Thus, after this point, we consider $C_{\mathrm{next}}$ to be from either the symmtrized or asymmetric  (vanilla) method.

\begin{lemma}[One step, asymmetric] \label{lem:onestep-nonsymm}
 For the asymmetric update  $C^{-1}_{\mathrm{next}} = \tilde C_{t+1}^{-1}$ 
 we have 
     \begin{multline*}
\|G - C_{\mathrm{next}}^{-1} \|_G - \mu \|F_0\|_G
      \leq  
 \|P(G- C^{-1} ) P^\top \|_G + \|(G^{-1}Z-  R) (R^\top Z)^{-1} R^\top  \|_G 
 \\+   \|R(Z^\top R)^{-1}(Z^\top G^{-1}- R^\top )P \|_G  .
 \end{multline*}
 
\end{lemma}
\begin{proof}
    
         Beginning with 
         \begin{eqnarray*}
C_{\mathrm{next}}^{-1} 
      &=& 
 \underbrace{(I- R(Z^\top  R)^{-1}Z^\top )}_{P} C^{-1} (I-Z(R^\top Z)^{-1}R^\top ) +  R (R^\top Z)^{-1} R^\top  + \mu_t I,
     \end{eqnarray*} 
then    
\begin{eqnarray*}
G^{-1} - C_{\mathrm{next}}^{-1} - \mu I
      &=& 
 P(G^{-1}- C^{-1} ) P^\top  -  R (R^\top Z)^{-1} R^\top  +  G^{-1}(I-P^\top ) + (I-P)G^{-1}P^\top  \\ 
      &=& 
 P(G^{-1}- C^{-1} ) P^\top  -  R (R^\top Z)^{-1} R^\top  \\&&+  G^{-1}Z(R^\top Z)^{-1}R^\top  + R(Z^\top R)^{-1}Z^\top G^{-1}P^\top  \\ 
      &=& 
 P(G^{-1}- C^{-1} ) P^\top  +(G^{-1}Z-  R) (R^\top Z)^{-1} R^\top  \\
&& + R(Z^\top R)^{-1}(G^{-1}Z-R)^\top P^\top  +   R(Z^\top R)^{-1}R^\top P^\top .
     \end{eqnarray*}
     Since 
     \begin{multline*}
     R(Z^\top R)^{-1}R^\top P^\top  =  R(Z^\top R)^{-1}R^\top (I-Z(R^\top Z)^{-1}R^\top ) \\= 
     R(Z^\top R)^{-1}R^\top  - 
     R(Z^\top R)^{-1}R^\top Z(R^\top Z)^{-1}R^\top  = 0 ,
     \end{multline*}
we use triangle inequality of the $\|\cdot\|_G$ to complete the proof. 
\end{proof}

\begin{lemma}[One step, symmetric  ] \label{lem:onestep-symm}
 For the symmetric update 
 $C^{-1}_{\mathrm{next}} = \hat C_{t+1}^{-1}$, we have 
     \begin{multline*}
\|G - C_{\mathrm{next}}^{-1} \|_G - \mu \|F_0\|_G
      \leq
 \|P(G- C^{-1} ) P^\top \|_G + \|(G^{-1}Z-  R) (R^\top Z)^{-1} R^\top  \|_G 
 \\+   \|R(Z^\top R)^{-1}(Z^\top G^{-1}- R^\top )P \|_G  .
 \end{multline*}
 
\end{lemma}

     \begin{proof}
         Beginning with 
         \begin{multline*}
C_{\mathrm{next}}^{-1} 
      =
 \underbrace{(I- R(Z^\top  R)^{-1}Z^\top )}_{P} C^{-1} (I-Z(R^\top Z)^{-1}R^\top ) + \\ \frac{1}{2}( R ((R^\top Z)^{-1}+(Z^\top R)^{-1}) R^\top ) + \mu_t ,
     \end{multline*} 
then  
\begin{eqnarray*}
&&G^{-1} - C_{\mathrm{next}}^{-1} -P(G^{-1}- C^{-1} ) P^\top  -\mu I \\
      &=& 
 - \frac{1}{2} R (R^\top Z)^{-1} R^\top - \frac{1}{2}R (Z^\top R)^{-1} R^\top  +  G^{-1}(I-P^\top ) + (I-P)G^{-1} \\
 && \qquad -  (I-P)G^{-1}(I-P)^\top   \\ 
      &=& 
 - \frac{1}{2} R (R^\top Z)^{-1} R^\top  -\frac{1}{2} R (Z^\top R)^{-1} R^\top   +  G^{-1}Z(R^\top Z)^{-1}R^\top  + R(Z^\top R)^{-1}Z^\top G^{-1}\\
 && \qquad  - (I-P)G^{-1}(I-P)^\top   \\ 
      &=& 
\frac{1}{2}(G^{-1}Z-  R) (R^\top Z)^{-1} R^\top  
+
\frac{1}{2}R(Z^\top R)^{-1} (Z^\top G^{-1}-  R^\top )   + \frac{1}{2} (I-P)G^{-1} \\
 && \qquad + \frac{1}{2}G^{-1}(I-P^\top ) - (I-P)G^{-1}(I-P) \\
&=& 
\frac{1}{2}(G^{-1}Z-  R) (R^\top Z)^{-1} R^\top  
+
\frac{1}{2}R(Z^\top R)^{-1} (Z^\top G^{-1}-  R^\top ) \\
 && \qquad   + \frac{1}{2} (I-P)G^{-1}P + \frac{1}{2}PG^{-1}(I-P^\top ) \\
&=& 
\frac{1}{2}(G^{-1}Z-  R) (R^\top Z)^{-1} R^\top  
+
\frac{1}{2}R(Z^\top R)^{-1} (Z^\top G^{-1}-  R^\top )   + \frac{1}{2} R(Z^\top R)^{-1}Z^\top G^{-1}P \\
 && \qquad + \frac{1}{2}PG^{-1} Z(R^\top Z)^{-1}R^\top  \\
&=& 
\frac{1}{2}(G^{-1}Z-  R) (R^\top Z)^{-1} R^\top  
+
\frac{1}{2}R(Z^\top R)^{-1} (Z^\top G^{-1}-  R^\top )  \\
&& \qquad  + \frac{1}{2} R(Z^\top R)^{-1}(Z^\top G^{-1}- R^\top )P + \frac{1}{2}P(G^{-1} Z-R)(R^\top Z)^{-1}R^\top \\
&& \qquad  + \underbrace{\frac{1}{2} R(Z^\top R)^{-1} R^\top P + \frac{1}{2}PR(R^\top Z)^{-1}R^\top }_{(**)}
     \end{eqnarray*}
     where 
     \begin{eqnarray*}
     (**)& =& R(Z^\top R)^{-1} R^\top  - R(Z^\top R)^{-1} R^\top  Z(R^\top Z)^{-1}R^\top  +   R(R^\top Z)^{-1}R^\top \\
&& \qquad - R(Z^\top R)^{-1}Z^\top  R(R^\top Z)^{-1}R^\top  \\
     &=& 
     R(Z^\top R)^{-1} R^\top  - R(Z^\top R)^{-1} R^\top  +   R(R^\top Z)^{-1}R^\top  - R(R^\top Z)^{-1}R^\top  \\
     &=& 0 .
     \end{eqnarray*}
We use triangle inequality of the $\|\cdot\|_G$ to complete the proof.

\end{proof}

\subsubsection{Main contraction lemma}
\label{app:main_contract_lemma}
The next three lemmas can be read as one lemma, whose point is to show the one-step contraction of $\|G^{-1}_t-C_{t}^{-1}\|$. However, we break it up into three parts for improved readability.

\begin{lemma}[Contraction, part 1]\label{lem:onestep-stage1} 
Using the above variable assignments, with $G = G_t$, $Z=Z_t$, $R=R_t$, $C_{\mathrm{next}} = C_{t+1}$, then for the symmetric or asymmetric update, we derive
 \[
 \|G^{-1} - C_{\mathrm{next}}^{-1} \|_G \leq   \frac{1}{\omega^2}\|G^{-1}-C^{-1} \|_G 
 + \frac{2}{\omega^2} \frac{\|G^{-1/2}Z-  G^{1/2}R\|_2}{\|G^{-1/2} Z\|}  + \mu_t \|F_0\|_G
\]    
where
\[
\omega = \frac{1}{\|G^{1/2}R(Z^\top R)^{-1}Z^\top G^{-1/2}\|}
\]
\end{lemma}

     \begin{proof}
         Beginning with Lemmas \ref{lem:onestep-nonsymm} and \ref{lem:onestep-symm}, we have
     \begin{multline*}
\|G - C_{\mathrm{next}}^{-1} \|_G - \mu \|F_0\|_G
      \leq
 \|P(G- C^{-1} ) P^\top \|_G + \|(G^{-1}Z-  R) (R^\top Z)^{-1} R^\top  \|_G 
 \\+   \|R(Z^\top R)^{-1}(Z^\top G^{-1}- R^\top )P \|_G  .
 \end{multline*}
Next, 
\begin{eqnarray*}
\|(G^{-1}Z-  R) (R^\top Z)^{-1} R^\top \|_G &=& 
\|(G^{-1/2}Z-  G^{1/2}R)(R^\top Z)^{-1} R^\top G^{1/2}\|_F\\
&\leq & \|(G^{-1/2}Z- G^{1/2}R)\|_2 \|(R^\top Z)^{-1} R^\top G^{1/2}\|_F\\
\end{eqnarray*}

\begin{align*}
\|R(Z^\top R)^{-1} & (G^{-1}Z-R)^\top P^\top \|_G\\
&\leq  \|(G^{-1/2}Z- G^{1/2}R)\|_2 \|(R^\top Z)^{-1} R^\top G^{1/2}\|_F \underbrace{\|G^{1/2}PG^{-1/2}\|}_{1/\omega_1}\\
&\leq  \frac{\|(G^{-1/2}Z- G^{1/2}R)\|_2 }{\|G^{-1/2} Z\|}\underbrace{\|G^{-1/2}Z (R^\top Z)^{-1} R^\top G^{1/2}\|_F }_{=:1/\omega_2 }\underbrace{\|G^{1/2}PG^{-1/2}\|}_{=:1/\omega_1}.
\end{align*}
Here, we define
\[
\frac{1}{\omega_1} = \|G^{1/2}PG^{-1/2}\|,
\qquad 
\frac{1}{\omega_2} = \|G^{-1/2}Z (R^\top Z)^{-1}R^\top G^{1/2}\| .
\]
Specifically, expanding, 
\[
\frac{1}{\omega_1} = \|I-G^{1/2}R(Z^\top R)^{-1}Z^\top G^{-1/2}\|
\]
and taking $U = G^{1/2}R$ and $V = G^{-1/2}Z$, using Lemma \ref{lem:multisecant_asymmproj}, we get that 
\[
\frac{1}{\omega_1} = \|G^{1/2}R(Z^\top R)^{-1}Z^\top G^{-1/2}\| = \frac{1}{\omega_2}=:\frac{1}{\omega}.
\]

\end{proof}

\begin{lemma}[Contraction, part 2]\label{lem:onestep-stage2} 
Take $G = G_t$, $Z=Z_t$, $R=R_t$, $C_{\mathrm{next}} = C_{k+1}$.
Suppose   that 
\begin{equation}
 \|R\| \leq \frac{m}{M\sqrt{L}}.
 \label{eq:T_lim}
 \end{equation}
For $c_1 = \frac{LM^2}{2p m^2}$, $c_2 =  \frac{3}{2}\frac{ LM^2  }{ m^2}$, we have
\[
 \|G^{-1} - C_{\mathrm{next}}^{-1} \|_G 
 \leq    (\frac{1}{p} + c_1\|R\|^2)\|G^{-1}-C^{-1} \|_G +  c_2\|R\|+\mu \|F_0\|_G.
\]

\end{lemma}

\begin{proof}
    Using Taylor interpolation, we had previously shown that for each $r = \xi_1-\xi_2$, there exists some $\tilde \xi$ where $r = g'(\tilde \xi)^{-1}(g(\xi_1)-g(\xi_2))$. Since $mI\preceq g'(\xi)\preceq MI$, then $\|Z\| \geq \frac{m}{M} \|R\|$, and so 
\begin{equation}
\frac{1}{\|G^{-1/2}Z\|} \leq \frac{M}{m}\frac{\|G^{1/2}\|}{\|R\|}.
\label{eq:inv_gz}
\end{equation}
So including Lemma \ref{lem:TZ_contraction}
 \begin{eqnarray*}
 \|G^{-1} - C_{\mathrm{next}}^{-1} \|_G-\mu \|F_0\|_G &\leq&  \frac{1}{\omega^2}\|G^{-1}-C^{-1} \|_G + \frac{2}{\omega^2} \frac{\|G^{-1/2}Z-  G^{1/2}R\|_2}{\|G^{-1/2} Z\|}\\
 &\overset{l. \ref{lem:TZ_contraction},\eqref{eq:inv_gz}}{\leq}&   \frac{1}{\omega^2}\|G^{-1}-C^{-1} \|_G +\frac{2}{\omega^2  } 
 \frac{M \|G^{1/2}\|}{m\|R\|}
 \frac{p\|G^{-1/2}\|\|R\|^2 L}{2}  \\
&  \leq&   \frac{1}{\omega^2}\|G^{-1}-C^{-1} \|_G +\frac{ 1 }{\omega^2 } \frac{ pLM^2 \|R\|}{ m^2}.
\end{eqnarray*}
Next, taking $U = G^{-1/2}Z$ and $V=G^{1/2} R$ and invoking Lemma \ref{lem:multisecant_contraction_angle},
\[
p-\omega^2 \leq \frac{\|G^{1/2}R-G^{-1/2}Z\|^2_F}{\|G^{-1/2}Z\|_F^2}\leq \frac{pLM^2}{2 m^2}\|R\|^2
\]
and
\[
\omega \geq p(1-\frac{LM^2}{2m^2}\|R\|^2) \overset{\eqref{eq:T_lim}}{\geq }
 p(1-\frac{1}{2}) 
\]
so 
\[
\frac{p}{\omega^2} = 1+\frac{p-\omega^2}{\omega^2}\leq 1 + \frac{LM^2}{2 m^2}\|R\|^2
\] 

\begin{align*}
 \|G^{-1} & - C_{\mathrm{next}}^{-1} \|_G-\mu \|F_0\|_G\\
&  \leq   \frac{1}{p}(1 + \frac{LM^2}{2 m^2}\|R\|^2)\|G^{-1}-C^{-1} \|_G +\frac{ 1 }{p} (1 + \frac{LM^2}{2 m^2}\|R\|^2)\frac{ pLM^2 \|R\|}{ m^2}\\
&\overset{\eqref{eq:T_lim}}{  \leq}   (\frac{1}{p} + \frac{LM^2}{2p m^2}\|R\|^2)\|G^{-1}-C^{-1} \|_G +  \frac{3}{2}\frac{ LM^2 \|R\|}{ m^2}.
\end{align*}

\end{proof}

\begin{lemma}[Contraction, part 3]\label{lem:onestep-stage3} 
In addition to the previously listed assumptions, suppose  that 
\begin{itemize}
    \item $\|\xi_0-\xi_t\|\leq \tau/L$, which implies  $\|G^{-1}_t\|\leq \frac{1}{1-\tau}$ by Lemma \ref{lem:estimate_local} 
    \item $
 \|R_t\| \leq \frac{m}{M\sqrt{L}}.
$

\end{itemize}
Take
\[c_1 = \frac{LM^2}{2q m^2}, \qquad c_2 =  \frac{3}{2}\frac{ LM^2  }{ m^2},\qquad 
c_3 =  \frac{L}{2} +  \frac{c_1m}{M\sqrt{L}} +  \frac{L}{2}  \frac{c_1m^2}{M^2L},
\]
\[
c_4  =
(\frac{m\sqrt{L}}{2M}+1)  c_2 +    \frac{m  \sqrt{L}  (m M^{-1} \sqrt{L} \bar G +1)}{2 M}, \qquad c_5 = (\frac{m\sqrt{L}}{2M}+1) M\|F_0\|_F.
\]
Then 
 \begin{eqnarray*}
\|G_{t+1}^{-1}-C_{t+1}^{-1}\|_{t+1} - \|G_t^{-1}-C_t^{-1}\|_t 
&\leq& 
 c_3\|R_t\|
 \|G_t^{-1}-C_t^{-1}\|_t +c_4\|R_t\|+\mu_t c_5 \\ 
\end{eqnarray*}

\end{lemma}

\begin{proof}
Since 
\[
\|G_{t+1}G_t^{-1}\| = \|(G_{t+1}-G_t)G_t^{-1} + I \| \leq \underbrace{\|G_{t+1}-G_t\|}_{\leq L\|r_t\|}\underbrace{\|G_t^{-1}\|}_{\leq \bar G} + 1,
\]
then
\[
\|X\|_{t+1} \leq \|G_{t+1}G_t^{-1}\|\|X\|_t \leq (L\bar G \|r_t\| +1)\|X\|_t.
\]
Moreover, 
 \begin{eqnarray*}
\|G^{-1}_{t+1}-G^{-1}_t\|_{t+1} &=& \|G^{1/2}_{t+1} G^{-1}_{t+1} (G_{t+1}-G_t) G^{-1}_{t} G^{1/2}_{t+1}\|  \\
&\leq& \underbrace{\|G^{-1/2}_{t+1}  \|}_{\leq 1/\sqrt{2}} \underbrace{\|G_{t+1}-G_t\|}_{\leq L\|r_t\|}  \underbrace{\| G^{-1}_{t} G_{t+1}\| }_{L\bar G \|r_t\|+1}\underbrace{\|G^{-1/2}_{t+1}  \|}_{\leq 1/\sqrt{2}}
\end{eqnarray*}
Therefore, since
\[
\|G^{-1}_{t+1}-C_{t+1}^{-1}\|_{t+1} - 
\|G^{-1}_{t}-C_{t+1}^{-1}\|_{t+1} \leq \|G^{-1}_{t+1}- G^{-1}_{t}\|_{t+1}, 
\]
 then
 \begin{eqnarray*}
\|G_{t+1}^{-1}-C_{t+1}^{-1}\|_{t+1}
&\leq& 
\|G_{t+1}G_t^{-1}\| \|G_{t}^{-1}-C_{t+1}^{-1}\|_{t} +  \|G^{-1}_{t+1}- G^{-1}_{t}\|_{t+1} \\
&\leq& 
(\frac{L\|r_t\|}{2}+1) \|G_{t}^{-1}-C_{t+1}^{-1}\|_{t} +  \frac{L\|r_t\|(L \bar G\|r_t\|+1)}{2}\\ 
&\leq& 
(\frac{L\|r_t\|}{2}+1) (\frac{1}{p} + c_1\|R_t\|^2)\|G_t^{-1}-C_t^{-1}\|_t + (\frac{L\|r_t\|}{2}+1)  c_2\|R_t\| \\
&&+\mu_t(\frac{L\|r_t\|}{2}+1) \|F_0\|_t +    \frac{L\|r_t\|(L \bar G \|r_t\|+1)}{2}.
\end{eqnarray*}
Since
\[
(\frac{L\|r_t\|}{2}+1) (\frac{1}{p} + c_1\|R_t\|^2) - 1 = \underbrace{\frac{1}{p} - 1}_{<0} + \left(\frac{L}{2} +  c_1\|R_t\| +  \frac{L}{2}  c_1\|R_t\|^2\right)\|R_t\|
\]
and $\|F_0\|_t \leq \|G_t\|_2 \|F_0\|_F \leq M\|F_0\|_F$,
then
 \begin{eqnarray*}
\|G_{t+1}^{-1}-C_{t+1}^{-1}\|_{t+1} - \|G_t^{-1}-C_t^{-1}\|_t 
&\leq& 
 c_3\|R_t\|
 \|G_t^{-1}-C_t^{-1}\|_t +c_4\|R_t\| +\mu_tc_5  \\ 
\end{eqnarray*}
where 
\[
\frac{L}{2} +  c_1\|R_t\| +  \frac{L}{2}  c_1\|R_t\|^2 \leq   \frac{L}{2} +  \frac{c_1m}{M\sqrt{L}} +  \frac{L}{2}  \frac{c_1m^2}{M^2L}   =:c_3 , 
\]
\[
(\frac{L\|r_t\|}{2}+1)  c_2 +  \frac{L\|r_t\|(L \bar G\|r_t\|+1)}{2}  \leq
(\frac{m\sqrt{L}}{2M}+1)  c_2 +    \frac{m  \sqrt{L}  (m M^{-1} \sqrt{L} \bar G +1)}{2 M}
=:c_4 
\]
\[
(\frac{L\|r_t\|}{2}+1)  M\|F_0\|_F\leq (\frac{m\sqrt{L}}{2M}+1) M\|F_0\|_F=: c_5 .
\]
\end{proof}

\subsubsection{Linear convergence}
\label{app:linconv}

\begin{lemma}[Linear convergence main steps]
\label{lem:linearconv}
Given assumptions \eqref{eq:hess_cond} and \eqref{eq:lipschitz_hessian},  suppose also that, at initialization, 
    \begin{equation}
    \sum_{i=1}^{p} \|\xi_i-\xi^*\| \leq \frac{\delta }{\max\{M,2\}} 
    \label{eq:local_init}
    \end{equation}
    and
\[c_1 = \frac{LM^2}{2p m^2}, \quad c_2 =  \frac{3}{2}\frac{ LM^2  }{ m^2},\quad 
c_3 =  \frac{L}{2} +  \frac{c_1m}{M\sqrt{L}} +  \frac{L}{2}  \frac{c_1m^2}{M^2L},
\]
\[
c_4  =
(\frac{m\sqrt{L}}{2M}+1)  c_2 +    \frac{m  \sqrt{L}  (m M^{-1} \sqrt{L} \bar G +1)}{2 M}, \quad c_5 = (\frac{m\sqrt{L}}{2M}+1) M\|F_0\|_F,
\]
\[
\beta = p\delta (\gamma+\frac{1}{2}), \qquad
\gamma = \rho^{-p}\min \{  \frac{1}{8M},\frac{1}{4},   \frac{ p^2m c_4^2(1-\rho)}{2ML^{3/2}}    \}  
\]
\[
\delta =\frac{1}{\gamma+\frac{1}{2}}\min\{ \frac{m}{Mp \sqrt{L}}, \frac{1-\rho}{2p L}, \frac{\rho^p(1-\rho)}{p (c_3+\gamma^{-1}c_4)},\frac{M(\gamma+\frac{1}{2})}{4L}\}.
\]
Assume also that $\mu_t$ is a decaying sequence, such that 
\[
\sum_{t=0}^\infty\mu_t \leq \bar\epsilon :=\min\{1/4,1/(8M)\}.
\]
Then, for all $t>p$,
\begin{enumerate} 

\item $\|g(\xi_t)\| \leq \delta \rho^{t-p}$

\item $\|C_t^{-1}\|    \leq  \gamma  + \frac{1}{2}$
    \item $\|\xi_0-\xi_t\| \leq  \frac{1}{2L}$

    \item $\|R_t\|\leq  \beta \rho^{t-p} \leq \frac{m}{M\sqrt{L}}$

    \item $\|G_{t}^{-1} - C_{t}^{-1}\|_{t}  \leq \gamma (1-\rho^{t}) +\epsilon_t$.

\end{enumerate}

\end{lemma}

\begin{proof} 
First, note that
  \[
  \beta \rho^{t-p} \leq \beta = p\delta(\gamma+1/2)\leq \frac{m}{M\sqrt{L}}
  \]
  by using $\delta \leq \frac{m}{Mp(\gamma+1/2)\sqrt{L}}$.
Now to prove the rest, inductively.\\

\textbf{Base case.} At $t \leq p$,
\begin{enumerate}
    \item Since $g(\xi^*) = 0$,  the initial assumption \eqref{eq:local_init} implies 
    \[
    \|g(\xi_t)\| = \|g(\xi_t)-g(\xi^*)\| \leq M\|\xi_t-\xi^*\| \leq \delta.
    \]
    \item This actually results from 3,5 (base case), following the same logic as in the inductive step.

    \item Since 
    \[
    \|\xi_0-\xi_t\|\leq \sum_{i=1}^p \|\xi_i - \xi_{i-1}\| \leq \sum_{i=1}^p \|\xi_i-\xi^*\| + \|\xi_{i-1}-\xi^*\| \leq \frac{2\delta}{M} \leq \frac{1}{2L}  .
    \]

    \item Since $\gamma \leq 1/4$,
    \begin{multline*}
    \|R_t\| \leq  \sum_{i=t-p+1}^t \sum_{j=t+1}^p \|\xi_j-\xi_{j-1}\|\\ \leq  \sum_{i=t-p+1}^t \sum_{j=t+1}^p \|\xi_j-\xi^*\| + \|\xi_{j-1}-\xi^*\| \leq \frac{2q\delta}{2} = \beta (\gamma+1/2) < \beta.
    \end{multline*}

    \item 
From Lemma \ref{lem:onestep-stage3},  
    \begin{multline*}
\|G_{t+1}^{-1} - C_{t+1}^{-1}\|_{t}-\bar\epsilon \leq (c_3 
    \gamma  +c_4)\sum_{i=0}^{t} \|R_i\| \leq (c_3\gamma+c_4) \beta \rho^{-p}\frac{1-\rho^t}{1-\rho}\\ \leq (c_3\gamma+c_4)\frac{p\delta(\gamma+1/2)}{1-\rho}\rho^{-p} \leq \gamma (1-\rho^t)
 \end{multline*}
    by using $\delta \leq  \frac{1-\rho}{p(\gamma+1/2)(c_3+c_4\gamma^{-1})}$.\\
\end{enumerate}

\textbf{Inductive proof.} Assume 1,2,3,4,5 are true at iteration $t$.
\begin{itemize} 

\item (3,5$_t$ $\to$ 1$_{t+1}$) By multisecant condition, $C_t R_t = Z_t$. So, it's also true that $C_t r_t = z_t$.

From 3, we have that $\|G_t^{-1}\| \leq \frac{1}{2}$, and $\|I-G_t\| = \|G_0-G_t\| \leq L\|\xi_0-\xi_t\| \leq 1/2$. Therefore, from 5,  and using Lemma \ref{lem:Cbound}, then we have that 
\[
\|I-C_t^{-1}\| \leq 1/4 + \gamma +\bar\epsilon =:c_7.
\]
Since $\gamma\leq 1/4$, $\bar \epsilon < 1/4$, then $c_7 < 1$. Therefore, $\|C_t\| \leq \frac{1}{1-c_7}$. 
So
\[
\|G_t-C_t\| \leq \underbrace{\|G_t\|}_{\leq M} \underbrace{\|G_t^{-1}-C_t^{-1}\|}_{ \gamma +\bar\epsilon}\underbrace{\|C_t\|}_{1/(1-c_7) } \leq M\gamma/(1-c_7)
\]
and thus 
\begin{eqnarray*}
\|g(\xi_{t+1})\| &\leq& \|g(\xi_{t+1})-g(\xi_t)-G_tr_t  + G_tr_t - C_tr_t\|\\
&\leq& \underbrace{\|g(\xi_{t+1})-g(\xi_t)-G_tr_t\|}_{(L/2) \|r_t\|^2}  + \underbrace{\|(G_t - C_t)r_t\|}_{\leq M (\gamma+\bar\epsilon) \|r_t\| /(1-c_7)}\\
&\leq& (\frac{L \|r_t\|_2}{2} + \frac{M(\gamma+\bar\epsilon)}{1-c_7}) \|r_t\|\\
&\leq& (\frac{L \beta \rho^{t-p}}{2} + \frac{M(\gamma+\bar\epsilon)}{1-c_7}) \beta \rho^{t-p}\\
&\leq& (\frac{L p(\gamma + \frac{1}{2}) \delta  \rho^{t-p}}{2} + \frac{M(\gamma+\bar\epsilon)}{1-c_7})  p(\gamma + \frac{1}{2}) \delta  \rho^{t-p}
\end{eqnarray*}.
Since $\gamma < \min\{ 1/(8M),1/4\}$ and $\delta \leq \frac{1}{2Lp(\gamma+1/2)}$, we derive
\[
\quad 
\frac{L p(\gamma+\frac{1}{2}) \delta }{2} + \frac{M\gamma}{1-c_7} p\bar C = \underbrace{\frac{Lp(\gamma+1/2)\delta}{2}}_{\leq 1/4}+\underbrace{\frac{M(\gamma+\bar\epsilon)}{1-\gamma-\bar \epsilon - 1/4}(\gamma+1/2)}_{\leq  4M(\gamma+\bar\epsilon)(\gamma+1/2) \leq 1/2} \leq 1.
\]

\item (3,4$_t$ $\to$ 5$_{t+1}$) From Lemma \ref{lem:estimate_local}, 3 implies $\|G_t^{-1}\|\leq 1/2$.
From Lemma \ref{lem:onestep-stage3},  
\begin{eqnarray*}
\|G_{t+1}^{-1} - C_{t+1}^{-1}\|_{t} &\leq& (c_3 
    \gamma  +c_4)\sum_{i=0}^{t} \|R_i\|\\
    &\leq& (c_3 
    \gamma  +c_4)\sum_{i=0}^{t} \beta \rho^{t-p} +c_5\sum_{i=0}^t \mu_i\\
    &=& (c_3 
    \gamma  +c_4)\sum_{i=0}^{t} p\delta(\gamma+\frac{1}{2}) \rho^{i-p}+c_5\sum_{i=0}^t \mu_i\\
    &\leq & (c_3 \gamma  +c_4)  p\delta(\gamma+\frac{1}{2})\rho^{-p}\frac{1-\rho^t}{1-\rho}+c_5\sum_{i=0}^t \mu_i\\
    &\leq & (c_3 \gamma  +c_4) \frac{p(1-\rho)\rho^p}{{p(\gamma+1/2)(c_3+\gamma^{-1}c_4)}}(\gamma+\frac{1}{2})\rho^{-p}\frac{1-\rho^t}{1-\rho}+c_5\sum_{i=0}^t \mu_i\\
    &=&  \gamma (1-\rho^t)+\underbrace{c_5\sum_{i=0}^t \mu_i}_{\epsilon_t}.
\end{eqnarray*}

\item (4$_{t}$ $\to$ 3$_{t+1}$) 
    Taking 
    \[
    \|\xi_0-\xi_{t+1}\|\leq \sum_{i=1}^{t+1}\|\xi_i-\xi_{i-1}\|\leq \sum_{i=1}^{t+1}\|R_i\|\leq \sum_{i=1}^{t+1}\beta \rho^{i-p} \leq \frac{\beta}{1-\rho}= \frac{p\delta(\gamma+1/2)}{1-\rho} \leq \frac{1}{2L}
    \]
    using $\delta \leq \frac{1-\rho}{2q(\gamma+1/2)L}$.

    \item (3$_{t+1}$,5$_{t+1}$  $\to$ 2$_{t+1}$)  From Lemma \ref{lem:estimate_local}, 3 implies $\|G_{t+1}^{-1}\|\leq \frac{1}{2}$. Then 
\[
\|C_{t+1}^{-1}\| \leq \|C_{t+1}^{-1}-G_{t+1}^{-1}\| + \|G_{t+1}^{-1}\| \leq  \gamma {+\bar \epsilon} + \frac{1}{2} =:\bar C.
\]

   \item (1$_{t+1}$,2$_{t+1}$ $\to$ 4$_{t+1}$)
\[
\|r_{t+1}\| =  \|-B_{t+1}^{-1} g(\xi_{t+1})\|\leq \|B_{t+1}^{-1}\| \|g(\xi_{t+1})\| \leq  \bar C  \delta \rho^{t+1-p}
\]
 and thus $\|R_{t+1}\|\leq \underbrace{p\bar C  \delta}_{=\beta} \rho^{t+1-p}$.

\end{itemize}

\end{proof}

 \begin{theorem}[Linear convergence] Under the same conditions as Lemma \ref{lem:linearconv}, convergence is  linear. In particular, 
 \[
 \frac{\|\xi_{k+1}-\xi^*\|}{\|\xi_{k}-\xi^*\|} \leq \frac{1}{2}.
 \]
\end{theorem}
\begin{proof}
\begin{multline*}
\sigma_{t+1} 
= \xi_t + r_t-\xi^* 
=\sigma_t -  C_t^{-1} (g(\xi_t)-g(\xi^*))
\\=C_t^{-1}(C_t -G_t)\sigma_t  + C_t^{-1}(G_t \sigma_t -   g(\xi_t)+ g(\xi^*))
\end{multline*}
so given that we previously had
\[
\|C_t-G_t\| \leq \frac{M(\gamma+\bar\epsilon)}{1-(1/4+\gamma+\bar\epsilon)} \overset{\gamma+\bar \epsilon < 1/2}{\leq}  4(\gamma+\bar\epsilon)\overset{\gamma+\bar \epsilon < 1/(4M)}{\leq} 1,
\]
then
\begin{multline*}
\|\sigma_{t+1}\| \leq \underbrace{\|C_t^{-1}\| }_{\leq \gamma\leq 1/4}  \underbrace{\|C_t -G_t\| }_{\leq 1}
\| \sigma_t\|  + \underbrace{\|C_t^{-1}\|}_{\leq 1/4}\underbrace{ \|G_t \sigma_t - g(\xi_t)+g(\xi^*)\|}_{\leq (L/2) \|\sigma\|^2}
\\
\leq\frac{1}{4}\|\sigma_t\| + \frac{1}{4}\|\sigma_t\|
\leq \frac{1}{2}\|\sigma_t\|.
\end{multline*}
\end{proof}

\subsubsection{Superlinear convergence}

Following Sections 6.4 and 6.5 of \citet{nocedal1999numerical}, we now show that linear convergence of the perturbed BFGS method implies q-superlinear convergence. Specifically, we extend the results   to the multisecant diagonally perturbed case.

 \paragraph{Scaling.} We now use a different scaling for the remainder of the proof, which is more traditional. We define  $F_* = \nabla^2 f(x^*)$,  and
 \[
 \tilde B_t = F_*^{-1/2} B_t F_*^{-1/2}, \qquad \tilde S_t = F_*^{1/2}S_t, \qquad \tilde Y_t = F_*^{-1/2}Y_t.
 \]
 So,  
\[
 B_{t+1}   =  B_t + \begin{bmatrix} Y_t & B_tS_t \end{bmatrix} \begin{bmatrix} \tfrac{1}{2}((Y_t^\top S_t)^{-1}+(S_t^\top Y_t)^{-1}) & 0 \\ 0 & -(S_t^\top B_tS_t)^{-1}  \end{bmatrix} \begin{bmatrix} Y_t^\top   \\ S_t^\top B_t\end{bmatrix} + \mu_t I
\]
which implies
 \[
 \tilde B_{t+1}   =   \tilde B_t + \begin{bmatrix}  \tilde Y_t &  \tilde B_t \tilde S_t \end{bmatrix} \begin{bmatrix} \tfrac{1}{2}(( \tilde Y_t^\top  \tilde S_t)^{-1}+( \tilde S_t^\top  \tilde Y_t)^{-1}) & 0 \\ 0 & -( \tilde S_t^\top  \tilde B_t \tilde S_t)^{-1}  \end{bmatrix} \begin{bmatrix}  \tilde Y_t^\top   \\  \tilde S_t^\top  \tilde B_t\end{bmatrix} + \mu_t F_*^{-1}.
\]
Moreover, using Taylor's theorem, for each single vector $s_t$ and $y_t$,
\[
y_t = \int_0^1 \nabla^2 f(x_{t-1}+\tau s_t)s_td\tau\quad  \Rightarrow \quad (y_t-F_0s_t) = (\underbrace{\int_0^1 \nabla^2 f(x_{t-1}+\tau s_t)d\tau}_{=:\bar F_t} -F_0)s_t.
\]
Therefore, given that we know $x_t\to x^*$, we can define $\epsilon_{t} :=  \|x_t-x^*\|\to 0$, and by $L$-smoothness of the Hessian, then $\|F_t-F_*\|\leq L \epsilon_{t}$. Thus,
\begin{eqnarray}
\| \tilde Y_t- \tilde S_t\|_2 &\leq& \|F_*^{-1/2}\|_2\|Y_t-F_*S_t\|_2 \nonumber\\
&\leq& \|F_*^{-1/2}\|_2 \|S_t\|_2\|F_t-F^*\|_2 \nonumber\\
&\leq& q^2L\|F_*^{-1/2}\|_2\|S_t\|_2\ \epsilon_{t}
\label{eq:more_than_linear_conv}
\end{eqnarray}
and
\begin{equation}
\| \tilde S_t^\top (\tilde Y_t- \tilde S_t)\|_2 \leq \|S_t^\top (Y_t-F_*S_t)\|_2 \leq\|S_t^\top (F_t-F_*)S_t\|_2  \leq q^2L\|S_t\|^2_2\ \epsilon_{t}.
\label{eq:more_than_linear_conv2}
\end{equation}

 \begin{lemma}[Matrix determinant property]
 \label{lem:matdetprop}
    Suppose that for some $X,Z$, $I+X^\top Z = 0$. Then 
    \[
        \det(I+XZ^\top +UV^\top ) =\det(-X^\top VU^\top Z).
    \]
\end{lemma}
\begin{proof}
    This simply follows from
\begin{eqnarray*}
    \det(I+XZ^\top +UV^\top ) &=& \det\left(I+\begin{bmatrix} X^\top \\U^\top  \end{bmatrix} \begin{bmatrix} Z & V \end{bmatrix}\right)\\
    &=& \det\left(\begin{bmatrix} 0 & X^\top V \\ U^\top Z & I+U^\top V \end{bmatrix} \right)\\
    &=& \det(I+U^\top V)\det(-X^\top V(I+U^\top V)^{-1}U^\top Z) \\
    &=& \det(-X^\top VU^\top Z).
\end{eqnarray*}
\end{proof}

\begin{lemma}[Matrix determinant perturbation]
 \label{lem:matdetperturb}
Assume $B\succ 0$. 
If $0<\epsilon \leq 1/\tr(B^{-1})$, then 
    \[
\det(B + \epsilon I) - \det(B) \leq 
2 \epsilon \cdot \det(B) \cdot \operatorname{tr}(B^{-1}).
\]
\end{lemma}
\begin{proof}

Let \(\lambda_1, \dots, \lambda_n > 0\) be the eigenvalues of \(B\). Then we know
\[
\det(B) = \prod_{i=1}^n \lambda_i,
\quad
\det(B + \epsilon I) = \prod_{i=1}^n (\lambda_i + \epsilon),
\]
and so
\[
\det(B + \epsilon I) - \det(B) = \prod_{i=1}^n (\lambda_i + \epsilon) - \prod_{i=1}^n \lambda_i
= \det(B) \left( \prod_{i=1}^n \left(1 + \frac{\epsilon}{\lambda_i} \right) - 1 \right).
\]
For all $x > 0$, $\log(1 + x) \leq x$. Thus,
\[
\log\left( \prod_{i=1}^n \left(1 + \frac{\epsilon}{\lambda_i} \right) \right)
= \sum_{i=1}^n \log\left(1 + \frac{\epsilon}{\lambda_i} \right)
\leq \sum_{i=1}^n \frac{\epsilon}{\lambda_i}
= \epsilon \cdot \operatorname{tr}(B^{-1}).
\]
Therefore,
\[
\det(B + \epsilon I) - \det(B) \leq \det(B) \cdot \left( \exp(\epsilon \cdot \operatorname{tr}(B^{-1})) - 1 \right)\leq \epsilon \cdot \det(B) \cdot \operatorname{tr}(B^{-1}) \cdot e^{\epsilon \cdot \operatorname{tr}(B^{-1})}
\]
since $\exp(x) - 1 \leq x e^x \; \text{for all }\; x \geq 0$.\\
For any \(\epsilon \in (0, 1 / \operatorname{tr}(B^{-1}))\), since \(e^{x} \leq 1 + 2x\) for small \(x\), we can conclude:
\[
\det(B + \epsilon I) - \det(B)
\leq 2 \epsilon \cdot \det(B) \cdot \operatorname{tr}(B^{-1}).
\]

\end{proof}

 \begin{lemma}[Proxy functions converge]
 \label{lem:proxy_converge}
Assume that $\tilde S_t^\top \tilde S_t$ is invertible and $\tilde B_t$ is positive definite. Additionally, assume that the method is convergent, e.g. $x_t\to x^*$, and the perturbation $\mu_t \leq c_5 \|x_t-x^*\|_2\to 0$. Then for 
\begin{eqnarray*}
    \psi_t&=& \tr(\tilde B_t)-\ln\det(\tilde B_t)\\
    \chi_t &=& q - \tr((\tilde S_t^\top \tilde B_t\tilde S_t)^{-1}(\tilde S_t^\top \tilde B_t\tilde B_t\tilde S_t)) + \ln\det((\tilde S_t^\top \tilde B_t\tilde S_t)^{-1}(\tilde S_t^\top \tilde B_t\tilde B_t\tilde S_t)),
\end{eqnarray*} then 
$\psi_t = \chi_t+\phi_t + \epsilon_t$ where $\epsilon_t\to 0$, which in turn implies $\chi_t\to 0$ and $\phi_t\to 0$.
     
 \end{lemma}
 \begin{proof}
 First, 
 \[
 \tr(\tilde B_{t+1}) = \tr(\tilde B_t) - \tr( \tilde B_t \tilde S_t(\tilde S_t^\top B_t\tilde S_t)^{-1}\tilde S_t^\top \tilde B_t) + \tr(\tilde Y_t(\tilde S_t^\top \tilde Y_t)^{-1}\tilde Y_t^\top ) + \mu_t\tr(F_0)
 \]
 and invoking Lemma \ref{lem:matdetprop} with
\[
X = -\tilde S_t, \quad Z = \tilde B_t\tilde S_t(\tilde S_t^\top \tilde B_t\tilde S_t)^{-1}, \quad U = \tilde B_t^{-1}\tilde Y_t, \quad V = \tilde Y_t(\tilde Y_t^\top \tilde S_t)^{-1}
\]
\begin{eqnarray*}
\det(\tilde B_{t+1}-\mu_tF_0) &=& \det(\tilde B_t)\det(I-\tilde S_t(\tilde S_t^\top \tilde B_t\tilde S_t)^{-1}\tilde S_t^\top \tilde B_t + \tilde B_t^{-1}\tilde Y_t(\tilde Y_t^\top \tilde S_t)^{-1}\tilde Y_t)\\
&=& \det(\tilde B_t)\det\left(I+\begin{bmatrix}\tilde B_t^{-1} \tilde Y_t &  \tilde S_t \end{bmatrix} \begin{bmatrix} (\tilde Y_t^\top \tilde S_t)^{-1} & 0 \\ 0 & -(\tilde S_t^\top \tilde B_t\tilde S_t)^{-1}  \end{bmatrix} \begin{bmatrix} \tilde Y_t^\top   \\ \tilde S_t^\top \tilde B_t\end{bmatrix} \right)\\
&=& \det(\tilde B_t)\det(\tilde S_t^\top \tilde Y_t (\tilde Y_t^\top \tilde S_t)^{-1})\det(\tilde Y_t^\top \tilde S_t(\tilde S_t^\top \tilde B_t\tilde S_t)^{-1})\\
&=& \det(\tilde B_t)\det(\tilde S_t^\top \tilde Y_t )\det((\tilde S_t^\top \tilde B_t\tilde S_t)^{-1}).
\end{eqnarray*} 
Invoking Lemma \ref{lem:matdetperturb}, this implies that for $t$ large enough, 
\[
\det(\tilde B_{t+1}) \leq \det(\tilde B_t)\det(\tilde S_t^\top \tilde Y_t)\det((\tilde S_t^\top \tilde B_t\tilde S_t)^{-1}) + \epsilon_{1,t}
\]
where $\epsilon_{1,t} = O(\mu_t)\to 0$. 
 So  we have 
\begin{eqnarray*}
\psi(\tilde B_{t+1}) - \psi(\tilde \tilde B_t) &=&  -\tr((\tilde S_t^\top \tilde B_t\tilde S_t)^{-1}(\tilde S_t^\top \tilde B_t\tilde B_t\tilde S_t)) + \tr((\tilde S_t^\top \tilde Y_t)^{-1}(\tilde Y_t^\top \tilde Y_t)) \\&&-\ln(\det(\tilde S_t^\top \tilde Y_t )) - \ln(\det((\tilde S_t^\top \tilde B_t\tilde S_t)^{-1})  ).
\end{eqnarray*}
Since $1-t+\ln(t)$ is nonpositive for all $t>0$, then $1-\lambda  + \ln(\lambda)\leq 0$ for each positive eigenvalue $\lambda$ of 
\[
M_t = (\tilde S_t^\top \tilde B_t\tilde S_t)^{-1/2}(\tilde S_t^\top \tilde B_t\tilde B_t\tilde S_t)(\tilde S_t^\top \tilde B_t\tilde S_t)^{-1/2}.
\]
So, summing them all up, 
\[
\chi_t := q -\tr((\tilde S_t^\top \tilde B_t\tilde S_t)^{-1}(\tilde S_t^\top \tilde B_t\tilde B_t\tilde S_t)) + \ln(\det((\tilde S_t^\top \tilde B_t\tilde S_t)^{-1}(\tilde S_t^\top \tilde B_t\tilde B_t\tilde S_t)) \leq 0.
\]
Also, since 
\[
\det((\tilde S_t^\top \tilde B_t\tilde S_t)(\tilde S_t^\top \tilde S_t)^{-1}(\tilde S_t^\top \tilde B_t\tilde S_t)) = \det(\tilde S_t^\top \tilde B_t\tilde S_t^\dagger \tilde S_t^\top \tilde B_t\tilde S_t) \leq \det(\tilde S_t^\top \tilde B_t\tilde B_t\tilde S_t)
\]
by monotonicity of gradient over Lowner partial ordering,
then 
\[
\frac{\det(\tilde S_t^\top \tilde B_t\tilde S_t)^2}{\det(\tilde S_t^\top \tilde B_t\tilde B_t\tilde S_t)\det(\tilde S_t^\top \tilde S_t)}\leq 1
\]
and
\[
\phi_t := \ln(\det(\tilde S_t^\top \tilde B_t\tilde S_t))^2 - \ln(\det(\tilde S_t^\top \tilde B_t\tilde B_t\tilde S_t)) - \ln\det(\tilde S_t^\top \tilde S_t) \leq 0.
\]
So
\begin{eqnarray*}
\psi( \tilde B_{t+1}) - \psi( \tilde B_t) &=&  \chi_t - q - \ln(\det((\tilde S_t^\top \tilde B_t\tilde S_t)^{-1}(\tilde S_t^\top \tilde B_t\tilde B_t\tilde S_t))+ \tr((\tilde S_t^\top \tilde Y_t)^{-1}(\tilde Y_t^\top \tilde Y_t)) \\&&-\ln(\det(\tilde S_t^\top \tilde Y_t )) - \ln(\det((\tilde S_t^\top \tilde B_t\tilde S_t)^{-1})  )\\
&=& \chi_t - q + \tr((\tilde S_t^\top \tilde Y_t)^{-1}(\tilde Y_t^\top \tilde Y_t)) -\ln(\det(\tilde S_t^\top \tilde Y_t )) +\phi_t  + \ln\det(\tilde S_t^\top \tilde S_t).
\end{eqnarray*}
Defining 
\begin{eqnarray*}
\epsilon_{2,t} &=& - q + \tr((\tilde S_t^\top \tilde Y_t)^{-1}(\tilde Y_t^\top \tilde Y_t)) -\ln(\det(\tilde S_t^\top \tilde Y_t )) + \ln\det (\tilde S_t^\top \tilde S_t) + \epsilon_{1,t},
\end{eqnarray*}
then $\psi( \tilde B_{t+1}) - \psi( \tilde B_t) = \chi_t + \phi_t + \epsilon_{1,t}$.

Let us now bound the $\epsilon_{2,t}$ term.
From Property \eqref{eq:more_than_linear_conv2}, we have that   $\|\tilde S_t(\tilde Y_{t}-\tilde S_{t})\|_2 \leq c_0\|\tilde S_t\|^2_2\epsilon_{t}$ for $c_0 = q^2L\|F_*^{-1/2}\|_2 $. 
So invoking Property \eqref{eq:more_than_linear_conv2},
\[
\|(S^\top S)^{-1/2}(S^\top Y)(S^\top S)^{-1/2} - I\| = \|(S^\top S)^{-1/2}(S^\top (Y-S))(S^\top S)^{-1/2}\| \leq c_0\|\tilde S_t\|_2\epsilon_t \to 0
\]
which implies that the eigenvalues of $\frac{1}{2}(S^\top S)^{-1/2}(S^\top Y+Y^\top S)(S^\top S)^{-1/2} $ converge to 1.
So, $\det(\tilde S_t^\top \tilde Y_t) (\tilde S_t^\top \tilde S_t)^{-1}\to 1$ and 
\[
\ln\det(\tilde S_t^\top \tilde Y_t) - \ln\det(\tilde S_t^\top \tilde S_t)\to 0.
\]
Along similar lines,
\begin{eqnarray*}
\tr((\tilde S_t^\top \tilde Y_t)^{-1}\tilde Y_t^\top \tilde Y_t)  &=& \tr((\tilde S_t^\top \tilde Y_t)^{-1}(\tilde S_t^\top \tilde S_t)^{1/2} (\tilde S_t^\top \tilde S_t)^{-1/2}(\tilde Y_t^\top \tilde Y_t) (\tilde S_t^\top \tilde S_t)^{-1/2} (\tilde S_t^\top \tilde S_t)^{1/2})\\
&\leq & \underbrace{\|(\tilde S_t^\top \tilde Y_t)^{-1}(\tilde S_t^\top \tilde S_t)\|_2 }_{\to 1} \cdot \tr( (\tilde Y_t^\top \tilde Y_t) (\tilde S_t^\top \tilde S_t)^{-1})
\end{eqnarray*}
Finally, since 
\[
Y^\top Y(S^\top S)^{-1} = (Y-S)^\top (Y-S)(S^\top S)^{-1} + S^\top (Y-S)(S^\top S)^{-1} + Y^\top S(S^\top S)^{-1}
\]
then for $c_1 = q^4L^2\|F_*^{-1/2}\|_2^2$, $c_2 = q^2L$.
\begin{multline*}
\tr(Y^\top Y(S^\top S)^{-1})= q\underbrace{\|(Y-S)^\top (Y-S)(S^\top S)^{-1}\|}_{\text{Prop \ref{eq:more_than_linear_conv}}: \leq c_1\epsilon^2_t} +\\ q\underbrace{\|S^\top (Y-S)(S^\top S)^{-1}\|}_{\text{Prop \ref{eq:more_than_linear_conv2}}: \leq c_2\epsilon_t} + \underbrace{\tr(Y^\top S(S^\top S)^{-1})}_{\to q}.
\end{multline*}
All of this implies that there exists a constant $c_3$ such that $\epsilon_{2,t}\leq c_3\epsilon_t$. 
Therefore, given that $\psi(\tilde B_{t+1})-\psi(\tilde B_t) \to 0$ and both $\chi_t\geq 0$, $\phi_t\geq 0$, and $\epsilon_t\to 0$, it must be that $\chi_t\to 0$ and $\phi_t\to 0$. 
 \end{proof}

\begin{theorem}[$q$-superlinear conv.]
\label{th:app:superlinear}
Given assumptions \eqref{eq:hess_cond} and \eqref{eq:lipschitz_hessian}, and those 
of Lemma \ref{lem:linearconv}, then 
\[
\frac{\|\tilde B_t \tilde S_t - \tilde S_t\|_F}{\|\tilde S_t\|_F} \to 0
\]
which implies $q$-superlinear convergence.
\end{theorem} 

\begin{proof}
    The property that 
    \[
    \chi_t = \sum_{i=1}^q \underbrace{(1+\tr(\lambda_i) -\ln(\lambda_i))}_{\geq 0} \to 0
    \]
    implies that all the eigenvalues $\lambda_i$ of $(\tilde S_t^\top \tilde B_t\tilde S_t)^{-1}(\tilde S_t^\top \tilde B_t\tilde B_t \tilde S_t)$ converge to 1. This implies that in the limit, $(\tilde S_t^\top \tilde B_t\tilde S_t)^{-1}(\tilde S_t^\top \tilde B_t\tilde B_t \tilde S_t)\to I$, and 
    \begin{equation}
    \lim_{t\to\infty} \tilde S_t^\top \tilde B_t(\tilde B_t-I)\tilde S_t=
    0.
    \label{eq:lemqsuperlinear-helper}
    \end{equation}
    At the same time, $\phi_t\to 0$ implies that 
    \[
    \frac{\det(\tilde S_t^\top \tilde B_t\tilde S_t)^2}{\det(\tilde S_t^\top \tilde B_t\tilde B_t \tilde S_t)\det(\tilde S_t^\top \tilde S_t)}\overset{\eqref{eq:lemqsuperlinear-helper}}{ =} \frac{\det(\tilde S_t^\top \tilde B_t\tilde S_t)}{\det(\tilde S_t^\top \tilde S_t)} \to 1.
    \]
For any two PSD matrices $A$ and $B$, $\tr(AB) \leq \tr(A)\|B\|_2 \leq \tr(A)\tr(B)$, so     
\begin{eqnarray*}
\frac{\|\tilde B_t\tilde S_t - \tilde S_t\|_F^2}{\|\tilde S_t\|^2_F} &=& \frac{\tr((\tilde B_t\tilde S_t-\tilde S_t)^\top (\tilde B_t\tilde S_t-\tilde S_t) )}{\tr(\tilde S_t^\top \tilde S_t)} \leq \tr((\tilde B_t\tilde S_t-\tilde S_t)^\top (\tilde B_t\tilde S_t-\tilde S_t) (\tilde S_t^\top \tilde S_t)^{-1}).
\end{eqnarray*}
Expanding out,
\begin{align*}
\tr((\tilde B_t\tilde S_t-\tilde S_t)^\top & (\tilde B_t\tilde S_t-\tilde S_t) (\tilde S_t^\top \tilde S_t)^{-1}) \\
&= \tr(\tilde S_t^\top \tilde B_t\tilde B_t\tilde S_t (\tilde S_t^\top \tilde S_t)^{-1} - 2\tilde S_t^\top \tilde B_t\tilde S_t(\tilde S_t^\top \tilde S_t)^{-1}  + I)\\
&\to   \tr(I  - \tilde S_t^\top \tilde B_t\tilde S_t (\tilde S_t^\top S)^{-1}) \\
&= q - \tr(U_t^\top \tilde B_tU_t)
\end{align*}
where $U_t = \tilde S_t(\tilde S_t^\top \tilde S_t)^{-1/2}$
Here, $\det(U_t^\top \tilde B_t U_t) = \det(\tilde S_t^\top \tilde B_t \tilde S_t(\tilde S_t^\top \tilde S_t)^{-1}) = 1$, so using AM/GM inequality, 
\[
\tr(U_t^\top \tilde B_tU_t) \geq q\det(U_t^\top \tilde B_tU_t)^{1/q} = q. 
\]
So, 
 \begin{eqnarray*}
\frac{\|\tilde B_t\tilde S_t - \tilde S_t\|_F^2}{\|\tilde S_t\|^2_F} &\to & 0
\end{eqnarray*}
This in turn implies that, for the batch of iterates $\tilde x_t = (x_t,x_{t-1},...,x_{t-q+1})$, that $\tilde x_t\to (x^*,...,x^*)$ $q$-superlinearly, as a direct result of   \citet{dennis1974characterization}.

\end{proof}

\section{Extended numerical results}
\label{app:numerical}
\subsection{Logistic regression extra experiments}
\label{app:extranumerics:logreg}
We experiment with two types of models
\begin{align}
 A_{i,j}& =   b_iz_{i,j} + \omega  z_{i,j} c_j
    \label{eq:logreg_matrix_highsignal} &\tag{High signal regime}\\
 A_{i,j} &=  b_iz_{i,j}(1-c_j) + \omega z_{i,j} c_j
    \label{eq:logreg_matrix_lowsignal} &\tag{Low signal regime}
\end{align}
where
$c_j = \exp(-\bar c j/n)$ is the data decay rate (decaying influence of each feature), and 
$z_{i,j} \sim \mathcal N(0,1)$ Gaussian distributed i.i.d. $\omega$ controls the signal to noise ratio of the data, and the labels   $b_i \in \{1,-1\}$ with equal probability (class balanced).

\setlength{\tabcolsep}{2pt}
\renewcommand{\arraystretch}{0.8}

\begin{table}[ht!]
{\small
\centering
\begin{tabular}{|l||cc|cc|cc||cc|cc|cc|}
\hline
&\multicolumn{6}{c||}{Low signal regime}&\multicolumn{6}{c|}{High signal regime}\\
 & \multicolumn{2}{c|}{$\bar c = 10$}  	& \multicolumn{2}{c|}{$\bar c =  30$} 	& \multicolumn{2}{c||}{$\bar c =  50$} 	& \multicolumn{2}{c|}{$\bar c = 10$} 	& \multicolumn{2}{c|}{$\bar c =  30$} 	& \multicolumn{2}{c|}{$\bar c =  50$} 	\\
&curve 	&anch. 	&curve 	&anch. 	&curve 	&anch. 	&curve 	& anch. 	&curve 	&  anch. 	&curve 	&anch. 	\\\hline
Newton's	&  11	 &  11	 &  11	 &  11	 &  11	 &  11	 &  11	 &  11	 &  11	 &  11	 &  11	 &  11\\
Grad. Desc.	&2051	 &2051	 &2010	 &2010	 &2002	 &2002	 &2357	 &2357	 &2106	 &2106	 &2060	 &2060\\\hline
Br. (d,1)       	& 520	 & 520	 & 513	 & 513	 & 510	 & 510	 & 575	 & 575	 & 529	 & 529	 & 518	 & 518\\
Br. (d,v)	& 483	 & 529	 & 522	 & 526	 & 517	 & 499	 & 715	 & 562	 & 570	 & 515	 & 520	 & \textbf{130}\\
Br. (d,v,r)	& 505	 & 521	 & 502	 & 514	 & 502	 & 512	 & 573	 & 577	 & 525	 & 532	 & 529	 & 523\\
Br. (d,s)	& Inf	 & Inf	 & 539	 &1297	 & 602	 & 708	 & 545	 & 821	 & \textbf{497}	 & 885	 & 484	 &1176\\
Br. (d,s,r)	& 749	 & 576	 & 845	 & 502	 & 745	 & 867	 & 933	 &6925	 &1020	 & 642	 & 835	 & 646\\
Br. (d,p)	& 484	 & 450	 & 502	 & 464	 & 501	 & 477	 & 604	 & 525	 &\textbf{ 466*}	 & 460	 & 438	 & 472\\
Br. (d,p,r)	& Inf	 &1044	 & Inf	 &1050	 &1824	 &1058	 &4828	 &1151	 &2781	 &1089	 &2927	 &1078\\
Br. (d,o)	& Inf	 &8076	 & 560	 & 745	 & 658	 &1406	 & 862	 &1166	 & 570	 & 959	 & 447	 & 709\\
Br. (d,o,32)	& Inf	 &8031	 & 606	 & 930	 & 726	 & 561	 & 651	 &1958	 & 647	 & Inf	 & 598	 & 515\\
Br. (d,o,1000)	& Inf	 &8049	 & 560	 & 745	 & 658	 &1580	 & 862	 &1182	 & 570	 & 959	 & 447	 & 709\\
Br. (d,o,r)	&1303	 & 590	 & 688	 & 500	 & 873	 &1022	 & 787	 &4029	 &1460	 & 748	 &2009	 & 741\\
Br. (d,o,32,r)	&3191	 & 576	 & 721	 & 500	 & 684	 & 878	 & 766	 &4071	 & 857	 & 662	 & 738	 & 678\\
Br. (d,o,1000,r)	&1303	 & 590	 & 688	 & 500	 & 873	 &1011	 & 787	 &4025	 &1364	 & 748	 &2009	 & 741\\
Br. (i,1)       	& 520	 & 520	 & 513	 & 513	 & 510	 & 510	 & 575	 & 575	 & 529	 & 529	 & 518	 & 518\\
Br. (i,v)	& 558	 & 507	 & 471	 & 593	 & 400	 & 813	 & 579	 & 585	 & 534	 & 780	 & 526	 & 507\\
Br. (i,v,r)	& 505	 & 521	 & 502	 & 514	 & 502	 & 512	 & 573	 & 577	 & 525	 & 532	 & 529	 & 523\\
Br. (i,s)	& Inf	 &2122	 & 903	 & 559	 &1543	 & 534	 &1002	 & 869	 &1242	 & 594	 & 555	 & 529\\
Br. (i,s,r)	&1000	 & 631	 &2025	 & 712	 &1017	 & 809	 & 785	 & 672	 &1527	 & 834	 & 596	 & 521\\
Br. (i,p)	& 425	 & 454	 & 144	 &   \textbf{6*}	 &  \textbf{42*}	 &   Inf	 &  91	 & 294	 &   Inf	 &   Inf	 &   Inf	 &   Inf\\
Br. (i,p,r)	& Inf	 & 631	 & Inf	 & 677	 &  \textbf{48}	 & 880	 & 107	 & 666	 &   Inf	 & Inf	 &   Inf	 & Inf\\
Br. (i,o,s)	&1063	 &1664	 &1559	 & 858	 & 957	 & 962	 &1034	 &1056	 &1210	 & 968	 & 746	 &1031\\
Br. (i,o)	&  \textbf{21}	 &   \textbf{21}	 &   \textbf{8}	 &  Inf	 &    Inf	 &  16	 &    \textbf{8}	 &   \textbf{8}	 &  Inf	 &  \textbf{23}	 &  Inf	 &  Inf\\
Br. (i,o,s,r)	&1122	 & 599	 & 800	 & 602	 & 903	 & 576	 & 822	 & 626	 &2671	 & 670	 &1508	 & 528\\
Br. (i,o,r)	            & 119	 & 599	 &    \textbf{8}	 & 602	 &    Inf	 & 576	 &    \textbf{8}	 & 626	 &   Inf	 & 670	 &  Inf	 & 528\\
Br. (i,o,s,32)	& 520	 & 548	 & 611	 & \textbf{331}	 & 606	 & 976	 & 650	 & Inf	 &1025	 & 609	 & \textbf{409}	 & 443\\
Br. (i,o,32)	&  \textbf{21}	 &  \textbf{21}	 &    \textbf{8}	 &  Inf	 &    Inf	 &  \textbf{16}	 &    \textbf{8}	 &    \textbf{8}	 &  Inf	 &  \textbf{23}	 &  Inf	 &  Inf\\
Br. (i,o,s,32,r)	& Inf	 & 599	 & 510	 & 602	 & 633	 & 576	 & 629	 & 626	 & 645	 & 670	 & 532	 & 528\\
Br. (i,o,32,r)	& Inf	 & 599	 &    \textbf{8}	 & 602	 &  Inf	 & 576	 &    \textbf{8}	 & 626	 &  Inf	 & 670	 &  Inf	 & 528\\
Br. (i,o,s,1000)	&1056	 &1684	 &1551	 & 858	 & 957	 & 962	 &1034	 &1056	 &1219	 & 968	 & 746	 &1032\\
Br. (i,o,1000)	&  \textbf{21}	 &  \textbf{21}	 &    \textbf{8}	 &  Inf	 &   Inf	 &  \textbf{16}	 &    \textbf{8}	 &    \textbf{8}	 &   Inf	 &  \textbf{23}	 &  Inf	 &  Inf\\
Br. (i,o,s,1000,r)	&1168	 & 599	 & 800	 & 602	 & 903	 & 576	 & 822	 & 626	 &2671	 & 670	 &1449	 & 528\\
Br. (i,o,1000,r)	& 119	 & 599	 &   \textbf{8}	 & 602	 &   Inf	 & 576	 &    \textbf{8}	 & 626	 &   Inf  & 670	 &  Inf	 & 528\\\hline
Pow. (d,1)       	& 532	 & 532	 & 529	 & 529	 & 528	 & 528	 & 584	 & 584	 & 535	 & 535	 & 521	 & 521\\
Pow. (d,v)	& 409	 & 506	 & 256	 & 325	 & Inf	 & 499	 & 508	 & Inf	 & 508	 & 593	 & 489	 & 535\\
Pow. (d,v,r)	& 524	 & 525	 & 501	 & 521	 & 499	 & 519	 & 574	 & 580	 & 599	 & 537	 & 531	 & 528\\
Pow. (d,s)	&1520	 & 496	 &1114	 & 667	 & 439	 & 538	 &1154	 & 494	 &1503	 &1442	 & 519	 &1483\\
Pow. (d,s,r)	& 809	 & 537	 & 565	 & 512	 & 639	 & 519	 & 644	 & 573	 & 577	 & 537	 & 542	 & 548\\
Pow. (d,p)	& 534	 & 520	 & 510	 & 463	 & 483	 & 487	 & 547	 & 554	 & 499	 & 488	 & 467	 & 500\\
Pow. (d,p,r)	&1599	 & 860	 &2802	 & 527	 &1742	 & 529	 &1846	 & 703	 &2063	 & 601	 &1911	 & 625\\
Pow. (d,o)	& 529	 & 492	 & 502	 & 798	 &1485	 & 545	 & 644	 & 589	 & 616	 & 610	 & 557	 &2978\\
Pow. (d,o,32)	&2205	 & 537	 & 608	 &1071	 & 657	 & 814	 & Inf	 & 965	 &1216	 & 687	 & 551	 &1197\\
Pow. (d,o,1000)	& 529	 & 492	 & 502	 & 798	 &1485	 & 545	 & 644	 & 589	 & 616	 & 610	 & 557	 &2978\\
Pow. (d,o,r)	& 583	 & 537	 & 498	 & 512	 & 692	 & 519	 & 615	 & 573	 & 471	 & 537	 & 542	 & 548\\
Pow. (d,o,32,r)	& 515	 & 537	 & 651	 & 512	 & 680	 & 519	 &1420	 & 573	 & 471	 & 537	 & 542	 & 548\\
Pow. (d,o,1000,r)	& 583	 & 537	 & 498	 & 512	 & 692	 & 519	 & 615	 & 573	 & 471	 & 537	 & 542	 & 548\\
Pow. (i,1)       	& Inf	 & Inf	 & Inf	 & Inf	 & Inf	 & Inf	 & 294	 & 294	 & 746	 & 746	 & Inf	 & Inf\\
Pow. (i,v)	& Inf	 & Inf	 & Inf	 & Inf	 & Inf	 & Inf	 & Inf	 & Inf	 & Inf	 & Inf	 & Inf	 & Inf\\
Pow. (i,v,r)	&  Inf	 & Inf	 & Inf	 & Inf	 & Inf	 & Inf	 & Inf	 & Inf	 & Inf	 & Inf	 & Inf	 & Inf\\
Pow. (i,s)	 & Inf	 & Inf	 & 407	 & 308	 & Inf	 & 532	 & 519	 & 612	 & Inf	 & 487	 & Inf	 & 251\\
Pow. (i,s,r)	& Inf	 & Inf	 & Inf	 & Inf	 &1122	 & 190	 &1022	 & Inf	 & Inf	 & Inf	 & Inf	 & 281\\
Pow. (i,p)	& 537	 &1822	 & 367	 & 377	 & 245	 & 367	 & 423	 &1655	 & 245	 & 601	 & 277	 & 356\\
Pow. (i,p,r)	&2002	 & Inf	 & 465	 & 754	 & Inf	 & 434	 &3302	 & Inf	 & 322	 & 501	 & Inf	 & 286\\
Pow. (i,o,s)	& Inf	 & Inf	 & Inf	 & 389	 & 453	 & Inf	 & Inf	 & Inf	 & Inf	 & 107	 & 828	 & 471\\ 
Pow. (i,o)	&   8	 &  Inf	 &   8	 &   \textbf{6}	 &   8	 &  \textbf{ 6}	 &   \textbf{5}	 &   \textbf{5}	 &   \textbf{8}	 &   \textbf{6}	 &  \textbf{10}	 &   \textbf{6}\\
Pow. (i,o,s,r)	&2091	 & Inf	 &1743	 &  60	 & Inf	 & Inf	 &  Inf	 & Inf	 &2152	 & Inf	 & 102	 &1289\\
Pow. (i,o,r)	&   \textbf{7}	 &   \textbf{7}	 &   \textbf{7}	 &   7	 &   7	 &   7	 &   \textbf{5}	 &   \textbf{5}	 &  49	 &  Inf	 &  82	 &  89\\
Pow. (i,o,s,32)	& Inf	 & Inf	 & Inf	 & 389	 & 453	 & Inf	 & Inf	 & Inf	 & Inf	 & 107	 & 828	 & 471\\
Pow. (i,o,32)	&   8	 &   Inf	 &   8	 &   \textbf{6}	 &   8	 &   \textbf{6}	 &   \textbf{5}	 &   \textbf{5}	 &   \textbf{8}	 &   \textbf{6}	 &  \textbf{10}	 &   \textbf{6}\\
Pow. (i,o,s,32,r)	&2091	 & Inf	 &1743	 &  60	 & Inf	 & Inf	 &  Inf	 & Inf	 &2152	 & Inf	 & 102	 &1289\\
Pow. (i,o,32,r)	&   \textbf{7}	 &   \textbf{7}	 &   \textbf{7}	 &   7	 &   \textbf{7}	 &   7	 &   \textbf{5}	 &   \textbf{5}	 &  49	 &  Inf	 &  82	 &  89\\
Pow. (i,o,s,1000)	& Inf	 & Inf	 & Inf	 & 389	 & 453	 & Inf	 & Inf	 & Inf	 & Inf	 & 107	 & 828	 & 471\\
Pow. (i,o,1000)	&   8	 &   Inf	 &   8	 &   \textbf{6}	 &   8	 &   \textbf{6}	 &   \textbf{5}	 &   \textbf{5}	 &   \textbf{8}	 &   \textbf{6}	 &  \textbf{10}	 &   \textbf{6}\\
Pow. (i,o,s,1000,r)	&2091	 & Inf	 &1743	 &  60	 & Inf	 & Inf	 &  Inf	 & Inf	 &2152	 & Inf	 & 102	 &1289\\
Pow. (i,o,1000,r)	&   \textbf{7}	 &   \textbf{7}	 &   \textbf{7}	 &   7	 &   \textbf{7}	 &   7	 &   \textbf{5}	 &   \textbf{5}	 &  49	 &  Inf	 &  82	 &  89\\
\hline
\end{tabular}
}
\caption{\textbf{LogReg.} Number of iterations until $\|\nabla f(x_t)\|/\|\nabla f(x_0)\| \leq \epsilon = 10^{-4}$. $q = 5$. inf = more than 10000 iterations. $\sigma = 10, m = 100, n = 50$.   
d = direct update, i = inverse update.
1 = single-secant, v = vanilla, s = symmetric, p = PSD, o = ours.  
s = scaling, r = rejection with 0.01 tolerance.
Number refers to $\nu$ value in $\mu$-correction.
Broyden and Powell methods shown.
  }
\label{tab:app:logreg:methodsA}
\end{table}

\begin{table}[ht!]
{\small
\centering
\begin{tabular}{|l||cc|cc|cc||cc|cc|cc|}
\hline
&\multicolumn{6}{c||}{Low signal regime}&\multicolumn{6}{c|}{High signal regime}\\
 & \multicolumn{2}{c|}{$\bar c = 10$}  	& \multicolumn{2}{c|}{$\bar c =  30$} 	& \multicolumn{2}{c||}{$\bar c =  50$} 	& \multicolumn{2}{c|}{$\bar c = 10$} 	& \multicolumn{2}{c|}{$\bar c =  30$} 	& \multicolumn{2}{c|}{$\bar c =  50$} 	\\
&curve 	&anch. 	&curve 	&anch. 	&curve 	&anch. 	&curve 	& anch. 	&curve 	&  anch. 	&curve 	&anch. 	\\\hline

DFP (d,1)       	& 504	 & 504	 & 500	 & 500	 & 499	 & 499	 & 570	 & 570	 & 522	 & 522	 & 512	 & 512\\
DFP (d,v)	& 663	 & 439	 & 764	 & 990	 & 515	 & 548	 &1320	 & 645	 & 675	 & 456	 & 687	 & 578\\
DFP (d,v,r)	& 508	 & 509	 & 505	 & 507	 & 506	 & 507	 & 575	 & 576	 & 528	 & 530	 & 542	 & 521\\
DFP (d,s)	& 635	 & 812	 & 547	 & 623	 &1961	 & 515	 & 596	 & 903	 & 940	 & 452	 & Inf	 & 785\\
DFP (d,s,r)	& 671	 & 509	 & 681	 & 651	 & 636	 & 551	 & 602	 &2582	 & 639	 & 687	 & 518	 & 668\\
DFP (d,p)	& 512	 & 501	 & 421	 & 453	 & 456	 & 451	 & 487	 & 489	 & 501	 & 484	 & 487	 & 481\\
DFP (d,p,r)	& 909	 & 549	 &1210	 & 539	 & 644	 & 537	 & 838	 & 624	 & 770	 & 561	 &1306	 & 559\\
DFP (d,o)	& 666	 & 603	 & 646	 & 447	 & 535	 &1115	 & Inf	 & 707	 & 511	 & Inf	 & 598	 & 463\\
DFP (d,o,32)	& 666	 & 603	 &1040	 & 719	 & 526	 &1764	 &1079	 & 681	 &2640	 & 560	 & Inf	 & 851\\
DFP (d,o,1000)	& 666	 & 603	 & 646	 & 447	 & 535	 &1115	 & Inf	 & 707	 & 511	 & Inf	 & 598	 & 463\\
DFP (d,o,r)	& 651	 & 509	 &1333	 & 513	 & 648	 & 504	 & 690	 & 579	 & 739	 & 670	 & 624	 & 670\\
DFP (d,o,32,r)	&2221	 & 509	 & 554	 & 513	 &1698	 & 504	 &1145	 & 579	 & 868	 & 670	 & 683	 & 670\\
DFP (d,o,1000,r)	& 651	 & 509	 &1333	 & 513	 & 648	 & 504	 & 690	 & 579	 & 739	 & 670	 & 624	 & 670\\
DFP (i,1)       	& 504	 & 504	 & 500	 & 500	 & 499	 & 499	 & 570	 & 570	 & 522	 & 522	 & 512	 & 512\\
DFP (i,v)	& Inf	 & Inf	 & Inf	 & Inf	 & Inf	 & Inf	 & Inf	 &   Inf	 & Inf	 & Inf	 & Inf	 & Inf\\
DFP (i,v,r)	& Inf	 & Inf	 & Inf	 & Inf	 & Inf	 & Inf	 & Inf	 & Inf	 & Inf	 & Inf	 & Inf	 & Inf\\
DFP (i,s)	& 530	 & 513	 & 524	 & 513	 & 518	 & 524	 & 602	 & 594	 & 582	 & 784	 & 559	 & 729\\
DFP (i,s,r)	& 760	 & 511	 & 548	 & 510	 &1394	 & 506	 & 621	 & 575	 & 565	 & 526	 & 557	 & 516\\
DFP (i,p)	& 425	 & 708	 & 434	 & 369	 & 442	 & 330	 & 503	 & 592	 & 389	 &  37	 & 420	 & 426\\
DFP (i,p,r)	& 703	 & 511	 & 438	 & 501	 & 477	 & 523	 & 598	 & 581	 & 444	 & 542	 & 411	 & 508\\
DFP (i,o,s)	& 461	 & 390	 & 451	 & 408	 & 442	 & 398	 & 524	 & 419	 & 582	 & 418	 & 467	 & 405\\
DFP (i,o)	&   \textbf{7}	 &   \textbf{7}	 &   \textbf{6}	 &   \textbf{6}	 &   \textbf{6}	 &   \textbf{6}	 &   \textbf{7}	 &   \textbf{7}	 &   \textbf{6}	 &   \textbf{6}	 &   \textbf{6}	 &   \textbf{6}\\
DFP (i,o,s,r)	& 968	 & 405	 & 470	 & 403	 & 433	 & 402	 & 475	 & 424	 & 471	 & 410	 & 456	 & 408\\
DFP (i,o,r)	&   Inf	 &  12	 &   \textbf{6}	 &  12	 &   \textbf{6}	 &  12	 &   9	 &  13	 &   \textbf{6}	 &  12	 &   \textbf{6}	 &  13\\
DFP (i,o,s,32)	& 460	 & 405	 & 456	 & 419	 & 450	 & 393	 & 461	 & 472	 & 443	 & 410	 & 455	 & 408\\
DFP (i,o,32)	& Inf	 & Inf &   \textbf{6}	 &   \textbf{6}	 &   \textbf{6}	 &   \textbf{6}	 &   \textbf{7}	 &   \textbf{7}	 &   \textbf{6}	 &   \textbf{6}	 &   \textbf{6}	 &   \textbf{6}\\
DFP (i,o,s,32,r)	& 456	 & 405	 & 445	 & 403	 & 407	 & 403	 & 465	 & 435	 & 543	 & 414	 & 494	 & 410\\
DFP (i,o,32,r)	&   Inf	 &  12	 &   \textbf{6}	 &  12	 &   \textbf{6}	 &  12	 &   9	 &  13	 &   \textbf{6}	 &  12	 &   \textbf{6}	 &  13\\ 
DFP (i,o,s,1000)	& 461	 & 390	 & 451	 & 408	 & 442	 & 398	 & 524	 & 419	 & 582	 & 418	 & 467	 & 405\\
DFP (i,o,1000)		& Inf	 & Inf		 &   \textbf{6}	 &   \textbf{6}	 &   \textbf{6}	 &  \textbf{6}	 &   \textbf{7}	 &   \textbf{7}	 &   \textbf{6}	 &   \textbf{6}	 &   \textbf{6}	 &   \textbf{6}\\
DFP (i,o,s,1000,r)	& 968	 & 405	 & 470	 & 403	 & 433	 & 402	 & 475	 & 424	 & 471	 & 410	 & 456	 & 408\\
DFP (i,o,1000,r)	&   Inf	 &  12	 &   \textbf{6}	 &  12	 &   \textbf{6}	 &  12	 &   9	 &  13	 &   \textbf{6}	 &  12	 &   \textbf{6}	 &  13\\\hline
BFGS (d,1)       	& 502	 & 502	 & 498	 & 498	 & 498	 & 498	 & 570	 & 570	 & 522	 & 522	 & 512	 & 512\\
BFGS (d,v)	& 499	 & 499	 & 498	 & 498	 & 500	 & 490	 & 569	 & 569	 & 522	 & 519	 & 517	 & 515\\
BFGS (d,v,r)	& 502	 & 503	 & 500	 & 501	 & 500	 & 500	 & 572	 & 572	 & 524	 & 525	 & 541	 & 515\\
BFGS (d,s)	& 576	 &1087	 &1246	 & 643	 &1268	 & 560	 & 916	 & 869	 & 659	 &1114	 & 542	 & 571\\
BFGS (d,s,r)	& 750	 & 506	 & 596	 & 505	 &1170	 & 557	 & 734	 & 584	 & 631	 & 586	 & 558	 & 715\\
BFGS (d,p)	& 475	 & 453	 & 487	 & 462	 & 459	 & 460	 & 467	 & 506	 & 480	 & 496	 & 482	 & 500\\
BFGS (d,p,r)	&1807	 & 545	 & 740	 & 535	 & 660	 & 533	 &1688	 & 631	 & 878	 & 556	 & 643	 & 507\\
BFGS (d,o)	&1768	 & 725	 & 784	 & 935	 & 859	 &1007	 & 623	 & 928	 & 565	 & 811	 & 537	 & 688\\
BFGS (d,o,32)	& 796	 & 740	 & 815	 & 654	 & 546	 & 634	 & 810	 & 782	 &1755	 & 669	 & 731	 &1332\\
BFGS (d,o,1000)	&1768	 & 725	 & 784	 & 935	 & 859	 &1007	 & 623	 & 928	 & 565	 & 811	 & 537	 & 688\\
BFGS (d,o,r)	& 544	 & 503	 & 578	 & 502	 & 689	 & 502	 & 804	 & 574	 & 543	 & 525	 & 523	 & 515\\
BFGS (d,o,32,r)	& 583	 & 503	 & 717	 & 502	 & 689	 & 502	 & 710	 & 574	 & 543	 & 525	 & 523	 & 515\\
BFGS (d,o,1000,r)	& 544	 & 503	 & 578	 & 502	 & 689	 & 502	 & 804	 & 574	 & 543	 & 525	 & 523	 & 515\\
BFGS (i,1)       	& 502	 & 502	 & 498	 & 498	 & 498	 & 498	 & 570	 & 570	 & 522	 & 522	 & 512	 & 512\\
BFGS (i,v)	& 499	 & 502	 & 500	 & 502	 & 497	 & 502	 & 569	 & 576	 & 527	 & 519	 & 511	 & 517\\
BFGS (i,v,r)	& 502	 & 503	 & 500	 & 501	 & 500	 & 500	 & 572	 & 572	 & 524	 & 525	 & 541	 & 515\\
BFGS (i,s)	& 530	 & 539	 & 926	 &1151	 &1509	 & 567	 & 608	 & 811	 & 851	 &1552	 & Inf	 & 563\\
BFGS (i,s,r)	& 884	 & 507	 & 588	 & 505	 & 568	 & 504	 &1235	 & 574	 &1131	 & 527	 &1141	 & 520\\
BFGS (i,p)	& 265	 & 268	 & 152	 & 183	 & 281	 & 289	 & 377	 & 305	 & 256	 & 288	 & 549	 & 195\\
BFGS (i,p,r)	& 665	 & 507	 &1006	 & 505	 & 258	 & 505	 & 399	 & 575	 & 255	 & 527	 & 194	 & 516\\
BFGS (i,o,s)	& 466	 & 667	 & 446	 & 854	 & 651	 & 488	 & 646	 & 418	 &1063	 & 494	 & 513	 & 473\\
BFGS (i,o)	&   \textbf{5}	 &   \textbf{5}	 &   \textbf{5}	 &  \textbf{5}	 &   \textbf{5}	 &   \textbf{5}	 &   \textbf{5}	 &   \textbf{5}	 &   \textbf{6}	 &   \textbf{5}	 &   \textbf{5}	 &   \textbf{5}\\
BFGS (i,o,s,r)	& 474	 & 474	 & 472	 & 472	 & 472	 & 472	 & 498	 & 498	 & 487	 & 487	 & 485	 & 485\\
BFGS (i,o,r)	&  10	 &  10	 &  10	 &  10	 &  10	 &  10	 &  10	 &  10	 &  10	 &  10	 &  10	 &  10\\
BFGS (i,o,s,32)	& 459	 & 417	 &1336	 & 409	 & 450	 & 574	 & 460	 & 459	 & 451	 & 424	 & 501	 & 408\\
BFGS (i,o,32)	&   \textbf{5}	 &   \textbf{5}	 &   \textbf{5}	 &   \textbf{5}	 &   \textbf{5}	 &   \textbf{5}	 &   \textbf{5}	 &   \textbf{5}	 &   \textbf{6}	 &   \textbf{5}	 &   \textbf{5}	 &   \textbf{5}\\
BFGS (i,o,s,32,r)	& 474	 & 474	 & 472	 & 472	 & 472	 & 472	 & 498	 & 498	 & 487	 & 487	 & 485	 & 485\\
BFGS (i,o,32,r)	&  10	 &  10	 &  10	 &  10	 &  10	 &  10	 &  10	 &  10	 &  10	 &  10	 &  10	 &  10\\
BFGS (i,o,s,1000)	& 466	 & 667	 & 446	 & 854	 & 651	 & 488	 & 646	 & 418	 &1064	 & 494	 & 513	 & 473\\
BFGS (i,o,1000)	&   \textbf{5}	 &   \textbf{5}	 &   \textbf{5}	 &   \textbf{5}	 &   \textbf{5}	 &   \textbf{5}	 &   \textbf{5}	 &   \textbf{5}	 &   \textbf{6}	 &   \textbf{5}	 &   \textbf{5}	 &   \textbf{5}\\
BFGS (i,o,s,1000,r)	& 474	 & 474	 & 472	 & 472	 & 472	 & 472	 & 498	 & 498	 & 487	 & 487	 & 485	 & 485\\
BFGS (i,o,1000,r)	&  10	 &  10	 &  10	 &  10	 &  10	 &  10	 &  10	 &  10	 &  10	 &  10	 &  10	 &  10\\

\hline
\end{tabular}
}
\caption{\textbf{LogReg.} Number of iterations until $\|\nabla f(x_t)\|/\|\nabla f(x_0)\| \leq \epsilon = 10^{-4}$. $q = 5$. inf = more than 10000 iterations. $\sigma = 10, m = 100, n = 50$.  
d = direct update, i = inverse update.
1 = single-secant, v = vanilla, s = symmetric, p = PSD, o = ours.
s = scaling, r = rejection with 0.01 tolerance.
Number refers to $\nu$ value in $\mu$-correction.
DFP and BFGS methods shown.}
\label{tab:app:logreg:methodsB}
\end{table}

\setlength{\tabcolsep}{6pt}  
\renewcommand{\arraystretch}{1}

\subsection{$p$-order minimization, extra experiments}
\label{app:extranumerics:pnormraised}

This section presents extended results for the $p$-power minimization, in tables \ref{tab:app:pnormraised:p1.5:methodsA}- \ref{tab:app:pnormraised:p3.5:methodsB}.

\setlength{\tabcolsep}{3pt}
\renewcommand{\arraystretch}{0.8}

\begin{table}[ht!]
{\small
\centering
\begin{tabular}{|l||ccc|ccc||ccc|ccc||ccc|ccc||ccc|ccc||ccc|ccc||ccc|ccc|}
\hline
&\multicolumn{3}{c|}{high noise ($\sigma = 10$)}&\multicolumn{3}{c|}{medium noise ($\sigma = 1$)}&\multicolumn{3}{c|}{low noise ($\sigma = 0.1$)}\\
 & $\bar c = 10$  	&$\bar c =  30$ 	&$\bar c =  50$  & $\bar c = 10$  	&$\bar c =  30$ 	&$\bar c =  50$ 	&$\bar c = 10$ 	&$\bar c =  30$ 	&$\bar c =  50$ 	\\\hline
Newton's 	   &2730	 &2632	 &2569	 &2801	 &2811	 &2884	 &2834	 &2790	 &2842\\
Grad. Desc. 	 &1809	 &1392	 &1760	 &1213	 &1245	 &1787	 &1405	 &1358	 &2042\\\hline
Br. (S)        	   &1083	 & 677	 & 954	 & 647	 & 694	 & 993	 &1359	 & 705	 & 877\\
Br. (i,S)        	    &1145	 &1539	 &1240	 & 627	 &1118	 &1000	 & 956	 &1445	 &1672\\
Br. (d,v)   	& \textbf{524}	 & 551	 & 380	 & 417	 & \textbf{226}	 & \textbf{245}	 & 383	 & 606	 & \textbf{264}\\
Br. (d,v,r) 	    &1083	 & 677	 & 928	 & 647	 & 694	 & 993	 & 707	 & 820	 & 904\\
\hline
\end{tabular}
}
\caption{\textbf{$p$-order minimization, $p = 1.5$.}  Number of iterations until $\|\nabla f(x_t)\|/\|\nabla f(x_0)\| \leq \epsilon = 10^{-1}$. $q = 5$. inf = more than 10000 iterations. $m = 100, n = 50$.   
d = direct update, i = inverse update.
1 = single-secant, v = vanilla, s = symmetric, p = PSD, o = ours.
s = scaling, r = rejection with 0.01 tolerance.
Number refers to $\nu$ value in $\mu$-correction.
Broyden and Powell methods shown.}
\label{tab:app:pnormraised:p1.5:methodsA}
\end{table}

\begin{table}[ht!]
{\small
\centering
\begin{tabular}{|l||ccc|ccc||ccc|ccc||ccc|ccc||ccc|ccc||ccc|ccc||ccc|ccc|}
\hline
&\multicolumn{3}{c|}{high noise ($\sigma = 10$)}&\multicolumn{3}{c|}{medium noise ($\sigma = 1$)}&\multicolumn{3}{c|}{low noise ($\sigma = 0.1$)}\\
 & $\bar c = 10$  	&$\bar c =  30$ 	&$\bar c =  50$  & $\bar c = 10$  	&$\bar c =  30$ 	&$\bar c =  50$ 	&$\bar c = 10$ 	&$\bar c =  30$ 	&$\bar c =  50$ 	\\\hline

DFP (d,1)        	& Inf	 & Inf	 & Inf	 & Inf	 & Inf	 &7405	 & Inf	 & Inf	 & Inf\\
DFP (i,S)        	& Inf	 & Inf	 & Inf	 & Inf	 & Inf	 &7405	 & Inf	 & Inf	 & Inf\\
DFP (d,v)  	&  \multicolumn{9}{c|}{------------------------------ inf ------------------------------}\\
DFP (d,v,r) 	& Inf	 & Inf	 & Inf	 & Inf	 & Inf	 &7405	 & Inf	 & Inf	 & Inf\\
DFP (d,s)  	&  \multicolumn{9}{c|}{------------------------------ inf ------------------------------}\\
DFP (d,s,r) 	& Inf	 & Inf	 & Inf	 & Inf	 & Inf	 &7405	 & Inf	 & Inf	 & Inf\\
\hline
\end{tabular}
}
\caption{\textbf{$p$-order minimization, $p = 1.5$.}  Number of iterations until $\|\nabla f(x_t)\|/\|\nabla f(x_0)\| \leq \epsilon = 10^{-1}$. $q = 5$. inf = more than 10000 iterations. $m = 100, n = 50$.    
d = direct update, i = inverse update.
1 = single-secant, v = vanilla, s = symmetric, p = PSD, o = ours.
s = scaling, r = rejection with 0.01 tolerance.
Number refers to $\nu$ value in $\mu$-correction.
DFP and BFGS methods shown.}
\label{tab:app:pnormraised:p1.5:methodsB}
\end{table}

\begin{table}[ht!]
{\small
\centering
\begin{tabular}{|l||ccc|ccc||ccc|ccc||ccc|ccc||ccc|ccc||ccc|ccc||ccc|ccc|}
\hline
&\multicolumn{3}{c|}{high noise ($\sigma = 10$)}&\multicolumn{3}{c|}{medium noise ($\sigma = 1$)}&\multicolumn{3}{c|}{low noise ($\sigma = 0.1$)}\\
 & $\bar c = 10$  	&$\bar c =  30$ 	&$\bar c =  50$  & $\bar c = 10$  	&$\bar c =  30$ 	&$\bar c =  50$ 	&$\bar c = 10$ 	&$\bar c =  30$ 	&$\bar c =  50$ 	\\\hline
Newtons method 	&5105	 &5061	 &5035	 &5190	 &5178	 &5239	 &5228	 &5201	 &5255\\
Gradient Descent 	& \multicolumn{9}{c|}{------------------------------ inf ------------------------------}\\
\hline
Br. (d,1)        	& Inf	 & Inf	 & Inf	 & Inf	 &9670	 & Inf	 & Inf	 & Inf	 & Inf\\
Br. (i,S)        	&  \multicolumn{9}{c|}{------------------------------ inf ------------------------------}\\
Br. (d,v)  	& Inf	 & Inf	 &3937	 &2197	 & 669	 & \textbf{428}	 & 737	 & 638	 & 837\\
Br. (d,v,r) 	& Inf	 & Inf	 & Inf	 & Inf	 &4113	 &9701	 &2506	 &9208	 &5352\\
\hline
\end{tabular}
}
\caption{\textbf{$p$-order minimization, $p = 2.5$.}  Number  of iterations until $\|\nabla f(x_t)\|/\|\nabla f(x_0)\| \leq \epsilon = 10^{-2}$. $q = 5$. inf = more than 10000 iterations. $m = 100, n = 50$.    
d = direct update, i = inverse update.
1 = single-secant, v = vanilla, s = symmetric, p = PSD, o = ours.
s = scaling, r = rejection with 0.01 tolerance.
Number refers to $\nu$ value in $\mu$-correction.
Broyden and Powell methods shown.}
\label{tab:app:pnormraised:p2.5:methodsA}
\end{table}

\begin{table}[ht!]
{\small
\centering
\begin{tabular}{|l||ccc|ccc||ccc|ccc||ccc|ccc||ccc|ccc||ccc|ccc||ccc|ccc|}
\hline
&\multicolumn{3}{c|}{high noise ($\sigma = 10$)}&\multicolumn{3}{c|}{medium noise ($\sigma = 1$)}&\multicolumn{3}{c|}{low noise ($\sigma = 0.1$)}\\
 & $\bar c = 10$  	&$\bar c =  30$ 	&$\bar c =  50$  & $\bar c = 10$  	&$\bar c =  30$ 	&$\bar c =  50$ 	&$\bar c = 10$ 	&$\bar c =  30$ 	&$\bar c =  50$ 	\\\hline
DFP (d,1)        	&  \multicolumn{9}{c|}{------------------------------ inf ------------------------------}\\
DFP (i,S)        	&  \multicolumn{9}{c|}{------------------------------ inf ------------------------------}\\
DFP (d,v)  	&  \multicolumn{9}{c|}{------------------------------ inf ------------------------------}\\
DFP (d,v,r) 	& \multicolumn{9}{c|}{------------------------------ inf ------------------------------}\\
DFP (d,s)  	& Inf	 & Inf	 & Inf	 & Inf	 & Inf	 & Inf	 & Inf	 & Inf	 &4521\\
DFP (d,s,r) 	& \multicolumn{9}{c|}{------------------------------ inf ------------------------------}\\
\hline
\end{tabular}
}
\caption{\textbf{$p$-order minimization, $p = 2.5$.}  Number   of iterations until $\|\nabla f(x_t)\|/\|\nabla f(x_0)\| \leq \epsilon = 10^{-2}$. $q = 5$. inf = more than 10000 iterations. $m = 100, n = 50$.  
d = direct update, i = inverse update.
1 = single-secant, v = vanilla, s = symmetric, p = PSD, o = ours.
s = scaling, r = rejection with 0.01 tolerance.
Number refers to $\nu$ value in $\mu$-correction.
DFP and BFGS methods shown.}
\label{tab:app:pnormraised:p2.5:methodsB}
\end{table}

\begin{table}[ht!]
{\small
\centering
\begin{tabular}{|l||ccc|ccc||ccc|ccc||ccc|ccc||ccc|ccc||ccc|ccc||ccc|ccc|}
\hline
&\multicolumn{3}{c|}{high noise ($\sigma = 10$)}&\multicolumn{3}{c|}{medium noise ($\sigma = 1$)}&\multicolumn{3}{c|}{low noise ($\sigma = 0.1$)}\\
 & $\bar c = 10$  	&$\bar c =  30$ 	&$\bar c =  50$  & $\bar c = 10$  	&$\bar c =  30$ 	&$\bar c =  50$ 	&$\bar c = 10$ 	&$\bar c =  30$ 	&$\bar c =  50$ 	\\\hline
Newtons method 	&4882	 &4862	 &4847	 &4911	 &4895	 &4940	 &4940	 &4914	 &4954\\
Gradient Descent 	& Inf	 & Inf	 & Inf	 & Inf	 & Inf	 & Inf	 & Inf	 & Inf	 & Inf\\\hline

Newtons method, 	&4882	 &4862	 &4847	 &4911	 &4895	 &4940	 &4940	 &4914	 &4954\\
Gradient Descent, 	& Inf	 & Inf	 & Inf	 & Inf	 & Inf	 & Inf	 & Inf	 & Inf	 & Inf\\
Br. (d,1)        	& Inf	 & Inf	 & Inf	 &9884	 & Inf	 & Inf	 & Inf	 &8217	 & Inf\\
Br. (i,S)        	&  \multicolumn{9}{c|}{------------------------------ inf ------------------------------}\\
\hline
\end{tabular}
}
\caption{\textbf{$p$-order minimization, $p = 3.5$.}  Number    of iterations until $\|\nabla f(x_t)\|/\|\nabla f(x_0)\| \leq \epsilon = 10^{-2}$. $q = 5$. inf = more than 10000 iterations. $m = 100, n = 50$.   
d = direct update, i = inverse update.
1 = single-secant, v = vanilla, s = symmetric, p = PSD, o = ours.
s = scaling, r = rejection with 0.01 tolerance.
Number refers to $\nu$ value in $\mu$-correction.
Broyden and Powell methods shown.}
\label{tab:app:pnormraised:p3.5:methodsA}
\end{table}

\begin{table}[ht!]
{\small
\centering
\begin{tabular}{|l||ccc|ccc||ccc|ccc||ccc|ccc||ccc|ccc||ccc|ccc||ccc|ccc|}
\hline
&\multicolumn{3}{c|}{high noise ($\sigma = 10$)}&\multicolumn{3}{c|}{medium noise ($\sigma = 1$)}&\multicolumn{3}{c|}{low noise ($\sigma = 0.1$)}\\
 & $\bar c = 10$  	&$\bar c =  30$ 	&$\bar c =  50$  & $\bar c = 10$  	&$\bar c =  30$ 	&$\bar c =  50$ 	&$\bar c = 10$ 	&$\bar c =  30$ 	&$\bar c =  50$ 	\\\hline
DFP (d,1)        	&  \multicolumn{9}{c|}{------------------------------ inf ------------------------------}\\
DFP (i,S)        	&  \multicolumn{9}{c|}{------------------------------ inf ------------------------------}\\
DFP (d,v)  	&  \multicolumn{9}{c|}{------------------------------ inf ------------------------------}\\
DFP (d,v,r) 	&  \multicolumn{9}{c|}{------------------------------ inf ------------------------------}\\
DFP (d,s)  	&  \multicolumn{9}{c|}{------------------------------ inf ------------------------------}\\
DFP (d,s,r) 	&  \multicolumn{9}{c|}{------------------------------ inf ------------------------------}\\
\hline
\end{tabular}
}
\caption{\textbf{$p$-order minimization, $p = 3.5$.}  Number  of iterations until $\|\nabla f(x_t)\|/\|\nabla f(x_0)\| \leq \epsilon = 10^{-2}$. $q = 5$. inf = more than 10000 iterations. $m = 100, n = 50$.   
d = direct update, i = inverse update.
1 = single-secant, v = vanilla, s = symmetric, p = PSD, o = ours.
s = scaling, r = rejection with 0.01 tolerance.
Number refers to $\nu$ value in $\mu$-correction.
DFP and BFGS methods shown.}
\label{tab:app:pnormraised:p3.5:methodsB}
\end{table}

\setlength{\tabcolsep}{6pt}  
\renewcommand{\arraystretch}{1}


\subsection{Cross entropy extra experiments}
\label{app:extranumerics:mcloss}
This section presents extended results for the multiclass logistic regression minimization, in tables \ref{tab:app:mcloss:methodsA} and  \ref{tab:app:mcloss:methodsB}.

\setlength{\tabcolsep}{3pt}
\renewcommand{\arraystretch}{0.8}

\begin{table}[ht!]
{\small
\centering
\begin{tabular}{|l||ccc|ccc|ccc|}
\hline
&\multicolumn{3}{c|}{high noise ($\sigma = 1$)}&\multicolumn{3}{c|}{medium noise ($\sigma = 0.1$)}&\multicolumn{3}{c|}{low noise ($\sigma = 0.01$)}\\
 & $\bar c = 10$  	&$\bar c =  30$ 	&$\bar c =  50$  & $\bar c = 10$  	&$\bar c =  30$ 	&$\bar c =  50$ 	& $\bar c = 10$  	&$\bar c =  30$ 	&$\bar c =  50$ 		\\\hline
Grad. Desc.	&  \multicolumn{9}{c|}{------------------------------ inf ------------------------------}\\
Br. (d,1)       	&  \multicolumn{9}{c|}{------------------------------ inf ------------------------------}\\
Br. (d,v)	&  \multicolumn{9}{c|}{------------------------------ inf ------------------------------}\\
Br. (d,v,r)	&  \multicolumn{9}{c|}{------------------------------ inf ------------------------------}\\
Br. (d,s)	&  \multicolumn{9}{c|}{------------------------------ inf ------------------------------}\\
Br. (d,s,r)	&  \multicolumn{9}{c|}{------------------------------ inf ------------------------------}\\
Br. (d,p)	&  \multicolumn{9}{c|}{------------------------------ inf ------------------------------}\\
Br. (d,p,r)	& \multicolumn{9}{c|}{------------------------------ inf ------------------------------}\\
Br. (d,o)	&  \multicolumn{9}{c|}{------------------------------ inf ------------------------------}\\
Br. (d,o,10)	&  \multicolumn{9}{c|}{------------------------------ inf ------------------------------}\\
Br. (d,o,100)	& \multicolumn{9}{c|}{------------------------------ inf ------------------------------}\\
Br. (d,o,r)	& \multicolumn{9}{c|}{------------------------------ inf ------------------------------}\\
Br. (d,o,10,r)	&  \multicolumn{9}{c|}{------------------------------ inf ------------------------------}\\
Br. (d,o,100,r)	&  \multicolumn{9}{c|}{------------------------------ inf ------------------------------}\\
Br. (i,1)       	&  \multicolumn{9}{c|}{------------------------------ inf ------------------------------}\\
Br. (i,v) 	&  \multicolumn{9}{c|}{------------------------------ inf ------------------------------}\\
Br. (i,v,r)	&  \multicolumn{9}{c|}{------------------------------ inf ------------------------------}\\
Br. (i,s) 	&\textbf{1006}	 &9679	 &4566	 &\textbf{1052}	 & Inf	 & Inf	 &\textbf{1083}	 &\textbf{1347}	 & Inf\\
Br. (i,s,r)	&  \multicolumn{9}{c|}{------------------------------ inf ------------------------------}\\
Br. (i,p) 	&  \multicolumn{9}{c|}{------------------------------ inf ------------------------------}\\
Br. (i,p,r)	&  \multicolumn{9}{c|}{------------------------------ inf ------------------------------}\\
Br. (i,o,s) 	&  \multicolumn{9}{c|}{------------------------------ inf ------------------------------}\\
Br. (i,o) 	&  \multicolumn{9}{c|}{------------------------------ inf ------------------------------}\\
Br. (i,o,s,10) 	&1435	 &\textbf{8457}	 &\textbf{4519}	 &5284	 &\textbf{3089}	 & Inf	 &2651	 &1600	 &\textbf{2845}\\
Br. (i,o,10) 	&  \multicolumn{9}{c|}{------------------------------ inf ------------------------------}\\
Br. (i,o,s,100) 	& Inf	 & Inf	 & Inf	 &8882	 &7813	 & Inf	 & Inf	 & Inf	 & Inf\\
Br. (i,o,100) 	&  \multicolumn{9}{c|}{------------------------------ inf ------------------------------}\\
Br. (i,o,s,r)	&  \multicolumn{9}{c|}{------------------------------ inf ------------------------------}\\
Br. (i,o,r)	&  \multicolumn{9}{c|}{------------------------------ inf ------------------------------}\\
Br. (i,o,s,10,r)	& \multicolumn{9}{c|}{------------------------------ inf ------------------------------}\\
Br. (i,o,10,r)	&  \multicolumn{9}{c|}{------------------------------ inf ------------------------------}\\
Br. (i,o,s,100,r)	&  \multicolumn{9}{c|}{------------------------------ inf ------------------------------}\\
Br. (i,o,100,r)	&  \multicolumn{9}{c|}{------------------------------ inf ------------------------------}\\
\hline
Pow. (d,1)       	&  \multicolumn{9}{c|}{------------------------------ inf ------------------------------}\\
Pow. (d,v)	& \multicolumn{9}{c|}{------------------------------ inf ------------------------------}\\
Pow. (d,v,r)	& \multicolumn{9}{c|}{------------------------------ inf ------------------------------}\\
Pow. (d,s)	&  \multicolumn{9}{c|}{------------------------------ inf ------------------------------}\\
Pow. (d,s,r)	&  \multicolumn{9}{c|}{------------------------------ inf ------------------------------}\\
Pow. (d,p)	&  \multicolumn{9}{c|}{------------------------------ inf ------------------------------}\\
Pow. (d,p,r)	&  \multicolumn{9}{c|}{------------------------------ inf ------------------------------}\\
Pow. (d,o)	&  \multicolumn{9}{c|}{------------------------------ inf ------------------------------}\\
Pow. (d,o,10)	&  \multicolumn{9}{c|}{------------------------------ inf ------------------------------}\\
Pow. (d,o,100)	&  \multicolumn{9}{c|}{------------------------------ inf ------------------------------}\\
Pow. (d,o,r)	&  \multicolumn{9}{c|}{------------------------------ inf ------------------------------}\\
Pow. (d,o,10,r)	& \multicolumn{9}{c|}{------------------------------ inf ------------------------------}\\
Pow. (d,o,100,r)	&  \multicolumn{9}{c|}{------------------------------ inf ------------------------------}\\
Pow. (i,1)       	&  \multicolumn{9}{c|}{------------------------------ inf ------------------------------}\\
Pow. (i,v) 	&  \multicolumn{9}{c|}{------------------------------ inf ------------------------------}\\
Pow. (i,v,r)	& \multicolumn{9}{c|}{------------------------------ inf ------------------------------}\\
Pow. (i,s) 	& \multicolumn{9}{c|}{------------------------------ inf ------------------------------}\\
Pow. (i,s,r)	&  \multicolumn{9}{c|}{------------------------------ inf ------------------------------}\\
Pow. (i,p) 	&  \multicolumn{9}{c|}{------------------------------ inf ------------------------------}\\
Pow. (i,p,r)	& \multicolumn{9}{c|}{------------------------------ inf ------------------------------}\\
Pow. (i,o,s) 	& \multicolumn{9}{c|}{------------------------------ inf ------------------------------}\\
Pow. (i,o) 	&  \multicolumn{9}{c|}{------------------------------ inf ------------------------------}\\
Pow. (i,o,s,10) 	&  \multicolumn{9}{c|}{------------------------------ inf ------------------------------}\\
Pow. (i,o,10) 	&  \multicolumn{9}{c|}{------------------------------ inf ------------------------------}\\
Pow. (i,o,s,100) 	& \multicolumn{9}{c|}{------------------------------ inf ------------------------------}\\
Pow. (i,o,100) 	&  \multicolumn{9}{c|}{------------------------------ inf ------------------------------}\\
Pow. (i,o,s,r)	&  \multicolumn{9}{c|}{------------------------------ inf ------------------------------}\\
Pow. (i,o,r)	&  \multicolumn{9}{c|}{------------------------------ inf ------------------------------}\\
Pow. (i,o,s,10,r)	&  \multicolumn{9}{c|}{------------------------------ inf ------------------------------}\\
Pow. (i,o,10,r)	& \multicolumn{9}{c|}{------------------------------ inf ------------------------------}\\
Pow. (i,o,s,100,r)	& \multicolumn{9}{c|}{------------------------------ inf ------------------------------}\\
Pow. (i,o,100,r)	&  \multicolumn{9}{c|}{------------------------------ inf ------------------------------}\\

\hline
\end{tabular}
}
\caption{\textbf{Cross-entropy loss}. Number of iterations until $\|\nabla f(x_k)\|/\|\nabla f(x_0)\| \leq \epsilon = 10^{-3}$. $q = 5$. inf = more than 10000 iterations. $m = 100, n = 50$.   
d = direct update, i = inverse update.
1 = single-secant, v = vanilla, s = symmetric, p = PSD, o = ours.
s = scaling, r = rejection with 0.01 tolerance.
Number refers to $\nu$ value in $\mu$-correction.
Broyden and Powell methods shown.}
\label{tab:app:mcloss:methodsA}
\end{table}

\begin{table}[ht!]
{\small
\centering
\begin{tabular}{|l||ccc|ccc|ccc|}
\hline
&\multicolumn{3}{c|}{high noise ($\sigma = 1$)}&\multicolumn{3}{c|}{medium noise ($\sigma = 0.1$)}&\multicolumn{3}{c|}{low noise ($\sigma = 0.01$)}\\
 & $\bar c = 10$  	&$\bar c =  30$ 	&$\bar c =  50$  & $\bar c = 10$  	&$\bar c =  30$ 	&$\bar c =  50$ 	& $\bar c = 10$  	&$\bar c =  30$ 	&$\bar c =  50$ 		\\\hline

DFP (d,1)       	&  \multicolumn{9}{c|}{------------------------------ inf ------------------------------}\\
DFP (d,v)	&  \multicolumn{9}{c|}{------------------------------ inf ------------------------------}\\
DFP (d,v,r)	&  \multicolumn{9}{c|}{------------------------------ inf ------------------------------}\\
DFP (d,s)	& \multicolumn{9}{c|}{------------------------------ inf ------------------------------}\\
DFP (d,s,r)	&  \multicolumn{9}{c|}{------------------------------ inf ------------------------------}\\
DFP (d,p)	&  \multicolumn{9}{c|}{------------------------------ inf ------------------------------}\\
DFP (d,p,r)	&  \multicolumn{9}{c|}{------------------------------ inf ------------------------------}\\
DFP (d,o)	&  \multicolumn{9}{c|}{------------------------------ inf ------------------------------}\\
DFP (d,o,10)	& \multicolumn{9}{c|}{------------------------------ inf ------------------------------}\\
DFP (d,o,100)	&  \multicolumn{9}{c|}{------------------------------ inf ------------------------------}\\
DFP (d,o,r)	& \multicolumn{9}{c|}{------------------------------ inf ------------------------------}\\
DFP (d,o,10,r)	&  \multicolumn{9}{c|}{------------------------------ inf ------------------------------}\\
DFP (d,o,100,r)	&  \multicolumn{9}{c|}{------------------------------ inf ------------------------------}\\
DFP (i,1)       	&  \multicolumn{9}{c|}{------------------------------ inf ------------------------------}\\
DFP (i,v) 	&  \multicolumn{9}{c|}{------------------------------ inf ------------------------------}\\
DFP (i,v,r)	&  \multicolumn{9}{c|}{------------------------------ inf ------------------------------}\\
DFP (i,s) 	&1206	 & Inf	 & Inf	 & Inf	 & Inf	 & Inf	 & 977	 &1187	 & Inf\\
DFP (i,s,r)	&  \multicolumn{9}{c|}{------------------------------ inf ------------------------------}\\
DFP (i,p) 	&  \multicolumn{9}{c|}{------------------------------ inf ------------------------------}\\
DFP (i,p,r)	&  \multicolumn{9}{c|}{------------------------------ inf ------------------------------}\\
DFP (i,o,s) 	&2924	 & Inf	 & Inf	 &3727	 &3469	 & Inf	 &3349	 &3312	 &4344\\
DFP (i,o) 	& Inf	 & Inf	 & Inf	 & Inf	 & Inf	 & Inf	 & Inf	 & Inf	 & Inf\\
DFP (i,o,s,10) 	&\textbf{1037}	 & Inf	 & Inf	 & \textbf{852}	 & \textbf{917}	 & Inf	 & \textbf{898}	 &\textbf{1021}	 &\textbf{1091}\\
DFP (i,o,10) 	&  \multicolumn{9}{c|}{------------------------------ inf ------------------------------}\\
DFP (i,o,s,100) 	&2064	 & Inf	 &\textbf{8840}	 &1362	 & 972	 &\textbf{1315}	 &2250	 &1643	 &2106\\
DFP (i,o,100) 	& \multicolumn{9}{c|}{------------------------------ inf ------------------------------}\\
DFP (i,o,s,r)	& \multicolumn{9}{c|}{------------------------------ inf ------------------------------}\\
DFP (i,o,r)	&  \multicolumn{9}{c|}{------------------------------ inf ------------------------------}\\
DFP (i,o,s,10,r)	&  \multicolumn{9}{c|}{------------------------------ inf ------------------------------}\\
DFP (i,o,10,r)	&  \multicolumn{9}{c|}{------------------------------ inf ------------------------------}\\
DFP (i,o,s,100,r)	&  \multicolumn{9}{c|}{------------------------------ inf ------------------------------}\\
DFP (i,o,100,r)	& \multicolumn{9}{c|}{------------------------------ inf ------------------------------}\\
\hline
BFGS (d,1)       	&  \multicolumn{9}{c|}{------------------------------ inf ------------------------------}\\
BFGS (d,v)	& 817	 & Inf	 &\textbf{1177}	 & \textbf{681}	 &1035	 & Inf	 & 876	 & Inf	 & Inf\\
BFGS (d,v,r)	&  \multicolumn{9}{c|}{------------------------------ inf ------------------------------}\\
BFGS (d,s)	&1093	 & Inf	 &6714	 & Inf	 & Inf	 & Inf	 & Inf	 & Inf	 & Inf\\
BFGS (d,s,r)	& \multicolumn{9}{c|}{------------------------------ inf ------------------------------}\\
BFGS (d,p)	&  \multicolumn{9}{c|}{------------------------------ inf ------------------------------}\\
BFGS (d,p,r)	&  \multicolumn{9}{c|}{------------------------------ inf ------------------------------}\\
BFGS (d,o)	&1494	 & Inf	 &9609	 &1497	 &2012	 & Inf	 &1881	 & Inf	 &1552\\
BFGS (d,o,10)	&1153	 &7282	 &9609	 &1497	 &2683	 & Inf	 &1881	 & Inf	 &1552\\
BFGS (d,o,100)	& Inf	 & Inf	 &9609	 &1497	 &2168	 & Inf	 &1881	 & Inf	 &1552\\
BFGS (d,o,r)	& Inf	 & Inf	 & Inf	 &5033	 & Inf	 & Inf	 & Inf	 & Inf	 & Inf\\
BFGS (d,o,10,r)	&  \multicolumn{9}{c|}{------------------------------ inf ------------------------------}\\
BFGS (d,o,100,r)	&  \multicolumn{9}{c|}{------------------------------ inf ------------------------------}\\
BFGS (i,1)       	&  \multicolumn{9}{c|}{------------------------------ inf ------------------------------}\\
BFGS (i,v) 	& \textbf{666}	 &\textbf{3069}	 &1907	 & 691	 & \textbf{830}	 & Inf	 & \textbf{837}	 & Inf	 & \textbf{690}\\
BFGS (i,v,r)	& \multicolumn{9}{c|}{------------------------------ inf ------------------------------}\\
BFGS (i,s) 	&1296	 &5729	 &2523	 &1001	 &1471	 & Inf	 & Inf	 & Inf	 & Inf\\
BFGS (i,s,r)	& \multicolumn{9}{c|}{------------------------------ inf ------------------------------}\\
BFGS (i,p) 	&1415	 &5838	 &3049	 &1109	 &1220	 & Inf	 & 988	 & Inf	 & Inf\\
BFGS (i,p,r)	&  \multicolumn{9}{c|}{------------------------------ inf ------------------------------}\\
BFGS (i,o,s) 	&6664	 & Inf	 & Inf	 &4435	 &5581	 &\textbf{9759}	 &3542	 &6905	 & Inf\\
BFGS (i,o) 	&  \multicolumn{9}{c|}{------------------------------ inf ------------------------------}\\
BFGS (i,o,s,10) 	&1303	 & Inf	 &2649	 & Inf	 &1170	 & Inf	 &1045	 & Inf	 &2379\\
BFGS (i,o,10) 	&  \multicolumn{9}{c|}{------------------------------ inf ------------------------------}\\
BFGS (i,o,s,100) 	& Inf	 & Inf	 &2565	 &2244	 & Inf	 & Inf	 & Inf	 & Inf	 & Inf\\
BFGS (i,o,100) 	&  \multicolumn{9}{c|}{------------------------------ inf ------------------------------}\\
BFGS (i,o,s,r)	&3251	 & Inf	 & Inf	 &2768	 &4816	 & Inf	 &2827	 &\textbf{5122}	 &7782\\
BFGS (i,o,r)	&  \multicolumn{9}{c|}{------------------------------ inf ------------------------------}\\
BFGS (i,o,s,10,r)	&  \multicolumn{9}{c|}{------------------------------ inf ------------------------------}\\
BFGS (i,o,10,r) & \multicolumn{9}{c|}{------------------------------ inf ------------------------------}\\
BFGS (i,o,s,100,r)	&5830	 & Inf	 & Inf	 &4750	 & Inf	 & Inf	 & Inf	 & Inf	 & Inf\\
BFGS (i,o,100,r)	&  \multicolumn{9}{c|}{------------------------------ inf ------------------------------}\\
\hline
\end{tabular}
}
\caption{\textbf{Cross-entropy loss}. Number of iterations until $\|\nabla f(x_k)\|/\|\nabla f(x_0)\| \leq \epsilon = 10^{-3}$. $q = 5$. inf = more than 10000 iterations. $m = 100, n = 50$.   
d = direct update, i = inverse update.
1 = single-secant, v = vanilla, s = symmetric, p = PSD, o = ours.
s = scaling, r = rejection with 0.01 tolerance.
Number refers to $\nu$ value in $\mu$-correction.
DFP and BFGS methods shown.}
\label{tab:app:mcloss:methodsB}
\end{table}

\setlength{\tabcolsep}{6pt}  
\renewcommand{\arraystretch}{1}

\subsection{Logistic regression limited memory BFGS}
\label{app:extranumerics:logreg_limited}

This section presents extended results for the multiclass logistic regression minimization, in table \ref{tab:app:logreg_limited_appendix_A}-\ref{tab:app:logreg_limited_appendix_C}.

\setlength{\tabcolsep}{3pt}
\renewcommand{\arraystretch}{0.8}

\begin{table}[ht!]
{\small
\centering
\begin{tabular}{|ll||cc|cc|cc||cc|cc|cc||cc|cc|cc||cc|cc|cc||}
\hline
&&\multicolumn{6}{c||}{Low signal regime}&\multicolumn{6}{c|}{High signal regime}\\
 && \multicolumn{2}{c|}{$\bar c = 10$}  	& \multicolumn{2}{c|}{$\bar c =  30$} 	& \multicolumn{2}{c||}{$\bar c =  50$} 	& \multicolumn{2}{c|}{$\bar c = 10$} 	& \multicolumn{2}{c|}{$\bar c =  30$} 	& \multicolumn{2}{c|}{$\bar c =  50$} 	\\
&&cu 	&an 	&cu 	&an 	&cu 	&an 	&cu 	& an 	&cu 	&  an 	&cu 	&an 	\\\hline   
\multicolumn{2}{|l||}{Newton's}	&  11	 &  11	 &  11	 &  11	 &  11	 &  11	 &  11	 &  11	 &  11	 &  11	 &  11	 &  11\\
\multicolumn{2}{|l||}{Grad Desc}	&2051	 &2051	 &2010	 &2010	 &2002	 &2002	 &2357	 &2357	 &2106	 &2106	 &2060	 &2060\\\hline
($L$,$q$) &(type,$\gamma$,*)&&&&&&&&&&&&\\

(1,1)&(1,0.1)      	&7991	 &7991	 &8001	 &8001	 &8003	 &8003	 &8125	 &8125	 &8049	 &8049	 &8034	 &8034\\
(1,1)&(1,1)      	&5668	 &5668	 &5663	 &5663	 &5663	 &5663	 &5777	 &5777	 &5700	 &5700	 &5687	 &5687\\
(1,1)&(1,10)      	&3332	 &3332	 &3341	 &3341	 &3342	 &3342	 &3451	 &3451	 &3389	 &3389	 &3377	 &3377\\
(1,1)&(1,100)      	&   4	 &   4	 &   4	 &   4	 &   4	 &   4	 &   4	 &   4	 &   4	 &   4	 &   4	 &   4\\
(1,5)&(v,0.1)	& Inf	 &9257	 & Inf	 & Inf	 & Inf	 & Inf	 &9311	 & Inf	 & Inf	 & Inf	 & Inf	 & Inf\\
(1,5)&(v,1)	&9369	 &8771	 & Inf	 &7776	 &8282	 &7933	 &9444	 &8285	 & Inf	 &7784	 &8845	 &9021\\
(1,5)&(v,10)	&6958	 &5242	 &6598	 &5755	 &5480	 &5734	 &5306	 &7878	 &9293	 & Inf	 &6748	 &5695\\
(1,5)&(v,100)	&   8	 &   8	 &4471	 & Inf	 &2318	 &3216	 &   8	 &   8	 &2974	 &2701	 &3430	 &3225\\
(1,5)&(v,0.1,r)	&8943	 &8933	 & Inf	 &8936	 &8947	 &8938	 & Inf	 &8981	 &8952	 &8952	 &8946	 &8946\\
(1,5)&(v,1,r)	& Inf	 &8063	 & Inf	 &8163	 & Inf	 &8063	 & Inf	 &7824	 & Inf	 &8087	 & Inf	 &8215\\
(1,5)&(v,10,r)	& Inf	 &5568	 & Inf	 &5656	 &9349	 &5775	 & Inf	 &5617	 & Inf	 &5682	 & Inf	 &5618\\
(1,5)&(v,100,r)	&   8	 &   8	 &5956	 &3012	 &3702	 &1563	 &   8	 &   8	 &3889	 &2586	 &8435	 &2985\\
(1,5)&(s,0.1)	& Inf	 & Inf	 & Inf	 & Inf	 & Inf	 & Inf	 & Inf	 & Inf	 & Inf	 & Inf	 & Inf	 & Inf\\
(1,5)&(s,1)	& \multicolumn{12}{c|}{------------------------------ inf ------------------------------}\\
(1,5)&(s,10)	& Inf	 & Inf	 & Inf	 & Inf	 & Inf	 & Inf	 & Inf	 & Inf	 & Inf	 &1341	 &1400	 & Inf\\
(1,5)&(s,100)	&   8	 &   8	 &   8	 &   8	 &   8	 &   8	 &   8	 &   8	 & Inf	 & Inf	 & Inf	 & Inf\\
(1,5)&(s,0.1,r)	& Inf	 &8890	 & Inf	 &8911	 & Inf	 &8858	 & Inf	 &8962	 &8952	 &8952	 &8946	 &8946\\
(1,5)&(s,1,r)	& Inf	 &7899	 & Inf	 &7996	 & 790	 &8067	 & Inf	 &7993	 & Inf	 &7589	 & Inf	 &8100\\
(1,5)&(s,10,r)	&1005	 &4813	 & Inf	 &5184	 & Inf	 &5661	 &1232	 &5534	 & Inf	 &5384	 & Inf	 &5562\\
(1,5)&(s,100,r)	&   8	 &   8	 &   8	 &   8	 &   8	 &   8	 &   8	 &   8	 & Inf	 &3001	 & Inf	 & Inf\\
(1,5)&(o,0.1,sc)	&\multicolumn{12}{c|}{------------------------------ inf ------------------------------}\\
(1,5)&(o,1,sc)	&\multicolumn{12}{c|}{------------------------------ inf ------------------------------}\\
(1,5)&(o,10,sc)	& \multicolumn{12}{c|}{------------------------------ inf ------------------------------}\\
(1,5)&(o,100,sc)	& Inf	 & Inf	 & Inf	 & Inf	 & Inf	 & Inf	 & Inf	 & Inf	 & Inf	 & Inf	 & Inf	 & Inf\\
(1,5)&(o,0.1,r,sc)	& Inf	 &   7	 & Inf	 &   5	 & Inf	 & Inf	 & Inf	 &4125	 & Inf	 & Inf	 & Inf	 & Inf\\
(1,5)&(o,1,r,sc)	& Inf	 & Inf	 & Inf	 &6795	 & Inf	 & Inf	 & Inf	 & Inf	 & Inf	 & Inf	 & Inf	 & Inf\\
(1,5)&(o,10,r,sc)	& Inf	 &4918	 & Inf	 & Inf	 & Inf	 & Inf	 & Inf	 & Inf	 & Inf	 &5180	 & Inf	 & Inf\\
(1,5)&(o,100,r,sc)	& Inf	 & Inf	 & Inf	 &   7	 & Inf	 & Inf	 & Inf	 & Inf	 & Inf	 &2593	 & Inf	 & Inf\\
(1,5)&(o,0.1)	& \multicolumn{12}{c|}{------------------------------ inf ------------------------------}\\
(1,5)&(o,1)	& \multicolumn{12}{c|}{------------------------------ inf ------------------------------}\\
(1,5)&(o,10)	& \multicolumn{12}{c|}{------------------------------ inf ------------------------------}\\
(1,5)&(o,100)	& \multicolumn{12}{c|}{------------------------------ inf ------------------------------}\\
(1,5)&(o,0.1,r)	& Inf	 &   7	 & Inf	 &   5	 & Inf	 & Inf	 & Inf	 &4125	 & Inf	 & Inf	 & Inf	 & Inf\\
(1,5)&(o,1,r)	&\multicolumn{12}{c|}{------------------------------ inf ------------------------------}\\
(1,5)&(o,10,r)	& Inf	 & Inf	 & Inf	 & Inf	 & Inf	 & Inf	 & Inf	 & Inf	 & Inf	 &5501	 & Inf	 & Inf\\
(1,5)&(o,100,r)	& Inf	 & Inf	 & Inf	 &   7	 & Inf	 & Inf	 & Inf	 & Inf	 & Inf	 & Inf	 & Inf	 & Inf\\

\hline
\end{tabular}
}
\caption{\textbf{Logistic regression, L-MS-BFGS.}  Number  of iterations until $\|\nabla f(x_k)\|/\|\nabla f(x_0)\| \leq \epsilon = 10^{-4}$. $q = 5$. inf = more than 10000 iterations. $\sigma = 10, m = 2000, n = 1000$.    
For type,
1 = single-secant, v = vanilla, s = symmetric, o = ours.
sc = $\mu$-scaling, r = rejection.
}
\label{tab:app:logreg_limited_appendix_A}
\end{table}

\begin{table}[ht!]
{\small
\centering
\begin{tabular}{|ll||cc|cc|cc||cc|cc|cc||cc|cc|cc||cc|cc|cc||}
\hline
&&\multicolumn{6}{c||}{Low signal regime}&\multicolumn{6}{c|}{High signal regime}\\
 && \multicolumn{2}{c|}{$\bar c = 10$}  	& \multicolumn{2}{c|}{$\bar c =  30$} 	& \multicolumn{2}{c||}{$\bar c =  50$} 	& \multicolumn{2}{c|}{$\bar c = 10$} 	& \multicolumn{2}{c|}{$\bar c =  30$} 	& \multicolumn{2}{c|}{$\bar c =  50$} 	\\
&&cu 	&an 	&cu 	&an 	&cu 	&an 	&cu 	& an 	&cu 	&  an 	&cu 	&an 	\\\hline   
\multicolumn{2}{|l||}{Newton's}	&  11	 &  11	 &  11	 &  11	 &  11	 &  11	 &  11	 &  11	 &  11	 &  11	 &  11	 &  11\\
\multicolumn{2}{|l||}{Grad Desc}	&2051	 &2051	 &2010	 &2010	 &2002	 &2002	 &2357	 &2357	 &2106	 &2106	 &2060	 &2060\\\hline
($L$,$q$) &(type,$\gamma$,*)&&&&&&&&&&&&\\

(5,1)&(1,0.1)      	&8919	 &8919	 &8923	 &8923	 &8924	 &8924	 &8964	 &8964	 &8939	 &8939	 &8934	 &8934\\
(5,1)&(1,1)      	&7514	 &7514	 &7570	 &7570	 &7574	 &7574	 &7598	 &7598	 &7581	 &7581	 &7598	 &7598\\
(5,1)&(1,10)      	&5166	 &5166	 &5212	 &5212	 &5235	 &5235	 &5241	 &5241	 &5238	 &5238	 &5243	 &5243\\
(5,1)&(1,100)      		& 508	 & 508	 &2439	 &2439	 &2938	 &2938	 &1644	 &1644	 &2819	 &2819	 &2904	 &2904\\
(5,5)&(v,0.1)	& Inf	 & Inf	 & Inf	 & Inf	 & Inf	 &9166	 & Inf	 & Inf	 & Inf	 & Inf	 & Inf	 &9351\\
(5,5)&(v,1)	&9051	 & Inf	 &9305	 &9151	 & Inf	 & Inf	 &9353	 & Inf	 &9105	 & Inf	 &9433	 &9432\\
(5,5)&(v,10)	&6358	 &6634	 &6464	 &6333	 &8821	 &7178	 &6687	 &6484	 &7356	 &6348	 &6365	 &6771\\
(5,5)&(v,100)	&3783	 &3588	 &3922	 &3913	 &4032	 &6309	 &3652	 &3671	 &4146	 &7440	 &6516	 &4168\\
(5,5)&(v,0.1,r)	&8943	 &8934	 & Inf	 &8936	 &8947	 &8938	 & Inf	 &8981	 &8952	 &8952	 &8946	 &8946\\
(5,5)&(v,1,r)	& Inf	 &8729	 & Inf	 &8856	 & Inf	 &8895	 & Inf	 &8804	 & Inf	 &8882	 & Inf	 &8912\\
(5,5)&(v,10,r)	& Inf	 &6473	 & Inf	 &6670	 & Inf	 &6798	 & Inf	 &6508	 & Inf	 &6801	 & Inf	 &6748\\
(5,5)&(v,100,r)	&4337	 &3254	 &4952	 &1500	 &6497	 &4290	 &4536	 &3084	 &5131	 &  17	 &5841	 &  18\\
(5,5)&(s,0.1)	& Inf	 & Inf	 & Inf	 & Inf	 & Inf	 & Inf	 &   6	 &   6	 &   5	 &   5	 & Inf	 & Inf\\
(5,5)&(s,1)	& Inf	 & Inf	 & Inf	 &2797	 & Inf	 & Inf	 & Inf	 & Inf	 & Inf	 & Inf	 & Inf	 & Inf\\
(5,5)&(s,10)	&2438	 & Inf	 & Inf	 &2571	 &2903	 & Inf	 & Inf	 &2544	 & Inf	 & Inf	 & Inf	 &2351\\
(5,5)&(s,100)	& Inf	 & Inf	 & Inf	 &2954	 & Inf	 &2671	 &2924	 &2909	 &2688	 &2383	 &2201	 &2570\\
(5,5)&(s,0.1,r)	& Inf	 &8933	 & Inf	 &7377	 & Inf	 &8919	 &   6	 &4368	 &8952	 &8952	 &8946	 &8946\\
(5,5)&(s,1,r)	& Inf	 &8780	 & Inf	 &8860	 &2464	 &8885	 & Inf	 &8773	 & Inf	 &8855	 &2929	 &8883\\
(5,5)&(s,10,r)	&2279	 &6507	 & Inf	 &6697	 &2308	 &6789	 &2576	 &6555	 &2443	 &6649	 & Inf	 &6754\\
(5,5)&(s,100,r)	&2925	 &3345	 &2518	 & Inf	 &2545	 & 300	 &2777	 &3409	 &2448	 &3860	 &2636	 &4057\\
(5,5)&(o,0.1,sc)	&\multicolumn{12}{c|}{------------------------------ inf ------------------------------}\\
(5,5)&(o,1,sc)	& Inf	 &2714	 & Inf	 & Inf	 & Inf	 & Inf	 & Inf	 & Inf	 & Inf	 & Inf	 & Inf	 & Inf\\
(5,5)&(o,10,sc)	& Inf	 & Inf	 & Inf	 & Inf	 & Inf	 &  21	 & Inf	 &  28	 & Inf	 & Inf	 & Inf	 & Inf\\
(5,5)&(o,100,sc)	& \multicolumn{12}{c|}{------------------------------ inf ------------------------------}\\
(5,5)&(o,0.1,r,sc)	& Inf	 &8876	 & Inf	 & Inf	 & Inf	 & Inf	 & Inf	 &8766	 & Inf	 & Inf	 & Inf	 & Inf\\
(5,5)&(o,1,r,sc)	&2766	 &8747	 & Inf	 &8852	 &2527	 &8895	 & Inf	 &8782	 & Inf	 &8863	 & Inf	 &8912\\
(5,5)&(o,10,r,sc)	& \multicolumn{12}{c|}{------------------------------ inf ------------------------------}\\
(5,5)&(o,100,r,sc)	& Inf	 &3410	 & Inf	 & Inf	 & Inf	 & Inf	 & Inf	 & Inf	 & Inf	 & Inf	 & Inf	 &4727\\
(5,5)&(o,0.1)	&\multicolumn{12}{c|}{------------------------------ inf ------------------------------}\\
(5,5)&(o,1)	& Inf	 &2714	 & Inf	 & Inf	 & Inf	 & Inf	 & Inf	 & Inf	 & Inf	 & Inf	 & Inf	 & Inf\\
(5,5)&(o,10)	& Inf	 & Inf	 & Inf	 & Inf	 & Inf	 &  21	 & Inf	 &  28	 & Inf	 & Inf	 & Inf	 & Inf\\
(5,5)&(o,100)	& \multicolumn{12}{c|}{------------------------------ inf ------------------------------}\\
(5,5)&(o,0.1,r)	& Inf	 &8876	 & Inf	 & Inf	 & Inf	 & Inf	 & Inf	 &8766	 & Inf	 & Inf	 & Inf	 & Inf\\
(5,5)&(o,1,r)	&2766	 &8747	 & Inf	 &8852	 &2527	 &8895	 & Inf	 &8782	 & Inf	 &8863	 & Inf	 &8912\\
(5,5)&(o,10,r)	&\multicolumn{12}{c|}{------------------------------ inf ------------------------------}\\
(5,5)&(o,100,r)	& \multicolumn{12}{c|}{------------------------------ inf ------------------------------}\\

\hline
\end{tabular}
}
\caption{\textbf{Logistic regression, L-MS-BFGS.}  Number  of iterations until $\|\nabla f(x_t)\|/\|\nabla f(x_0)\| \leq \epsilon = 10^{-4}$.  inf = more than 10000 iterations. $\sigma = 10, m = 2000, n = 1000$.    
For type,
1 = single-secant, v = vanilla, s = symmetric, o = ours.
sc = $\mu$-scaling, r = rejection.
}
\label{tab:app:logreg_limited_appendix_B}
\end{table}

\begin{table}[ht!]
{\small
\centering
\begin{tabular}{|ll||cc|cc|cc||cc|cc|cc||cc|cc|cc||cc|cc|cc||}
\hline
&&\multicolumn{6}{c||}{Low signal regime}&\multicolumn{6}{c|}{High signal regime}\\
 && \multicolumn{2}{c|}{$\bar c = 10$}  	& \multicolumn{2}{c|}{$\bar c =  30$} 	& \multicolumn{2}{c||}{$\bar c =  50$} 	& \multicolumn{2}{c|}{$\bar c = 10$} 	& \multicolumn{2}{c|}{$\bar c =  30$} 	& \multicolumn{2}{c|}{$\bar c =  50$} 	\\
&&cu 	&an 	&cu 	&an 	&cu 	&an 	&cu 	& an 	&cu 	&  an 	&cu 	&an 	\\\hline   
\multicolumn{2}{|l||}{Newton's}	&  11	 &  11	 &  11	 &  11	 &  11	 &  11	 &  11	 &  11	 &  11	 &  11	 &  11	 &  11\\
\multicolumn{2}{|l||}{Grad Desc}	&2051	 &2051	 &2010	 &2010	 &2002	 &2002	 &2357	 &2357	 &2106	 &2106	 &2060	 &2060\\\hline
($L$,$q$) &(type,$\gamma$,*)&&&&&&&&&&&&\\

(10,1)&(1,0.1)      	&8926	 &8926	 &8931	 &8931	 &8934	 &8934	 &8973	 &8973	 &8947	 &8947	 &8944	 &8944\\
(10,1)&(1,1)      	&7875	 &7875	 &7917	 &7917	 &7944	 &7944	 &7934	 &7934	 &7961	 &7961	 &7947	 &7947\\
(10,1)&(1,10)      &5676	 &5676	 &5754	 &5754	 &5783	 &5783	 &5775	 &5775	 &5787	 &5787	 &5823	 &5823\\
(10,1)&(1,100)      	 	&3067	 &3067	 &3461	 &3461	 &3418	 &3418	 &3162	 &3162	 &3387	 &3387	 &3466	 &3466\\
(10,5)&(v,0.1)	&9379	 & Inf	 & Inf	 & Inf	 & Inf	 &9204	 &9730	 & Inf	 & Inf	 &9105	 &9911	 &9345\\
(10,5)&(v,1)	&9389	 &8894	 &9636	 &9031	 &9308	 &9413	 & Inf	 &9060	 &9053	 & Inf	 &9558	 &9234\\
(10,5)&(v,10)	&7007	 &7135	 &7216	 &6999	 &7319	 &7941	 &7827	 &7042	 &7006	 & Inf	 &7142	 &7232\\
(10,5)&(v,100)	&4400	 &4392	 &5257	 &4597	 & Inf	 &4904	 &4372	 &5307	 &4820	 &4873	 &9186	 &4972\\
(10,5)&(v,0.1,r)	&8943	 &8934	 & Inf	 &8936	 &8947	 &8938	 & Inf	 &8981	 &8952	 &8952	 &8946	 &8946\\
(10,5)&(v,1,r)	& Inf	 &8810	 & Inf	 &8897	 & Inf	 &8926	 & Inf	 &8838	 & Inf	 &8901	 & Inf	 &8940\\
(10,5)&(v,10,r)	& Inf	 &6848	 & Inf	 &7085	 &8572	 &7211	 & Inf	 &7005	 & Inf	 &7149	 & Inf	 &7217\\
(10,5)&(v,100,r)	&6677	 &4001	 &8304	 &4611	 &6084	 &4839	 &8620	 &4366	 &6532	 &4726	 &6291	 &4778\\
(10,5)&(s,0.1)	& Inf	 & Inf	 & Inf	 & Inf	 & Inf	 & Inf	 &   6	 &   6	 &   5	 &   5	 & Inf	 & Inf\\
(10,5)&(s,1)	&2849	 & Inf	 & Inf	 & Inf	 & Inf	 &2717	 & Inf	 & Inf	 &2782	 & Inf	 & Inf	 &2826\\
(10,5)&(s,10)	&3064	 & Inf	 & Inf	 & Inf	 & Inf	 & Inf	 & Inf	 &3032	 & Inf	 & Inf	 & Inf	 & Inf\\
(10,5)&(s,100)	& Inf	 &3047	 &3233	 &3014	 & Inf	 & Inf	 &3068	 &3041	 & Inf	 &3073	 &2795	 &3271\\
(10,5)&(s,0.1,r)	& Inf	 &8933	 & Inf	 &7390	 & Inf	 &8919	 &   6	 &4138	 &8952	 &8952	 &8946	 &8946\\
(10,5)&(s,1,r)	&2466	 &8787	 & Inf	 &8882	 & Inf	 &8922	 & Inf	 &8818	 & Inf	 &8895	 & Inf	 &8936\\
(10,5)&(s,10,r)	& Inf	 &6840	 & Inf	 &7092	 & Inf	 &7159	 & Inf	 &6948	 & Inf	 &7163	 &2974	 &7158\\
(10,5)&(s,100,r)	& Inf	 &4071	 &2930	 &4750	 & Inf	 &4806	 &2941	 &4456	 & Inf	 &4665	 &3130	 &4768\\
(10,5)&(o,0.1,sc)	&\multicolumn{12}{c|}{------------------------------ inf ------------------------------}\\
(10,5)&(o,1,sc)	&2705	 & Inf	 & Inf	 & Inf	 & Inf	 & Inf	 & Inf	 & Inf	 &3008	 & Inf	 & Inf	 &2749\\
(10,5)&(o,10,sc)	& Inf	 &  39	 & Inf	 & Inf	 & Inf	 & Inf	 & Inf	 & Inf	 & Inf	 & Inf	 &  39	 & Inf\\
(10,5)&(o,100,sc)	& \multicolumn{12}{c|}{------------------------------ inf ------------------------------}\\
(10,5)&(o,0.1,r,sc)	& Inf	 &8456	 & Inf	 & Inf	 & Inf	 & Inf	 & Inf	 &8786	 & 205	 & Inf	 & 203	 & Inf\\
(10,5)&(o,1,r,sc)	& Inf	 &8775	 & Inf	 &8886	 &2747	 &8930	 &2542	 &8852	 & Inf	 &8919	 & Inf	 &8944\\
(10,5)&(o,10,r,sc)	& Inf	 & Inf	 & Inf	 & Inf	 & Inf	 & Inf	 & Inf	 & Inf	 & Inf	 & Inf	 &  39	 &7278\\
(10,5)&(o,100,r,sc)	& Inf	 &3876	 & Inf	 & Inf	 & Inf	 & Inf	 & Inf	 & Inf	 & Inf	 & Inf	 & Inf	 & Inf\\
(10,5)&(o,0.1)	& \multicolumn{12}{c|}{------------------------------ inf ------------------------------}\\
(10,5)&(o,1)	&2705	 & Inf	 & Inf	 & Inf	 & Inf	 & Inf	 & Inf	 & Inf	 &3008	 & Inf	 & Inf	 &2749\\
(10,5)&(o,10)	& Inf	 & Inf	 & Inf	 & Inf	 & Inf	 & Inf	 & Inf	 & Inf	 &  28	 & Inf	 & Inf	 & Inf\\
(10,5)&(o,100)	&\multicolumn{12}{c|}{------------------------------ inf ------------------------------}\\
(10,5)&(o,0.1,r)	& Inf	 &8456	 & Inf	 & Inf	 & Inf	 & Inf	 & Inf	 &8786	 & 205	 & Inf	 & 203	 & Inf\\
(10,5)&(o,1,r)	& Inf	 &8775	 & Inf	 &8886	 &2747	 &8930	 &2542	 &8852	 & Inf	 &8919	 & Inf	 &8944\\
(10,5)&(o,10,r)	& Inf	 & Inf	 & Inf	 & Inf	 & Inf	 & Inf	 & Inf	 & Inf	 &  28	 &5349	 & Inf	 & Inf\\
(10,5)&(o,100,r)	& \multicolumn{12}{c|}{------------------------------ inf ------------------------------}\\

\hline
\end{tabular}
}
\caption{\textbf{Logistic regression, L-MS-BFGS.}  Number  of iterations until $\|\nabla f(x_t)\|/\|\nabla f(x_0)\| \leq \epsilon = 10^{-4}$. $q = 5$. inf = more than 10000 iterations. $\sigma = 10, m = 2000, n = 1000$.    
For type,
1 = single-secant, v = vanilla, s = symmetric, o = ours.
sc = $\mu$-scaling, r = rejection.
}
\label{tab:app:logreg_limited_appendix_C}
\end{table}

\setlength{\tabcolsep}{6pt}  
\renewcommand{\arraystretch}{1}

\subsection{Data availability statement}
All experiments are done using simulations, which include code which will be made available on GitHub, upon acceptance. 
\end{document}